\documentclass{article}
\usepackage[utf8]{inputenc}

\usepackage{import}
\usepackage{mathrsfs}
\usepackage{enumerate}
\usepackage{amsthm}
\usepackage{amsmath}
\usepackage{amsfonts}
\usepackage{amssymb}
\usepackage{dsfont}
\usepackage{stmaryrd}
\usepackage{esint}
\usepackage{mathtools}
\usepackage{hyperref}
\usepackage{cleveref}
\usepackage{float}
\usepackage{accents}
\usepackage{verbatim}
\usepackage[toc]{appendix}
\usepackage{parskip}
\usepackage{relsize}
\usepackage{xspace}
\usepackage{array}
\hypersetup{
    colorlinks,
    linkcolor={blue!80!black},
    citecolor={blue!80!black},
    urlcolor={blue!80!black}
}
\usepackage{tikz}
\usetikzlibrary{arrows.meta,positioning,matrix}
\usepackage{pgfplots}
\pgfplotsset{compat=1.18}
\usepackage{xcolor}
\usepackage{mleftright}
\newcommand{\bm}[1]{\mathbf{#1}}

\let\bm\bmdouble

\usepackage{titlesec}
\titleformat*{\section}{\LARGE\bfseries}

\usepackage[top=1.0in,bottom=1.3in,left=1.15in,right=1.15in,twoside]{geometry}

\usepackage[backend=biber,url=false,isbn=false,doi=false,eprint=false,maxbibnames=99]{biblatex}
\counterwithin{figure}{section}
\numberwithin{equation}{section}

\theoremstyle{plain}
\newtheorem{theorem}{Theorem}[section]
\newtheorem*{theorem*}{Theorem}
\newtheorem{lemma}[theorem]{Lemma}
\newtheorem{corollary}[theorem]{Corollary}
\newtheorem{proposition}[theorem]{Proposition}
\newtheorem{remark}[theorem]{Remark}

\theoremstyle{definition}
\newtheorem{definition}[theorem]{Definition}
\newtheorem{example}[theorem]{Example}

\newcommand{\ZBC}{\ensuremath{\mathrm{\textsf{ZBC}}}\xspace}
\newcommand{\NZBC}{\ensuremath{\mathrm{\textsf{NZBC}}}\xspace}

\newcommand{\KdV}{\ensuremath{\mathrm{\textsf{KdV}}}\xspace}
\newcommand{\mKdV}{\ensuremath{\mathrm{\textsf{mKdV}}}\xspace}

\newcommand{\Eqn}{\ensuremath{\mathrm{\dagger}\xspace}}
\newcommand{\Eqns}{\ensuremath{\mathrm{\textsf{Eqns}}}\xspace}
\newcommand{\AKNS}{\ensuremath{\mathrm{\textsf{AKNS}}}\xspace}
\newcommand{\GP}{\ensuremath{\mathrm{\textsf{GP}}}\xspace}
\newcommand{\tGP}{\ensuremath{\mathrm{\ti{\textsf{GP}}}}\xspace}
\newcommand{\NLS}{\ensuremath{\mathrm{\textsf{NLS}}}\xspace}
\newcommand{\dNLS}{\ensuremath{\mathrm{\textsf{dNLS}}}\xspace}
\newcommand{\hGP}{\ensuremath{\mathrm{\textsf{hGP}}}\xspace}
\newcommand{\htGP}{\ensuremath{\mathrm{\textsf{h}\ti{\textsf{GP}}}}\xspace}
\newcommand{\rGP}{\ensuremath{\mathrm{\epsilon\textsf{GP}}}\xspace}
\newcommand{\rtGP}{\ensuremath{\mathrm{\epsilon\ti{\textsf{GP}}}}\xspace}

\DeclareMathOperator{\R}{\mathbb{R}}

\DeclareMathOperator{\C}{\mathbb{C}}

\DeclareMathOperator{\SSS}{\mathbb{S}}
\DeclareMathOperator{\N}{\mathbb{N}}
\DeclareMathOperator{\Z}{\mathbb{Z}}

\DeclareMathOperator{\K}{\mathbb{K}}

\DeclareMathOperator{\Real}{Re}
\DeclareMathOperator{\Imag}{Im}
\DeclareMathOperator{\Even}{Ev}
\DeclareMathOperator{\Odd}{Od}
\DeclareMathOperator{\sign}{sign}

\DeclareMathOperator{\loc}{loc}

\DeclareMathOperator{\sech}{sech}
\DeclareMathOperator{\diag}{diag}

\DeclareMathOperator{\tr}{tr}

\DeclareMathOperator{\HS}{HS}

\NewDocumentCommand{\opnorm}{sO{}m}{%
  \IfBooleanTF{#1}{
    \left|\opnormkern\left|\opnormkern\left|
    #3
    \right|\opnormkern\right|\opnormkern\right|
  }{
    \mathopen{#2|\opnormkern #2|\opnormkern #2|}
    #3
    \mathclose{#2|\opnormkern #2|\opnormkern #2|}
  }%
}
\newcommand{\opnormkern}{\mkern-1.5mu\relax}

\DeclarePairedDelimiter{\ceil}{\lceil}{\rceil}
\DeclarePairedDelimiter{\floor}{\lfloor}{\rfloor}
\makeatletter
  \let\oldceil\ceil
  \def\ceil{\@ifstar{\oldceil}{\oldceil*}}
  \let\oldfloor\floor
  \def\floor{\@ifstar{\oldfloor}{\oldfloor*}}
\makeatother
\newcommand*\dd{\mathop{}\!\mathrm{d}}

\DeclareFontFamily{U}{mathx}{\hyphenchar\font45}
\DeclareFontShape{U}{mathx}{m}{n}{
      <5> <6> <7> <8> <9> <10>
      <10.95> <12> <14.4> <17.28> <20.74> <24.88>
      mathx10
      }{}
\DeclareSymbolFont{mathx}{U}{mathx}{m}{n}
\DeclareFontSubstitution{U}{mathx}{m}{n}
\DeclareMathAccent{\widecheck}{0}{mathx}{"71}

\newcommand{\conj}[1]{\overline{#1}}
\newcommand{\ti}[1]{\widetilde{#1}}
\newcommand{\ha}[1]{\widehat{#1}}

\usetikzlibrary{tikzmark,arrows.meta,bending,calc}
\newcounter{coverarcarrow}

\let\phi=\varphi

\let\epsilon=\varepsilon

\let\del=\partial
\let\mc=\mathcal

\title{Long-wave KdV Hierarchy Approximation of the NLS Hierarchy with Nonzero Boundary Conditions}

\author{Robert Wegner}
\date{}

\begin{document}

\maketitle

\begin{abstract}
  We study the approximation of certain renormalized conserved quantities for the \NLS hierarchy with nonzero boundary conditions, in the long-wave regime, by the energies of the \KdV hierarchy.
  We extend this to all $n \in \N$ by proving an approximation result for the transmission coefficient of the Lax operator of the \NLS hierarchy, which is a Dirac operator in the nonrelativistic regime, by the transmission coefficient of a Schrödinger operator, which is the Lax operator of the \KdV hierarchy. 
  This yields a formal approximation result between the hierarchies, which we quantify using energy methods and previously established well-posedness results.
\end{abstract}
 
\tableofcontents

\mleftright

\section{Introduction}
\label{section:introduction}

We fix $N \in \N$ and consider the ``$N$-truncated'' defocusing \NLS hierarchy
\footnote{Some authors define the \NLS hierarchy to consist of only the even flows and refer to the odd flows as the \mKdV hierarchy.}
\begin{align*} \tag{$\NLS_n$} \label{eqn:NLS-n}
    i \del_{t_n} q &= \frac{\delta \mc{H}^\NLS_n(q)}{\delta \conj{q}} \qquad \forall\, n \in \{0, \dots, N\}
\end{align*}
for a function $q = q(\bm{t}, x) = q(t_0, \dots, t_N, x): \R^{N+1} \times \R \longrightarrow \C$.
Here $\mc{H}^\NLS_n$ are the Hamiltonians of the NLS hierarchy, the first of which are
\begin{align}
    \tag{Mass} \mc{H}^\NLS_0(q) &= \int_{\R} q \conj{q} \dd x
    \\ \tag{Momentum} \mc{H}^\NLS_1(q) &= - i \int_{\R} q \conj{q}_x \dd x
    \\ \tag{Energy} \mc{H}^\NLS_2(q) &= \int_{\R} q_x \conj{q}_x + q^2 \conj{q}^2 \dd x
    \\ \nonumber \mc{H}^\NLS_3(q) &= i \int_{\R} q \conj{q}_{xxx} - 4 q^2 \conj{q} \conj{q}_x - q_x q \conj{q}^2 \dd x
    \\ \nonumber \mc{H}^\NLS_4(q) &= \int_{\R} q_{xx} \conj{q}_{xx} - 6 q^2 \conj{q} \conj{q}_{xx} - 5 q^2 \conj{q}_x^2 
    - 6 q q_x \conj{q} \conj{q}_x - q q_{xx} \conj{q}^2 + 2 q^3 \conj{q}^3 \dd x \,.
\end{align}
We are interested in long-wave solutions $q$ of the form
\begin{align} \label{eqn:ansatz}
    q(\bm{t}, x) &= \left( 1 + \frac{\epsilon^2}{2} (W_+ - W_-)(\bm{T}, X) \right) \exp\left( i \frac{\epsilon}{\sqrt{2}} \int_0^X (W_+ + W_-)(\bm{T}, Y) \dd Y \right) e^{i \phi(\bm{t})}
\end{align}
which satisfy the nonzero boundary conditions
\begin{align*} \tag{\NZBC} \label{eqn:NZBC}
    \lim_{x \rightarrow \pm \infty} q(\bm{t}, x) e^{- i \psi(\bm{t})} &= q_\pm \text{ where } |q_\pm| = 1 \,.
\end{align*}
We use the rescaled space and time variables
\begin{align*}
    X &= \sqrt{2} \epsilon x & \bm{T} = (T_n)_{0 \leq n \leq N} = \left(\frac{\dd T_n}{\dd t_n} t_n\right)_{0 \leq n \leq N}
\end{align*}
with
\begin{align} \label{eqn:times}
    \frac{\dd T_{2n}}{\dd t_{2n}} &= 2 (\sqrt{2} \epsilon)^{2n-1}
    & \frac{\dd T_{2n-1}}{\dd t_{2n-1}} &= (\sqrt{2} \epsilon)^{2n-1} \,,
\end{align}
where $\epsilon > 0$ is a small parameter.
The functions $\phi, \psi: \R^{N+1} \rightarrow \R$ are assumed to depend only on $\bm{t}$, representing uniform-in-$x$ phase rotations.
We assume $W_+, W_-: \R^{N+1} \times \R \longrightarrow \R$ and use the variable $W = (W_+, W_-)$ to denote the pair.
In order to match \NZBC for $q$, we choose for $W$ the zero boundary condition 
\begin{align*}
    \lim_{|X| \rightarrow \infty} W(\bm{T}, X) &= 0 \,.
\end{align*}
We are interested only in the dynamics of $W$.
Define $M = \floor{\frac{N-1}{2}}$.
Our main result is Theorem \ref{thm:1} below, which states that the behavior of the $N$-truncated \NLS hierarchy with \NZBC in the long-wave regime can be described by the $M$-truncated \KdV hierarchy
\begin{align} \label{eqn:KdV-n} \tag{$\KdV_n$}
    \del_{T_n} U &= \frac12 \del_X \frac{\delta \mc{E}^\KdV_n(U)}{\delta U} \qquad \forall\, n \in \{0, \dots, M\}
\end{align}
for a function $U = U(\bm{T}, X) = U(T_0, \dots, T_M, X): \R^{M+1} \times \R \rightarrow \R$.
Here $\mc{E}^\KdV_n$ are the Hamiltonians of the \KdV hierarchy, the first of which are
\begin{align}
    \tag{Mass} \mc{E}^\KdV_{-1}(U) &= \int_{\R} U \dd X
    \\ \tag{Momentum} \mc{E}^\KdV_0(U) &= \int_{\R} U^2 \dd X
    \\ \tag{Energy} \mc{E}^\KdV_1(U) &= \int_{\R} U_X^2 + 2 U^3 \dd X
    \\ \nonumber \mc{E}^\KdV_2(U) &= \int_{\R} U_{X\!X}^2 + 10 U_X^2 U + 5 U^4 \dd X
    \\ \nonumber \mc{E}^\KdV_3(U) &= \int_{\R} U_{X\!X\!X}^2 + 14 U_{X\!X}^2 U + 70 U_X^2 U^2 + 14 U^5 \dd X \,.
\end{align}
Theorem \ref{thm:1} is stated below after further preliminary discussion.
In Examples \ref{ex:0}--\ref{ex:7} at the end of the introduction several explicit demonstrations of our main result may be found.

\subsection{\texorpdfstring{From \NLS to \GP and \tGP}{From NLS to GP and G̃P̃}}
Let us begin by recalling the defocusing nonlinear Schrödinger equation
\begin{align} \label{eqn:NLS} \tag{\NLS}
    i q_t + q_{xx} &= 2 q |q|^2
\end{align}
for a wavefunction $q = q(t, x): \R \times \R \longrightarrow \C$ that satisfies the zero boundary condition
\begin{align} \label{eqn:ZBC} \tag{\ZBC}
    \lim_{|x| \rightarrow \infty} q(t, x) &= 0 \,.
\end{align}
It is closely related to the defocusing Gross-Pitaevskii equation
\begin{align} \label{eqn:GP} \tag{\GP}
    i q_t + q_{xx} &= 2 q (|q|^2 - 1) \,,
\end{align}
which we associate with the nonzero boundary conditions
\begin{align*} \tag{\NZBC}
    \lim_{x \rightarrow \pm \infty} q(t, x) &= q_\pm \text{ where } |q_\pm| = 1 \,.
\end{align*}
Both equations are formally equivalent, as the transformation $q \mapsto e^{2 i t} q$ maps solutions of \NLS to solutions of \GP.
It is preferable to work with the system \GP--\NZBC instead of \NLS--\NZBC because solutions to the former system are stationary at infinity, which is not the case for the latter system.
When writing \NLS and \GP, we generally mean the systems \NLS--\ZBC and \GP--\NZBC. 

The Hamiltonians $\mc{H}^\NLS_n$ are formally conserved quantities for both \NLS and \GP, but ill-defined in the setting of \NZBC.
We use instead an alternative sequence of Hamiltonians $(\mc{H}^\GP_n)_{n \geq -1}$ which is compatible with \NZBC. 
The initial ones are
\begin{align*}
    \tag{Mass} \mc{H}^\GP_0(q) &= \int_{\R} q \conj{q} - 1 \dd x
    \\ \tag{Momentum} \mc{H}^\GP_1(q) &= - i \int_{\R} q \conj{q}_x \dd x
    \\ \tag{Energy} \mc{H}^\GP_2(q) &= \int_{\R} q_x \conj{q}_x + (q \conj{q} - 1)^2 \dd x
    \\ \nonumber \mc{H}^\GP_3(q) &= i \int_{\R} q \conj{q}_{xxx} - 4 q^2 \conj{q} \conj{q}_x - q_x q \conj{q}^2 + 4 q \conj{q}_x \dd x
    \\ \nonumber \mc{H}^\GP_4(q) &= \int_{\R} q_{xx} \conj{q}_{xx} - 6 q^2 \conj{q} \conj{q}_{xx} - 5 q^2 \conj{q}_x^2 - 6 q q_x \conj{q} \conj{q}_x
    - q q_{xx} \conj{q}^2 
    \\ \nonumber &\qquad - 6 q_x \conj{q}_x + 2 q^3 \conj{q}^3 - 6 q^2 \conj{q}^2 + 6 q \conj{q} - 2 \dd x \,.
\end{align*}
In this setting there exists an additional non-trivial conserved quantity: the asymptotic phase change
\begin{align*}
    \mc{H}^\GP_{-1}(q) &= - i \log_{\text{pr}}\left(\frac{q_-}{q_+}\right) \,.
\end{align*}
Here $\log_{\text{pr}}$ denotes the principal logarithm.

We consider also another sequence of Hamiltonians $(\mc{H}^\tGP_n)_{n \geq -1}$, which is defined by $\mc{H}^\tGP_{2n} = \mc{H}^\GP_{2n}$ and the equivalent relations
\begin{align} \label{eqn:GP-tGP}
    \mc{H}^\tGP_{2n-1} &= \sum_{m=0}^n \binom{-\frac12}{n-m} 4^{n-m} 2^{-\mathds{1}_{\{m=0\}}} \mc{H}^\GP_{2m-1}
    \;\text{ and }\; \mc{H}^\GP_{2n-1} = 2^{\mathds{1}_{\{n=0\}}} \sum_{m=0}^n \binom{\frac12}{n-m} 4^{n-m} \mc{H}^\tGP_{2m-1} \,.
\end{align}
Here the initial odd Hamiltonians are
\begin{align*}
    \mc{H}^\tGP_{-1}(q) &= \frac12 \mc{H}^\GP_{-1}(q)
    \\  \mc{H}^\tGP_1(q) &= - i \int_{\R} q \conj{q}_x \dd x - \mc{H}^\GP_{-1}(q)
    \\ \nonumber \mc{H}^\tGP_3(q) &= i \int_{\R} q \conj{q}_{xxx} - 4 q^2 \conj{q} \conj{q}_x - q_x q \conj{q}^2 + 6 q \conj{q}_x \dd x + 3 \mc{H}^\GP_{-1}(q) \,.
\end{align*}
These Hamiltonians play a crucial role in the preceding work \cite{LiaoWegner2025}, which concerns the well-posedness of the \NLS hierarchy with \NZBC.
In principle, the Hamiltonians $\mc{H}^\tGP_0, \dots, \mc{H}^\tGP_8$ have previously been used\footnote{
    They use the variables
    $\del_{\mathrm{x}} \Theta_\epsilon = 6 (W_+ + W_-)$ and
    $N_\epsilon = - 6 (W_+ - W_- + \frac{\epsilon^2}{4} (W_+ - W_-)^2)$,
    and their choice of energies and momenta corresponds to ours with
    $\mc{E}_n(N_\epsilon, \Theta_\epsilon) = \frac{18}{(\sqrt{2} \epsilon)^{2n+1}} \mc{H}^\rtGP_{2n}(W)$ and 
    $\mc{P}_n(N_\epsilon, \Theta_\epsilon) = \sqrt{2} \frac{18}{(\sqrt{2} \epsilon)^{2n+1}} \mc{H}^\rtGP_{2n-1}(W)$. 
    We have checked this only for $n \in \{1, 2, 3\}$ due to computational restrictions.
    Note that (101) in the reference contains a term twice, which we understand to be unintentional.
} in \cite{BethuelGravejatSautSmets2010,BethuelGravejatSautSmets2009}, but to the best of our knowledge\footnote{We base this on \cite[Remark 4.7]{BethuelGravejatSautSmets2009}.} the authors were not aware of the ``correct'' definition for all $n \in \N$.
We elaborate on this in Section \ref{section:approx-rGP-KdV}. 
They have also appeared in \cite[End of §1.3.4]{KochLiao2022}.

Using the invertible coefficient matrices $\bm{c}_{\GP \mapsto \tGP}, \bm{c}_{\tGP \mapsto \GP} \in \R^{(N+1) \times (N+1)}$ given by
\begin{align*}
    \bm{c}_{\tGP \mapsto \GP}^{2n,2m} &= \mathds{1}_{\{n=m\}}
    & \bm{c}_{\tGP \mapsto \GP}^{2n,2m-1} &= 0
    \\ \bm{c}_{\tGP \mapsto \GP}^{2n-1,2m} &= 0
    & \bm{c}_{\tGP \mapsto \GP}^{2n-1,2m-1} &= 2^{\mathds{1}_{\{n=0\}}} \binom{\frac12}{n-m} 4^{n-m}
\end{align*}
and $\bm{c}_{\GP \mapsto \tGP} = (\bm{c}_{\tGP \mapsto \GP})^{-1}$, we can write
\begin{align*}
    \left(\frac{\delta \mc{H}^\tGP_n}{\delta \conj{q}}\right)_{0 \leq n \leq N} &= \bm{c}_{\GP \mapsto \tGP} \left(\frac{\delta \mc{H}^\GP_n}{\delta \conj{q}}\right)_{0 \leq n \leq N}
    & \left(\frac{\delta \mc{H}^\GP_n}{\delta \conj{q}}\right)_{0 \leq n \leq N} &= \bm{c}_{\tGP \mapsto \GP} \left(\frac{\delta \mc{H}^\tGP_n}{\delta \conj{q}}\right)_{0 \leq n \leq N} \,.
\end{align*}
Let us also state at this point a relation between $\mc{H}^\NLS_n$ and $\mc{H}^\GP_n$, which is a consequence of \eqref{eqn:NLS-GP-1}--\eqref{eqn:NLS-GP-2} below (see also \cite[Chapter 1,(10.25)]{FaddeevTakhtajan}):
\begin{align}
    \label{eqn:NLS-GP-3} \frac{\delta \mathcal{H}^\NLS_{2n}(q)}{\delta \conj{q}} &= \sum_{m=0}^n \binom{n-\frac12}{n-m} 4^{n-m} \frac{\delta \mathcal{H}^\GP_{2m}(q)}{\delta \conj{q}} \\
    \label{eqn:NLS-GP-4} \frac{\delta \mathcal{H}^\NLS_{2n-1}(q)}{\delta \conj{q}} &= \sum_{m=0}^n \binom{n-1}{n-m} 4^{n-m} \frac{\delta \mathcal{H}^\GP_{2m-1}(q)}{\delta \conj{q}} \,.
\end{align}
Similarly, we have
\begin{align}
    \label{eqn:NLS-tGP-3} \frac{\delta \mathcal{H}^\NLS_{2n}(q)}{\delta \conj{q}} &= \sum_{m=0}^n \binom{n-\frac12}{n-m} 4^{n-m} \frac{\delta \mathcal{H}^\tGP_{2m}(q)}{\delta \conj{q}} \\
    \label{eqn:NLS-tGP-4} \frac{\delta \mathcal{H}^\NLS_{2n-1}(q)}{\delta \conj{q}} &= \sum_{m=0}^n \binom{n-\frac12}{n-m} 4^{n-m} \frac{\delta \mathcal{H}^\tGP_{2m-1}(q)}{\delta \conj{q}} \,.
\end{align}
We do not state these relations for the Hamiltonians themselves because $\mc{H}^\NLS_n(q)$ can only be well-defined if $q$ satisfies \ZBC, while $\mc{H}^\GP_n(q)$ and $\mc{H}^\tGP_n(q)$ require \NZBC.
We define again invertible coefficient matrices $\bm{c}_{\NLS \mapsto \GP}, \bm{c}_{\GP \mapsto \NLS}, \bm{c}_{\NLS \mapsto \tGP}, \bm{c}_{\tGP \mapsto \NLS} \in \R^{(N+1) \times (N+1)}$ by
\begin{align*}
    \bm{c}_{\NLS \mapsto \GP}^{2n,2m} &= \binom{n-\frac12}{n-m} (-4)^{n-m} & \bm{c}_{\NLS \mapsto \GP}^{2n,2m-1} &= 0
    \\ \bm{c}_{\NLS \mapsto \GP}^{2n-1,2m} &= 0 & \bm{c}_{\NLS \mapsto \GP}^{2n-1,2m-1} &= \binom{n-1}{n-m} (-4)^{n-m}
    \\ \bm{c}_{\NLS \mapsto \tGP}^{2n,2m} &= \binom{n-\frac12}{n-m} (-4)^{n-m} & \bm{c}_{\NLS \mapsto \tGP}^{2n,2m-1} &= 0
    \\ \bm{c}_{\NLS \mapsto \tGP}^{2n-1,2m} &= 0 & \bm{c}_{\NLS \mapsto \tGP}^{2n-1,2m-1} &= \binom{n-\frac12}{n-m} (-4)^{n-m}
\end{align*}
and $\bm{c}_{\GP \mapsto \NLS} = (\bm{c}_{\NLS \mapsto \GP})^{-1}$, $\bm{c}_{\tGP \mapsto \NLS} = (\bm{c}_{\NLS \mapsto \tGP})^{-1}$.

In analogy to the definition of the \NLS hierarchy, we consider the $N$-truncated \GP and \tGP hierarchies, which consist, respectively, of the Hamiltonian flows
\begin{align*} \tag{$\GP_n$} \label{eqn:GP-n}
    i \del_{t_n} q &= \frac{\delta \mc{H}^\GP_n(q)}{\delta \conj{q}}
    \\ \tag{$\tGP_n$} \label{eqn:tGP-n}
    i \del_{t_n} q &= \frac{\delta \mc{H}^\tGP_n(q)}{\delta \conj{q}}
\end{align*}
for $n \in \{0, \dots, N\}$.
While the \GP hierarchy\footnote{Not to be confused with the ``Gross--Pitaevskii hierarchy'' from kinetic theory (see e.~g. \cite{ErdosSchleinYau}).} 
is a more immediate renormalization of the \NLS hierarchy for working with \NZBC, the \tGP hierarchy is the ``better'' hierarchy to work with in both \cite{LiaoWegner2025} and this work.
Note, however, that by Proposition \ref{prop:2} below the theories of the \NLS, \GP, and \tGP hierarchies are equivalent, up to the application of explicit transformations.

Let it be said that $\frac{\delta \mc{H}^\GP_{-1}}{\delta \conj{q}} = 0$ and hence the Hamiltonian flow generated by this conserved quantity is trivial, meaning the asymptotic phase change $\mc{H}^\GP_{-1}$ is not of particular importance in this work.
Nevertheless, Theorem \ref{thm:approx-2} below \textit{does} require the involvement of $\mc{H}^\tGP_{-1}$, and its proof motivates our choice of normalization for this conserved quantity, which is different from $\mc{H}^\GP_{-1}$ that appears in \cite{KochLiao2022}.

\subsection{\texorpdfstring{From \GP to the hydrodynamic formulation \hGP and the long-wave rescaling \rGP}{From GP to the hydrodynamic formulation hGP and the long-wave rescaling εGP}}
The solutions we are interested in have no vacuum, meaning that there exists a constant $\delta > 0$ for which
\begin{align*}
    \inf_{x \in \R} |q(\bm{t}, x)| > \delta \,.
\end{align*}
As a result, it is possible to apply the so-called Madelung transform, which maps $q$ to the pair $(a, v)$ with
\begin{align*}
    a &= |q|
    & v &= \Imag\left( \frac{q_x}{q} \right)
    & q &= a \exp\left( i \int_0^x v(y) \dd y \right) e^{i \phi} \,.
\end{align*}
Here $\phi$ is a uniform-in-$x$ choice of phase, which is information that has to be provided in order to invert the Madelung transform.
It is well-known that if $q$ solves \GP, then $(a, v)$ solve a hydrodynamic system, which we call the hydrodynamic Gross-Pitaevskii equations:
\begin{align*}
    \del_t a + a \del_x v + 2 v \del_x a &= 0 
    \\
    \del_t v + \del_x(v^2) + 4 a \del_x a &= \del_x\left( \frac{\del_{xx} a}{a} \right) \,.
\end{align*}
We prefer to work with the slightly different pair of variables $w = (w_+, w_-)$ where
\begin{align*}
    w_\pm &= \frac{v}{2} \pm (a - 1) \,.
\end{align*}
We still refer to the system of PDEs that $w$ solves as the hydrodynamic Gross-Pitaevskii equations:
\begin{align}
    \label{eqn:hGP-1} \tag{\hGP-1}
    \del_t w_- &= - (3 w_- + w_+ - 2) \del_x w_- + \frac12 \del_x \left( \frac{\del_{xx} (w_+ - w_-)}{2 a} \right) 
    \\\label{eqn:hGP-2} \tag{\hGP-2}
    \del_t w_+ &= - (w_- + 3 w_+ + 2) \del_x w_+ + \frac12 \del_x \left( \frac{\del_{xx} (w_+ - w_-)}{2 a} \right) \,.
\end{align}
Of course, this system has an infinite family of conserved energies given by
\begin{align*}
    \mc{H}^\hGP_n(w) &= \mc{H}^\GP_n(q)
    & \mc{H}^\htGP_n(w) &= \mc{H}^\tGP_n(q)
\end{align*}
for $n \geq - 1$.
We define the $N$-truncated \hGP and \htGP hierarchies as the hierarchies of Hamiltonian flows
\begin{align*} \tag{$\hGP_n$} \label{eqn:hGP-n}
    \del_{t_n} w &= \frac{1}{4} \begin{pmatrix}
        - \del_x a^{-1} - a^{-1} \del_x
        & \del_x a^{-1} - a^{-1} \del_x
        \\ - \del_x a^{-1} + a^{-1} \del_x
        & \del_x a^{-1} + a^{-1} \del_x
    \end{pmatrix} \begin{pmatrix} \frac{\delta}{\delta w_+} \mc{H}^\hGP_n(w) \\ \frac{\delta}{\delta w_-} \mc{H}^\hGP_n(w) \end{pmatrix}
    \\ \tag{$\htGP_n$} \label{eqn:htGP-n}
    \del_{t_n} w &= \frac{1}{4} \begin{pmatrix}
        - \del_x a^{-1} - a^{-1} \del_x
        & \del_x a^{-1} - a^{-1} \del_x
        \\ - \del_x a^{-1} + a^{-1} \del_x
        & \del_x a^{-1} + a^{-1} \del_x
    \end{pmatrix} \begin{pmatrix} \frac{\delta}{\delta w_+} \mc{H}^\htGP_n(w) \\ \frac{\delta}{\delta w_-} \mc{H}^\htGP_n(w) \end{pmatrix}
\end{align*}
for $n \in \{0, \dots, N\}$. 
Note that here $a = 1 + \frac12 (w_+ - w_-)$, $|a| > \delta$, and $a^{-1}$ should be understood as the operator of multiplication by $a^{-1}$.
In Appendix \ref{appendix:hGP} we detail the complete Hamiltonian structure of \eqref{eqn:hGP-n}, which is chosen so that $q \leftrightarrow w$ is a symplecto-/Poisson morphism.
It is indeed the case that \hyperref[eqn:hGP-n]{$(\hGP_1)$}, \hyperref[eqn:htGP-n]{$(\htGP_1)$}, and \eqref{eqn:hGP-1}--\eqref{eqn:hGP-2} are all identical.

For a small $\epsilon > 0$ we now set
\begin{align*}
    X &= \sqrt{2} \epsilon x & T_2 &= 2 \sqrt{2} \epsilon t
\end{align*}
and consider the rescaled variables
\begin{align*}
    W_\pm(T_2, X) &= \epsilon^{-2} w_\pm(t, x)
    \\ A(T_2, X) &= 1 + \frac{\epsilon^2}{2} (W_+(T_2, X) - W_-(T_2, X)) = a(t, x) \,.
\end{align*}
Transforming back to $q$ yields the ansatz \eqref{eqn:ansatz} above, up to the use of $T_2$ instead of the time vector $\bm{T}$, and the necessary reconstruction of the uniform-in-$x$ phase function $\phi(\bm{t})$.
We refer to the system that $W = (W_+, W_-)$ solves as the $\epsilon$-rescaled hydrodynamic Gross-Pitaevskii equations:
\begin{align}
    \label{eqn:rGP-1} \tag{\rGP-1}
    \del_{T_2} W_+ &= - \del_X W_+ + \frac{\epsilon^2}{4} \left( - \left( 2 W_- + 6 W_+ \right) \del_X W_+ + 2 \del_X \left( \frac{\del_{X\!X} (W_+ - W_-)}{2 A} \right) \right)
    \\ \label{eqn:rGP-2} \tag{\rGP-2} \del_{T_2} W_- &= \del_X W_- + \frac{\epsilon^2}{4} \left( - \left( 6 W_- + 2 W_+ \right) \del_X W_- + 2 \del_X \left( \frac{\del_{X\!X} (W_+ - W_-)}{2 A} \right) \right) \,.
\end{align}
They have the infinite family of conserved energies
\begin{align*}
    \mc{H}^\rGP_n(W) &= \mc{H}^\hGP_n(w)
    & \mc{H}^\rtGP_n(W) &= \mc{H}^\htGP_n(w)
\end{align*}
for $n \geq - 1$, the first of which are
\begin{align*}
    \tag{Phase Change} \mc{H}^\rGP_{-1}(W) &= - \frac{\epsilon}{\sqrt{2}} \int_{\R} W_+ + W_- \dd X
    \\ \tag{Mass} \mc{H}^\rGP_0(W) &= \frac{1}{\sqrt{2} \epsilon} \int_{\R} \left( 1 + \epsilon^2 \frac{W_+ - W_-}{2} \right)^2 - 1 \dd X
    \\ \tag{Momentum} \mc{H}^\rGP_1(W) &= - \frac{\epsilon}{\sqrt{2}} \int_{\R} \left(1 + \frac{\epsilon^2}{2} (W_+ - W_-)\right)^2 (W_+ + W_-) \dd X
    \\ \tag{Energy} \mc{H}^\rGP_2(W) &= \frac{1}{\sqrt{2} \epsilon} \int \frac{\epsilon^6}{2} (\del_X (W_+ - W_-))^2 
        + \epsilon^4 \left( 1 + \epsilon^2 \frac{W_+ - W_-}{2} \right)^2 (W_+ + W_-)^2 
        \\ &\qquad \qquad + \left( \left( 1 + \epsilon^2 \frac{W_+ - W_-}{2} \right)^2 - 1 \right)^2 \dd X \,.
\end{align*}
We define the $N$-truncated \rGP and \rtGP hierarchies as the hierarchies of Hamiltonian flows
\begin{align*} \tag{$\rGP_n$} \label{eqn:rGP-n}
    \del_{T_n} W &= \frac{\dd t_n}{\dd T_n} \frac{1}{2 \epsilon^2} \begin{pmatrix}
        - \del_X A^{-1} - A^{-1} \del_X
        & \del_X A^{-1} - A^{-1} \del_X
        \\ - \del_X A^{-1} + A^{-1} \del_X
        & \del_X A^{-1} + A^{-1} \del_X
    \end{pmatrix} \begin{pmatrix} \frac{\delta}{\delta W_+} \mc{H}^\rGP_n(W) \\ \frac{\delta}{\delta W_-} \mc{H}^\rGP_n(W) \end{pmatrix}
    \\ \tag{$\rtGP_n$} \label{eqn:rtGP-n}
    \del_{T_n} W &= \frac{\dd t_n}{\dd T_n} \frac{1}{2 \epsilon^2} \begin{pmatrix}
        - \del_X A^{-1} - A^{-1} \del_X
        & \del_X A^{-1} - A^{-1} \del_X
        \\ - \del_X A^{-1} + A^{-1} \del_X
        & \del_X A^{-1} + A^{-1} \del_X
    \end{pmatrix} \begin{pmatrix} \frac{\delta}{\delta W_+} \mc{H}^\rtGP_n(W) \\ \frac{\delta}{\delta W_-} \mc{H}^\rtGP_n(W) \end{pmatrix}
\end{align*}
for $n \in \{0, \dots, N\}$. Here
\begin{align*}
    A(\bm{T}, X) &= 1 + \frac{\epsilon^2}{2} (W_+ - W_-)(\bm{T}, X)
\end{align*}
is the rescaled amplitude.
In Appendix \ref{appendix:rGP} we detail the complete Hamiltonian structure of \eqref{eqn:rGP-n}, which is chosen so that $w \leftrightarrow W$ is a symplecto-/Poisson morphism.
Recall that the \hGP, \htGP, \rGP, and \rtGP hierarchies are equipped with the zero boundary conditions
\begin{align*}
    \lim_{|x| \rightarrow \infty} w(\bm{t}, x) &= 0 & \lim_{|X| \rightarrow \infty} W(\bm{T}, X) &= 0 \,.
\end{align*}

\subsection{\texorpdfstring{Approximation of \rGP by transport and \KdV equations from \cite{BethuelGravejatSautSmets2010,BethuelGravejatSautSmets2009}}{Approximation of ɛGP by transport and KdV equations from Bethuel–Gravejat–Saut–Smets (2009, 2010)}} \label{section:approx-rGP-KdV}

Observe that the leading order behavior of \eqref{eqn:rGP-1}--\eqref{eqn:rGP-2} is given by two transport equations with constant velocities $\pm 1$.
As such, we can think of $W_+$ and $W_-$ as two largely decoupled waves moving in opposite directions. This was observed and rigorously studied in \cite{BethuelDanchinSmets2010}.
The next-to-leading order of the system is reminiscent of the \KdV equation
\begin{align*} \label{eqn:KdV} \tag{$\KdV$}
    \del_T U &= 6 U U_X - U_{X\!X\!X} \,,
\end{align*}
although non-trivial interaction terms such as $W_+ \del_X W_-$ formally persist in the limit $\epsilon \rightarrow 0$.
In \cite{BethuelGravejatSautSmets2010,BethuelGravejatSautSmets2009} F.~Béthuel, P.~Gravejat, J.-C.~Saut, and D.~Smets studied \hGP in the long-wave regime described above 
and proved that the system \eqref{eqn:rGP-1}--\eqref{eqn:rGP-2} indeed behaves at leading order like two transport equations, and at next-to-leading order like two \KdV equations.
Specifically, they proved \cite[Theorem 1]{BethuelGravejatSautSmets2010}, which states that for any solution $q$ of \GP in a Zhidkov space of sufficient regularity, for which $W(0)$ satisfies additional estimates,
there exist two solutions $\pm U_\pm$ of $\KdV$ so that
\begin{align} \label{eqn:BGSS-approx}
    \left\|6 W_\pm\left(\frac{8}{\epsilon^2} \tau, X \mp \frac{8}{\epsilon^2} \tau\right) - U_\pm(\tau, \pm X)\right\|_{H^s} &\leq \epsilon^2 C(W(0)) e^{C(W(0)) |\tau|}
\end{align}
for all $\tau \in \R$ and $s \in \N$. Both the transport equation \hyperref[eqn:KdV-n]{$(\KdV_0)$} and of course \KdV = \hyperref[eqn:KdV-n]{$(\KdV_1)$} are part of the \KdV hierarchy.
Note that after switching to a co-moving frame to remove the leading-order transport behavior, we need to accelerate the timescale by a factor of order $\epsilon^2$ in order to see the \KdV approximation.

It turns out to be the case that for the higher equations in the \NLS, \GP and \tGP hierarchies, or more precisely the \rGP and \rtGP hierarchies, the higher equations in the \KdV hierarchy show up at leading order as approximating equations.
We can demonstrate this by formally expanding
\begin{align*}
    A^{-1} &= \sum_{k=0}^\infty \left( - \frac{\epsilon^2}{2} (W_+ - W_-) \right)^k
\end{align*}
and computing the initial orders in powers of $\epsilon$ of the right-hand sides of \eqref{eqn:rGP-n} and \eqref{eqn:rtGP-n}.
For the flows with even $n$ there is no difference between the \rGP and \rtGP hierarchies, and we obtain
\begin{align}
    \label{eqn:rtGP-expand-0} \tag{$\rtGP_0$} \del_{T_0} W &= \begin{pmatrix} 0 \\ 0 \end{pmatrix}
    \\ \label{eqn:rtGP-expand-2} \tag{$\rtGP_2$} \del_{T_2} W &= \begin{pmatrix} - (W_+)_X \\ (W_-)_X \end{pmatrix}
    + \frac{\epsilon^2}{4} \begin{pmatrix}
            (W_+)_{X\!X\!X} - (W_-)_{X\!X\!X} - 6 (W_+)_{X} W_{+} - 2 (W_+)_{X} W_{-} 
            \\ (W_+)_{X\!X\!X} - (W_-)_{X\!X\!X} - 2 (W_-)_{X} W_{+} - 6 (W_-)_{X} W_{-} 
        \end{pmatrix} 
        + O(\epsilon^4)
    \\ \label{eqn:rtGP-expand-4} \tag{$\rtGP_4$} \del_{T_4} W &= \begin{pmatrix}
        (W_+)_{X\!X\!X} - 6 (W_+)_{X} W_{+}
        \\ - (W_-)_{X\!X\!X} - 6 (W_-)_{X} W_{-}
    \end{pmatrix}
    \\\nonumber &+ \frac{\epsilon^2}{4} \begin{pmatrix} 
        - (W_+)_{X\!X\!X\!X\!X} + (W_-)_{X\!X\!X\!X\!X} + 13 (W_+)_{X\!X\!X} W_{+} + 3 (W_+)_{X\!X\!X} W_{-} + 29 (W_+)_{X\!X} (W_+)_{X} 
        \\ - (W_+)_{X\!X\!X\!X\!X} + (W_-)_{X\!X\!X\!X\!X} + 13 (W_-)_{X\!X\!X} W_{-} + 3 (W_-)_{X\!X\!X} W_{+} + 29 (W_-)_{X\!X} (W_-)_{X} 
    \end{pmatrix} 
    \\\nonumber &+ \frac{\epsilon^2}{4} \begin{pmatrix} 
        9 (W_-)_{X\!X} (W_-)_{X} + 3 (W_+)_{X\!X} (W_-)_{X} - (W_+)_{X} (W_-)_{X\!X} + 3 (W_-)_{X\!X\!X} W_{-} - 3 (W_-)_{X\!X\!X} W_{+}  
        \\ 9 (W_+)_{X\!X} (W_+)_{X} + 3 (W_+)_{X} (W_-)_{X\!X} - (W_+)_{X\!X} (W_-)_{X} + 3 (W_+)_{X\!X\!X} W_{+} - 3 (W_+)_{X\!X\!X} W_{-}  
    \end{pmatrix}
    \\\nonumber &+ \frac{\epsilon^2}{4} \begin{pmatrix} 
        - 45 (W_+)_{X} W_{+}^{2} - 6 (W_+)_{X} W_{+} W_{-} + 3 (W_+)_{X} W_{-}^{2} 
        \\ 45 (W_-)_{X} W_{-}^{2} + 6 (W_-)_{X} W_{+} W_{-} - 3 (W_-)_{X} W_{+}^{2} 
    \end{pmatrix} + O(\epsilon^4) \,.
\end{align}
Observe that the leading orders correspond precisely to the initial equations \eqref{eqn:KdV-n} in the \KdV hierarchy, and this happens on a purely algebraic level.
For the flows with odd $n$ we first look at the \rGP hierarchy:
\begin{align}
    \label{eqn:rGP-expand-1} \tag{$\rGP_1$} \del_{T_1} W &= \begin{pmatrix}
        (W_+)_{X}
        \\ (W_-)_{X}
    \end{pmatrix}
    \\ \label{eqn:rGP-expand-3} \tag{$\rGP_3$} \del_{T_3} W &= \frac{1}{\epsilon^2} \begin{pmatrix}
        (W_+)_{X}
        \\ (W_-)_{X}
    \end{pmatrix}
    + \begin{pmatrix} 
        - (W_+)_{X\!X\!X} + 6 (W_+)_{X} W_{+}
        \\ - (W_-)_{X\!X\!X} - 6 (W_-)_{X} W_{-}
    \end{pmatrix} + O(\epsilon^2)
    \\ \label{eqn:rGP-expand-5} \tag{$\rGP_5$} \del_{T_5} W &= \frac{1}{2 \epsilon^4} \begin{pmatrix}
        - (W_+)_{X}
        \\ - (W_-)_{X}
    \end{pmatrix}
    + \frac{1}{\epsilon^2} \begin{pmatrix} 
        - (W_+)_{X\!X\!X} + 6 (W_+)_{X} W_{+}
        \\ - (W_-)_{X\!X\!X} - 6 (W_-)_{X} W_{-}
    \end{pmatrix} 
    \\\nonumber &+ \begin{pmatrix}
        (W_+)_{X\!X\!X\!X\!X} \!-\! 10 (W_+)_{X\!X\!X} W_{+} \!-\! 20 (W_+)_{X\!X} (W_+)_{X} \!+\! 30 (W_+)_{X} W_{+}^{2}
        \\ (W_-)_{X\!X\!X\!X\!X} \!+\! 10 (W_-)_{X\!X\!X} W_{-} \!+\! 20 (W_-)_{X\!X} (W_-)_{X} \!+\! 30 (W_-)_{X} W_{-}^{2}
    \end{pmatrix}
    + O(\epsilon^2) \,.
\end{align}
Here the dynamics at leading order are mixed up. 
The \rtGP hierarchy is a renormalization which separates these dynamics:
\begin{align}
    \label{eqn:rtGP-expand-1} \tag{$\rtGP_1$} \del_{T_1} W &= \begin{pmatrix}
        (W_+)_{X}
        \\ (W_-)_{X}
    \end{pmatrix}
    \\ \label{eqn:rtGP-expand-3} \tag{$\rtGP_3$} \del_{T_3} W &= \begin{pmatrix} 
        - (W_+)_{X\!X\!X} + 6 (W_+)_{X} W_{+}
        \\ - (W_-)_{X\!X\!X} - 6 (W_-)_{X} W_{-}
    \end{pmatrix} 
    + \frac{\epsilon^2}{4} \begin{pmatrix}
        - 3 (W_+)_{X\!X\!X} W_{+} - 3 (W_+)_{X\!X\!X} W_{-} 
        \\ 3 (W_-)_{X\!X\!X} W_{+} + 3 (W_-)_{X\!X\!X} W_{-}
    \end{pmatrix} 
    \\\nonumber &+ \frac{\epsilon^2}{4} \begin{pmatrix}
        - 9 (W_+)_{X\!X} (W_+)_{X} - 3 (W_+)_{X\!X} (W_-)_{X} + 3 (W_-)_{X\!X\!X} W_{-} + 3 (W_-)_{X\!X\!X} W_{+} + 3 (W_+)_{X} W_{-}^{2}
        \\ 9 (W_-)_{X\!X} (W_-)_{X} + 3 (W_+)_{X} (W_-)_{X\!X} - 3 (W_+)_{X\!X\!X} W_{+} - 3 (W_+)_{X\!X\!X} W_{-} + 3 (W_-)_{X} W_{+}^{2}
    \end{pmatrix}
    \\\nonumber &+ \frac{\epsilon^2}{4} \begin{pmatrix}
        3 (W_+)_{X} (W_-)_{X\!X} + 9 (W_-)_{X\!X} (W_-)_{X} + 15 (W_+)_{X} W_{+}^{2} + 6 (W_+)_{X} W_{+} W_{-} 
        \\ - 3 (W_+)_{X\!X} (W_-)_{X} - 9 (W_+)_{X\!X} (W_+)_{X} + 15 (W_-)_{X} W_{-}^{2} + 6 (W_-)_{X} W_{+} W_{-}
    \end{pmatrix}
    + O(\epsilon^4)
    \\ \label{eqn:rtGP-expand-5} \tag{$\rtGP_5$} \del_{T_5} W &= \begin{pmatrix}
        (W_+)_{X\!X\!X\!X\!X} - 10 (W_+)_{X\!X\!X} W_{+} - 20 (W_+)_{X\!X} (W_+)_{X} + 30 (W_+)_{X} W_{+}^{2}
        \\ (W_-)_{X\!X\!X\!X\!X} + 10 (W_-)_{X\!X\!X} W_{-} + 20 (W_-)_{X\!X} (W_-)_{X} + 30 (W_-)_{X} W_{-}^{2}
    \end{pmatrix}
    + O(\epsilon^2) \,.
\end{align}

The correspondence between the leading order of the equations in the \rtGP hierarchy and the equations in the \KdV hierarchy, for all $n \in \N$, is the content of Theorems \ref{thm:approx-2} and \ref{thm:2} below, and also our main result Theorem \ref{thm:1}.
Let us state clearly at this point that for the case of \GP = \hyperref[eqn:GP-n]{$(\GP_2)$} these theorems \textit{do not} generalize the next-to-leading order \KdV approximation result, i.~e. \eqref{eqn:BGSS-approx}. 
They only generalize the leading order transport approximation result. Note also that the leading order contains no interaction terms, which simplifies the analysis.

This work was inspired by \cite[Remarks 1.4, 1.5, 4.7]{BethuelGravejatSautSmets2009}, where the authors comment on their discovery of 
\textit{``[...] a strong and somewhat striking relationship between the (GP) invariants and the (KdV) invariants''}.
Specifically, they noticed that certain linear combinations of the conserved quantities $(\mc{H}^\rtGP_k(W))_{-1 \leq k \leq 2n}$ approximate the energy $\mc{E}^\KdV_{n-1}(\pm W_\pm)$ up to an error of size $O(\epsilon^2)$ for $0 \leq n \leq 4$.
They subsequently remark that \textit{``It would be of interest to investigate further the relationships between (GP) and
(KdV), in particular at the level of the spectral problems associated to the corresponding inverse
scattering methods.''}, and that \textit{``[the relationship mentioned above] might be extended to higher order (GP) and
(KdV) invariants, provided one was first able to compute some expressions for them.''}.

We carry this out, extending the relationship between $(\mc{H}^\rtGP_k(W))_{-1 \leq k \leq {2n}}$ and $\mc{E}^\KdV_{n-1}(\pm W_\pm)$ to all $n \in \N$.

\subsection{Approximation of Hamiltonians and transmission coefficients}
\begin{theorem} \label{thm:approx-2}
    All identities below hold up to formal integration by parts.
    For all $n \in \N$ there exist polynomials $\bm{P}_n$, depending on $\epsilon$, $W_+$, $W_-$, and up to $\floor{\frac{n}{2}}$ derivatives thereof, such that
    \begin{align}
        \label{eqn:approx-both} \mc{H}^\rtGP_{2n}(W) \mp 2 \mc{H}^\rtGP_{2n-1}(W) &= (\sqrt{2} \epsilon)^{2n+1} \left( \mc{E}^\KdV_{n-1}(\pm W_\pm) + \frac{\epsilon^2}{2} \int_{\R} \bm{P}_{2n} \mp \bm{P}_{2n-1} \dd X \right)
    \end{align}
    In particular
    \begin{align}
        \label{eqn:approx-even} \mc{H}^\rtGP_{2n}(W) &= \frac{(\sqrt{2} \epsilon)^{2n+1}}{2} \left( \mc{E}^\KdV_{n-1}(- W_-) + \mc{E}^\KdV_{n-1}(W_+) + \epsilon^2 \int_{\R} \bm{P}_{2n} \dd X \right)
        \\ \label{eqn:approx-todd} \mc{H}^{\rtGP}_{2n-1}(W) &= \frac{(\sqrt{2} \epsilon)^{2n+1}}{4} \left( \mc{E}^\KdV_{n-1}(- W_-) - \mc{E}^\KdV_{n-1}(W_+) + \epsilon^2 \int_{\R} \bm{P}_{2n-1} \dd X \right) \,.
    \end{align}
    Furthermore, there are no constant or linear monomials in $\bm{P}_n$.
    Lastly, there exist polynomials $\bm{Q}_{2n}$, depending on $\epsilon$, $W_+$, $W_-$, and up to $n - 1$ derivatives thereof, such that
    \begin{align} \label{eqn:extra}
        \int_{\R} \bm{P}_{2n} \dd X &= \frac14 \int_{\R} (\del_X^n (W_+ - W_-))^2 \dd X + \int_{\R} \bm{Q}_{2n} \dd X \,.
    \end{align}
\end{theorem}
\begin{remark}
    The initial values of the error polynomials $\bm{P}_n$ are as follows:
    \begin{align*}
        \bm{P}_{-1} &= 0 \qquad \qquad \bm{P}_0 = \frac14 (W_+ - W_-)^2
        \qquad \qquad \bm{P}_1 = \frac14 (- W_+^3 + W_+^2 W_- + W_+ W_-^2 - W_-^3)
        \\ \bm{P}_2 &= \frac14 (W_+ - W_-)_X^2 + \frac14 (3 W_+^3 - W_+^2 W_- + W_+ W_-^2 - 3 W_-^3)
        \\ &+ \frac{\epsilon^2}{32} (5 W_+^4 - 4 W_+^3 W_- - 2 W_+^2 W_-^2 - 4 W_+ W_-^3 + 5 W_-^4) 
        \\ \bm{P}_3 &= \frac14 (- (W_+)_X^2 (5 W_+ + W_-) + 6 (W_+)_X (W_-)_X (W_+ + W_-) - (W_-)_X^2 (5 W_- + W_+))
        \\ &+ \frac14(- 5 W_+^4 + 2 W_+^3 W_- - 2 W_+ W_-^3 + 5 W_-^4)
        \\ &+ \frac{\epsilon^2}{32} (- 7 W_+^5 + 5 W_+^4 W_- + 2 W_+^3 W_-^2 + 2 W_+^2 W_-^3 + 5 W_+ W_-^4 - 7 W_-^5) \,.
    \end{align*}
\end{remark}

In order to explain the origin of this result, we need to consider the Lax operators
\begin{align*}
    L^\rGP(W) &= \begin{pmatrix}
        - \epsilon^2 W_- + 1 & i \sqrt{2} \epsilon \del_X \\ i \sqrt{2} \epsilon \del_X & - \epsilon^2 W_+ - 1
    \end{pmatrix}
    & L^\KdV(\pm W_\pm) &= - \del_X^2 \pm W_\pm
    \\ L^\rGP_0 &= \begin{pmatrix}
        1 & i \sqrt{2} \epsilon \del_X \\ i \sqrt{2} \epsilon \del_X & - 1
    \end{pmatrix}
    & L^\KdV_0 &= - \del_X^2 \,.
\end{align*}
Their relevance is explained in Section \ref{section:lax-pairs}, where we give a brief overview of the Lax pair formalisms for the equations under consideration.
In the following, we denote by $\mc{Q}_j$, $j \in \{1, 2, 3, 4\}$ the four open quadrants of the complex plane.
We take two spectral parameters $\lambda \in \conj{\mc{Q}_1} \cup \conj{\mc{Q}_3}$ and $k \in \conj{\mc{Q}_1}$ such that $\lambda = \pm \sqrt{1 + 2 \epsilon^2 k^2}$, and consider the scattering problems associated to $L^\rGP - \lambda$ and $L^\KdV - k^2$. 
This is discussed in detail in Section \ref{section:transmissioncoefficient}, where we define in \eqref{eqn:trans-3} the corresponding transmission coefficients $a^\rGP(\lambda, k; W)$ and $a^\KdV(k; \pm W_\pm)$.
They are conserved quantities for every flow in their respective hierarchies and admit for $\lambda \in \conj{\mc{Q}_1}$ the following asymptotic expansions in powers of $2 k$ at infinity:
\begin{align} \label{eqn:a-expansion}
    \log a^\rGP(\lambda, k; W) &\sim i \sum_{n=-1}^\infty \frac{\mc{H}^\rGP_n(W)}{(2 k)^{n+1}}
    & \log a^\KdV(k; \pm W_\pm) &\sim i \sum_{n=-1}^\infty \frac{\mc{E}^\KdV_n(\pm W_\pm)}{(2 k)^{2n+3}} \,.
\end{align}
The precise sense of expansion is given in Definition \ref{def:expansion} below.
The following corollary presents a different perspective on Theorem \ref{thm:approx-2}.
Here we define for functions $f = f(\lambda): \C \rightarrow \C$ the even and odd parts
\begin{align} \label{eqn:def-even-odd}
    \Even_\lambda f(\lambda) &= \frac{f(\lambda) + f(- \lambda)}{2}
    & \Odd_\lambda f(\lambda) &= \frac{f(\lambda) - f(- \lambda)}{2} \,.
\end{align}
\begin{proposition} \label{prop:approx-4}
    Let $W = (W_+, W_-) \in \mc{S}(\R; \R^2)$, $\epsilon \in (0, 1)$, and $\lambda, k \in \conj{\mc{Q}_1}$ with $|k| > c(W)$ so that Lemma \ref{lem:jost} below is applicable and hence the asymptotic expansions \eqref{eqn:a-expansion} exist. 
    Then
    \begin{align}
        \label{eqn:ev-od-approx} \Even_\lambda\left[\log a^\rGP(\lambda, k; W)\right] \mp \frac{1}{\lambda} \Odd_\lambda\left[\log a^\rGP(\lambda, k; W)\right] 
        &\sim \log a^\KdV(k; (\pm W_\pm)) + \sum_{n=-1}^\infty \frac{O(\epsilon^2)}{(2 k)^{n+1}}
    \end{align}
    in the sense of asymptotic expansions as $2 i k \rightarrow \infty$. Here $O(\epsilon^2)$ must be understood as in Theorem \ref{thm:approx-2}, i.~e. in terms of polynomials in $\epsilon$, $W_+$, $W_-$, and derivatives thereof.
    
    If instead $\lambda \in (0, 1)$, $k \in i \R_+$, and $\kappa = - i k < \frac{1}{\sqrt{2} \epsilon}$ satisfies
    \begin{align*}
        \max\left\{\frac{1}{\kappa},\frac{1}{\kappa^3}\right\} \|W\|_{L^2}^2 &\leq \frac{1}{30} \,,
    \end{align*}
    then    
    \begin{align} 
        \label{eqn:a-ev-od-approx} \left| \Even_\lambda\left[\log a^\rGP(\lambda, k; W)\right] \mp \frac{1}{\lambda} \Odd_\lambda\left[\log a^\rGP(\lambda, k; W)\right] 
        - \log a^\KdV(k; \pm W_\pm) \right| &\leq 8 \epsilon^2 \kappa^{-1} \|W\|_{L^2}^2 \,.
    \end{align}
\end{proposition}
\begin{proof}
    The asymptotic expansion relation \eqref{eqn:ev-od-approx} is a reformulation of Lemma \ref{lem:approx-1}, or alternatively Theorem \ref{thm:approx-2}, involving the identities \eqref{eqn:GP-tGP}, \eqref{eqn:def-hamil}, \eqref{eqn:def-energy}, and \eqref{eqn:a-ev-od-expansion}.
    The second approximation result \eqref{eqn:a-ev-od-approx} is taken from Lemma \ref{lem:approx-3}.
\end{proof}
\begin{remark}[Splitting of the Dirac operator in the nonrelativistic limit] \label{rem:splitting}
    Setting $c = \frac{1}{\epsilon^2}$, we have
    \begin{align*}
        \frac{L^\rGP}{\epsilon^2} &= c \begin{pmatrix}
            0 & \sqrt{2} i \del_X \\ \sqrt{2} i \del_X & 0
        \end{pmatrix}
        + c^2 \begin{pmatrix}
            1 & 0 \\ 0 & - 1
        \end{pmatrix}
        + \begin{pmatrix}
        - W_- & 0 \\ 0 & - W_+
        \end{pmatrix} \,,
    \end{align*}
    which is a one-dimensional Dirac operator with speed of light $c$ and mass $m = 1$. 
    On the other hand, $L^\KdV(\pm W_\pm)$ is the Schrödinger operator with potential $\pm W_\pm$.
    The limit $\epsilon \rightarrow 0$ that we are interested in corresponds to the nonrelativistic limit $c \rightarrow \infty$,
    where one expects to recover from the Dirac operator a combination of two Schrödinger operators.
    The problem of taking this limit is well-studied in mathematical physics.
    We refer to \cite[Chapter 6]{Thaller1992} for an introduction in a general setting.
    As an example of such a result, in \cite[Corollary 4.1]{GesztesyGrosseThaller1984} (see also \cite[Theorem 4.1]{White1990}) it is proven that
    \begin{align*}
        \left(\frac{L^\rGP \mp 1 - \lambda}{\epsilon^2}\right)^{-1} &\xrightarrow{\epsilon \rightarrow 0} \mp (L^\KdV(\pm W_\pm) + \kappa^2)^{-1} \begin{pmatrix}
            \frac{1 \pm 1}{2} & 0
            \\ 0 & \frac{1 \mp 1}{2}
        \end{pmatrix}
    \end{align*}
    uniformly strongly for any $\lambda \in \C \setminus \R$. 
    More generally, one may study the asymptotic behavior of the Dirac operator in the nonrelativistic regime by considering various quantities as holomorphic functions of $\frac{1}{c^2}$ around $c = \infty$ and expanding them in a power series.
    For example, in \cite[Theorem 2.1]{GesztesyGrosseThaller1984} such a holomorphic representation of the Dirac operator is used to prove the aforementioned result.
    Most relevant for us is \cite[Theorem 4.2]{BullaGesztesyUnterkofler1992}, where such a holomorphic representation of the scattering matrix associated to $\frac{L^\rGP \mp 1}{\epsilon^2}$ is given, and the initial expansion coefficients are explicitly calculated.
    In particular, the scattering matrix associated to $L^\KdV(\pm W_\pm)$ is recovered in the limit. 
    Since the scattering matrices contain the transmission coefficients $a^\rGP(\lambda; W)$ and $a^\KdV(\kappa, \pm W_\pm)$, this result is in correspondence with \eqref{eqn:ev-od-approx}, although it is not identical.
    Admitting that we cannot speak with confidence about this broad field of mathematical physics, we hope that Proposition \ref{prop:approx-4} may present a novel perspective on the nonrelativistic limit of the one-dimensional Dirac operator.
\end{remark}

A consequence of Proposition \ref{prop:approx-4} is that the energies $\mc{E}^\KdV_n$ on the right-hand side of \eqref{eqn:approx-both} can be used to control the Sobolev norms of $\pm W_\pm$ \textit{individually}, up to an error of size $\epsilon^2$. 
This is the content of the following corollary, which is proven in Section \ref{section:proof-of-thm-approx-2} by combining Theorem \ref{thm:approx-2} with \cite[Theorem 9.1]{KochTataru2018}.
Here we use the functionals
\begin{align*}
    \mc{E}^{\rtGP,\pm}_n(W) &= \sum_{l=0}^n \binom{n}{l} (\sqrt{2} \epsilon)^{-2l-3} \big(\mc{H}^\rtGP_{2l+2}(W) \mp 2 \mc{H}^\rtGP_{2l+1}(W)\big) \,,
\end{align*}
which are of course conserved under every flow in the \rGP and \rtGP hierarchies.
\begin{corollary} \label{cor:approx-global}
    Let $n \in \N$ with $n \geq 1$ and $\epsilon \in (0, 1)$. 
    There exists some $\delta = \delta(n) \in (0, 1)$ and $C = C(n) > 0$ such that for all $W \in H^{n+1}$ with $\|W\|_{H^{-1}} < \delta$ we have $\mc{E}^{\rtGP,\pm}_n(W) < \infty$ and 
    \begin{align*}
        \left| \mc{E}^{\rtGP,\pm}_n(W) - \|W_\pm\|_{H^n}^2 - \frac{\epsilon^2}{8} \|\del_X (W_+ - W_-)\|_{H^n}^2 \right| 
        &\leq C \|W_\pm\|_{H^n}^2 \|W_\pm\|_{H^{-1}} 
        \\ &+ \epsilon^2 C \|W\|_{H^n}^2 (1 + \|W\|_{H^n})^{n+1} \,.
    \end{align*}
\end{corollary}
\begin{remark}
    We use Corollary \ref{cor:approx-global} to prove the crucial estimates \eqref{eqn:approx-global-explicit} and \eqref{eqn:approx-global-general} below, which control the size of $W$ globally in time.
    One may also attempt such estimates from \eqref{eqn:global-GP}, or more specifically the conserved energies $\sum_{l=0}^{h-1} \tau^{2(h-1-l)} \binom{h-1}{l} \mc{H}^\GP_{2l+2}$ constructed in \cite{KochLiao,KochLiao2022} for  $\tau \geq 2$.
    They control the size of the solution as measured by the semiclassical Sobolev norm
    \begin{align*}
        \|(|q|^2 - 1, \del_x q)\|_{H^h_\tau} &= \left( \int_{\R} (\tau^2 + \xi^2)^h \big(\big|\ha{|q|^2 - 1}\big|^2 + \big|\ha{\del_x q}\big|^2\big) \dd \xi \right)^{\frac12} \,.
    \end{align*}
    After obtaining an estimate, it is necessary to pass it through a rescaling and then the Madelung transform.
    The first step turns out to be problematic due to the inhomogeneity causing negative powers of $\epsilon$ to appear.
    This could be resolved by lifting the restriction $\tau \geq 2$.
\end{remark}

\subsection{Formal approximation of hierarchies}
In the previous section we have established an approximative relationship between the Hamiltonians of the \tGP and \KdV hierarchies.
This suggests a method of proof for the pattern we have observed in Section \ref{section:approx-rGP-KdV} to hold for the leading orders of the Hamiltonian PDEs in the \rtGP hierarchy.
Such an argument requires the symplectic structures of \GP and \KdV to have a similar approximative relationship, which turns out to be the case.
In Appendix \ref{appendix:poissonstructures} the reader may find a detailed description of the symplectic/Poisson structures of the hierarchies in question.
Let us write \eqref{eqn:tGP-n} in the Hamiltonian formulation
\begin{align*}
    i \del_{t_n} \begin{pmatrix} q \\ \conj{q} \end{pmatrix} &= \begin{pmatrix} 0 & 1 \\ - 1 & 0 \end{pmatrix} 
    \begin{pmatrix} \frac{\delta}{\delta \conj{q}} \mc{H}^\tGP_n(q) \\ \frac{\delta}{\delta q} \mc{H}^\tGP_n(q) \end{pmatrix} \,.
\end{align*}
Conjugation of the Hamiltonian operator with the map $(q, \conj{q}) \mapsto (w_+, w_-)$ and a careful change of variables for the functional derivative shows that the Hamiltonian formulation of \eqref{eqn:htGP-n} is 
\begin{align*}
    \del_{t_n} \begin{pmatrix} w_+ \\ w_- \end{pmatrix} &= \frac{1}{4} \begin{pmatrix}
        - \del_x a^{-1} - a^{-1} \del_x
        & \del_x a^{-1} - a^{-1} \del_x
        \\ - \del_x a^{-1} + a^{-1} \del_x
        & \del_x a^{-1} + a^{-1} \del_x
    \end{pmatrix} \begin{pmatrix} \frac{\delta}{\delta w_+} \mc{H}^\htGP_n(w) \\ \frac{\delta}{\delta w_-} \mc{H}^\htGP_n(w) \end{pmatrix} \,.
\end{align*}
Note that here $\del_x a^{-1}$ and $a^{-1} \del_x$ are compositions of operators.
Then the Hamiltonian formulation of \eqref{eqn:rtGP-n} is
\begin{align*}
    \del_{T_n} \begin{pmatrix} W_+ \\ W_- \end{pmatrix}
    &= \frac{\dd t_n}{\dd T_n} \frac{1}{2 \epsilon^2} \begin{pmatrix}
        - \del_X A^{-1} - A^{-1} \del_X
        & \del_X A^{-1} - A^{-1} \del_X
        \\ - \del_X A^{-1} + A^{-1} \del_X
        & \del_X A^{-1} + A^{-1} \del_X
    \end{pmatrix} \begin{pmatrix} \frac{\delta}{\delta W_+} \mc{H}^\rtGP_n(W) \\ \frac{\delta}{\delta W_-} \mc{H}^\rtGP_n(W) \end{pmatrix}
    \\ &= \frac{\dd t_n}{\dd T_n} \frac{2}{\epsilon^2} \left( \begin{pmatrix} - \frac{1}{2} \del_X & 0 \\ 0 & \frac{1}{2} \del_X \end{pmatrix} 
    + \epsilon^2 \bm{E}_n \right)
    \begin{pmatrix} \frac{\delta}{\delta W_+} \mc{H}^\rtGP_n(W) \\ \frac{\delta}{\delta W_-} \mc{H}^\rtGP_n(W) \end{pmatrix} \,,
\end{align*}
where the scaling factor $\frac{\dd t_n}{\dd T_n}$ was defined in \eqref{eqn:times} and
\begin{align*}
    \bm{E}_n &= \frac{A^{-1} - 1}{2 \epsilon^2} \begin{pmatrix} - \del_X & 0 \\ 0 & \del_X \end{pmatrix} 
        + \frac{A_X A^{-2}}{4 \epsilon^2} \begin{pmatrix} 1 & -1 \\ 1 & -1 \end{pmatrix}
        \\ &= - \frac14 \frac{W_+ - W_-}{1 + \frac{\epsilon^2}{2} (W_+ - W_-)} \begin{pmatrix} - \del_X & 0 \\ 0 & \del_X \end{pmatrix} 
        + \frac18 \frac{(W_+ - W_-)_X}{\left(1 + \frac{\epsilon^2}{2} (W_+ - W_-)\right)^2} \begin{pmatrix} 1 & -1 \\ 1 & -1 \end{pmatrix} \,.
\end{align*}
Applying Theorem \ref{thm:approx-2} yields
\begin{align*}
    \del_{T_{2n}} \begin{pmatrix} W_+ \\ W_- \end{pmatrix} &= \frac{\dd t_{2n}}{\dd T_{2n}} 
     \frac{2}{\epsilon^2} \left( \begin{pmatrix} - \frac{1}{2} \del_X & 0 \\ 0 & \frac{1}{2} \del_X \end{pmatrix} 
    + \epsilon^2 \bm{E}_{2n} \right) 
    \frac{(\sqrt{2} \epsilon)^{2n+1}}{2} 
    \\ &\qquad \qquad \times \left( \begin{pmatrix}
        \frac{\delta}{\delta W_+} \mc{E}^\KdV_{n-1}(W_+)
        \\ \frac{\delta}{\delta W_-} \mc{E}^\KdV_{n-1}(- W_-)
    \end{pmatrix} + \epsilon^2 \begin{pmatrix}
        \frac{\delta}{\delta W_+} \int_{\R} \bm{P}_{2n} \dd X
        \\ \frac{\delta}{\delta W_-} \int_{\R} \bm{P}_{2n} \dd X
    \end{pmatrix} \right)
    \\ &= \begin{pmatrix} - \frac{1}{2} \del_X & 0 \\ 0 & \frac{1}{2} \del_X \end{pmatrix} 
    \begin{pmatrix}
        \frac{\delta}{\delta W_+} \mc{E}^\KdV_{n-1}(W_+)
        \\ \frac{\delta}{\delta W_-} \mc{E}^\KdV_{n-1}(- W_-)
    \end{pmatrix}
    + \epsilon^2 \bm{R}_{2n}
    \\ \del_{T_{2n-1}} \begin{pmatrix} W_+ \\ W_- \end{pmatrix} &= \frac{\dd t_{2n-1}}{\dd T_{2n-1}} 
    \frac{2}{\epsilon^2} \left( \begin{pmatrix} - \frac{1}{2} \del_X & 0 \\ 0 & \frac{1}{2} \del_X \end{pmatrix} 
    + \epsilon^2 \bm{E}_{2n-1} \right) 
    \frac{(\sqrt{2} \epsilon)^{2n+1}}{4} 
    \\ &\qquad \qquad \times \left( \begin{pmatrix}
        - \frac{\delta}{\delta W_+} \mc{E}^\KdV_{n-1}(W_+)
        \\ \frac{\delta}{\delta W_-} \mc{E}^\KdV_{n-1}(- W_-)
    \end{pmatrix} + \epsilon^2 \begin{pmatrix}
        \frac{\delta}{\delta W_+} \int_{\R} \bm{P}_{2n-1} \dd X
        \\ \frac{\delta}{\delta W_-} \int_{\R} \bm{P}_{2n-1} \dd X
    \end{pmatrix} \right)
    \\ &= \begin{pmatrix} - \frac{1}{2} \del_X & 0 \\ 0 & \frac{1}{2} \del_X \end{pmatrix} 
    \begin{pmatrix}
        - \frac{\delta}{\delta W_+} \mc{E}^\KdV_{n-1}(W_+)
        \\ \frac{\delta}{\delta W_-} \mc{E}^\KdV_{n-1}(- W_-)
    \end{pmatrix}
    + \epsilon^2 \bm{R}_{2n-1}
\end{align*}
with residual terms
\begin{align*}
    \bm{R}_{2n} &= \bm{E}_{2n} \begin{pmatrix}
        \frac{\delta}{\delta W_+} \mc{E}^\KdV_{n-1}(W_+)
        \\ \frac{\delta}{\delta W_-} \mc{E}^\KdV_{n-1}(- W_-)
    \end{pmatrix}
    + \left( \begin{pmatrix} - \frac{1}{2} \del_X & 0 \\ 0 & \frac{1}{2} \del_X \end{pmatrix} 
    + \epsilon^2 \bm{E}_{2n} \right) \begin{pmatrix}
        \frac{\delta}{\delta W_+} \int_{\R} \bm{P}_{2n} \dd X
        \\ \frac{\delta}{\delta W_-} \int_{\R} \bm{P}_{2n} \dd X
    \end{pmatrix}
    \\ \bm{R}_{2n-1} &= \bm{E}_{2n-1} \begin{pmatrix}
        - \frac{\delta}{\delta W_+} \mc{E}^\KdV_{n-1}(W_+)
        \\ \frac{\delta}{\delta W_-} \mc{E}^\KdV_{n-1}(- W_-)
    \end{pmatrix}
    + \left( \begin{pmatrix} - \frac{1}{2} \del_X & 0 \\ 0 & \frac{1}{2} \del_X \end{pmatrix} 
    + \epsilon^2 \bm{E}_{2n-1} \right) \begin{pmatrix}
        \frac{\delta}{\delta W_+} \int_{\R} \bm{P}_{2n-1} \dd X
        \\ \frac{\delta}{\delta W_-} \int_{\R} \bm{P}_{2n-1} \dd X
    \end{pmatrix} \,.
\end{align*}
\begin{remark}
    The initial values of the residual terms $\bm{R}_n$ are as follows:
    \begin{align*}
        \bm{R}_0 &= \bm{R}_1 = 0
        \\ \bm{R}_2 &= \frac14 \begin{pmatrix}
            (W_+)_{X\!X\!X} - 6 (W_+)_X W_+ - (W_-)_{X\!X\!X} - 2 (W_+)_X W_-
            \\ - (W_-)_{X\!X\!X} - 6 (W_-)_X W_- + (W_+)_{X\!X\!X} - 2 (W_-)_X W_+
        \end{pmatrix} + O(\epsilon^2)
        \\ \bm{R}_3 &= \frac14 \begin{pmatrix}
            - 3 (W_+)_{X\!X\!X} W_+ - 3 (W_+)_{X\!X\!X} W_- - 9 (W_+)_{X\!X} (W_+)_X - 3 (W_+)_{X\!X} (W_-)_X
            \\ - 3 (W_+)_{X\!X\!X} W_+ - 3 (W_+)_{X\!X\!X} W_- - 9 (W_+)_{X\!X} (W_+)_X - 3 (W_+)_{X\!X} (W_-)_X
        \end{pmatrix}
        \\ &\quad + \frac14 \begin{pmatrix}
            3 (W_+)_X (W_-)_{X\!X} + 15 (W_+)_X W_+^2 + 6 (W_+)_X W_+ W_- + 3 (W_+)_X W_-^2
            \\ 3 (W_+)_X (W_-)_{X\!X} + 15 (W_-)_X W_-^2 + 6 (W_-)_X W_+ W_- + 3 (W_-)_X W_+^2
        \end{pmatrix}
        \\ &\quad + \frac14 \begin{pmatrix}
            3 (W_-)_{X\!X\!X} W_+ + 3 (W_-)_{X\!X\!X} W_- + 9 (W_-)_{X\!X} (W_-)_X
            \\ 3 (W_-)_{X\!X\!X} W_+ + 3 (W_-)_{X\!X\!X} W_- + 9 (W_-)_{X\!X} (W_-)_X  
        \end{pmatrix}
        + O(\epsilon^2) 
        \\ \bm{R}_4 &= \frac14 \begin{pmatrix}
            - (W_+)_{X\!X\!X\!X\!X} + 13 (W_+)_{X\!X\!X} W_+ + 3 (W_+)_{X\!X\!X} W_- + 29 (W_+)_{X\!X} (W_+)_X + 3 (W_+)_{X\!X} (W_-)_X
            \\ - (W_+)_{X\!X\!X\!X\!X} + 13 (W_-)_{X\!X\!X} W_- + 3 (W_+)_{X\!X\!X} W_+ + 29 (W_-)_{X\!X} (W_-)_X + 3 (W_+)_X (W_-)_{X\!X}
        \end{pmatrix}
        \\ &\quad + \frac14 \begin{pmatrix}
            - (W_+)_X (W_-)_{X\!X} - 45 (W_+)_X W_+^2 - 6 (W_+)_X W_+ W_- + 3 (W_+)_X W_-^2 + (W_-)_{X\!X\!X\!X\!X}
            \\ - (W_+)_{X\!X} (W_-)_X + 45 (W_-)_X W_-^2 + 6 (W_-)_X W_+ W_- - 3 (W_-)_X W_+^2 + (W_-)_{X\!X\!X\!X\!X}
        \end{pmatrix}
        \\ &\quad + \frac14 \begin{pmatrix}
            - 3 (W_-)_{X\!X\!X} W_+ + 3 (W_-)_{X\!X\!X} W_- + 9 (W_-)_{X\!X} (W_-)_X
            \\ - 3 (W_+)_{X\!X\!X} W_- + 3 (W_-)_{X\!X\!X} W_+ + 9 (W_+)_{X\!X} (W_+)_X
        \end{pmatrix}
        + O(\epsilon^2) \,.
    \end{align*}
\end{remark}
It is the \KdV-like structure of $\bm{R}_2$ that is proven in \cite{BethuelGravejatSautSmets2010,BethuelGravejatSautSmets2009} to yield a \KdV approximation at next-to-leading order.

We define invertible coefficient matrices $\bm{c}_{\NLS \mapsto \rtGP}, \bm{c}_{\rtGP \mapsto \NLS} \in \R^{(N+1) \times (N+1)}$ by
\begin{align*}
    \bm{c}_{\NLS \mapsto \rtGP}^{2n,2m} &= \binom{n-\frac12}{n-m} \left(- \frac{2}{\epsilon^2}\right)^{n-m}
    & \bm{c}_{\NLS \mapsto \rtGP}^{2n,2m-1} &= 0
    \\ \bm{c}_{\NLS \mapsto \rtGP}^{2n-1,2m} &= 0
    & \bm{c}_{\NLS \mapsto \rtGP}^{2n-1,2m-1} &= \binom{n-\frac12}{n-m} \left(- \frac{2}{\epsilon^2}\right)^{n-m}
\end{align*}
and $\bm{c}_{\rtGP \mapsto \NLS} = (\bm{c}_{\NLS \mapsto \rtGP})^{-1}$.
For any of the coefficient matrices $\bm{c}_{\cdots \mapsto \cdots}$ that we use, we define for their transposes the notation
\begin{align*}
    \bm{c}^{\cdots \mapsto \cdots} &= (\bm{c}_{\cdots \mapsto \cdots})^T 
    & \bm{c}^{\cdots \mapsto \cdots}_{m,n} &= \bm{c}_{\cdots \mapsto \cdots}^{n,m} \,.
\end{align*}
Then with
\begin{align*}
    (\nabla_{\bm{t}} \bm{T}) &= \diag\left(\left(\frac{\dd T_n}{\dd t_n}\right)_{0 \leq n \leq N}\right) 
    & (\nabla_{\bm{T}} \bm{t}) &= \diag\left(\left(\frac{\dd t_n}{\dd T_n}\right)_{0 \leq n \leq N}\right)
\end{align*}
we can write
\begin{align*}
    \bm{c}^{\rtGP \mapsto \NLS} &= (\nabla_{\bm{t}} \bm{T}) \bm{c}^{\tGP \mapsto \NLS} (\nabla_{\bm{T}} \bm{t})
    & \bm{c}^{\NLS \mapsto \rtGP} &= (\nabla_{\bm{t}} \bm{T}) \bm{c}^{\NLS \mapsto \tGP} (\nabla_{\bm{T}} \bm{t})
    \\ \bm{c}_{\rtGP \mapsto \NLS} &= (\nabla_{\bm{t}} \bm{T}) \bm{c}_{\tGP \mapsto \NLS} (\nabla_{\bm{T}} \bm{t})
    & \bm{c}_{\NLS \mapsto \rtGP} &= (\nabla_{\bm{T}} \bm{t}) \bm{c}_{\NLS \mapsto \tGP} (\nabla_{\bm{t}} \bm{T}) \,.
\end{align*}
For later, we define also
\begin{align*}
    \bm{c}^{\NLS \mapsto \rGP} &= (\nabla_{\bm{t}} \bm{T}) \bm{c}^{\NLS \mapsto \GP} (\nabla_{\bm{T}} \bm{t})
    & \bm{c}^{\rtGP \mapsto \rGP} &= \bm{c}^{\rtGP \mapsto \NLS} \bm{c}^{\NLS \mapsto \rGP} \,.
\end{align*}
For the convenience of the reader, we note the following facts: $\bm{c}_{\cdots \mapsto \cdots}$ is always lower triangular and $\bm{c}^{\cdots \mapsto \cdots}$ always upper triangular.
Furthermore, the diagonal values are all $1$, these matrices are always invertible, and the row and column indices of any nonzero entry have matching parity.

\begin{theorem}[Formal approximation of the \rtGP hierarchy by the \KdV hierarchy] \label{thm:2}
    Suppose that $q$ is formally a solution of the $N$-truncated \NLS hierarchy.
    If $W(\bm{T})$ is defined by \eqref{eqn:ansatz}, then $W(\bm{c}^{\NLS \mapsto \rtGP} \bm{T})$ solves the $N$-truncated \rtGP hierarchy and
    \begin{align*}
        \frac{\dd}{\dd T_n} W(\bm{c}^{\NLS \mapsto \rtGP} \bm{T})
        &= \begin{pmatrix}
            - (-1)^n \frac12 \del_X \frac{\delta}{\delta W_+} \mc{E}^\KdV_{\floor{\frac{n-1}{2}}}(W_+(\bm{c}^{\NLS \mapsto \rtGP} \bm{T}))
            \\ - \frac12 \del_X \frac{\delta}{\delta (- W_-)} \mc{E}^\KdV_{\floor{\frac{n-1}{2}}}(- W_-(\bm{c}^{\NLS \mapsto \rtGP} \bm{T}))
        \end{pmatrix} + \epsilon^2 \bm{R}_n(\bm{T}) \,.
    \end{align*}
    Correspondingly, $W(\bm{T})$ solves
    \begin{align*}
        \del_{T_{2n}} W(\bm{T})
        &= \sum_{m=0}^n \binom{n-\frac12}{n-m} \left(\frac{2}{\epsilon^2}\right)^{n-m} 
        \left( \begin{pmatrix}
            - \frac12 \del_X \frac{\delta}{\delta W_+} \mc{E}^\KdV_{m-1}(W_+(\bm{T}))
            \\ - \frac12 \del_X \frac{\delta}{\delta (- W_-)} \mc{E}^\KdV_{m-1}(- W_-(\bm{T}))
        \end{pmatrix} + \epsilon^2 \bm{R}_{2m}(\bm{c}^{\rtGP \mapsto \NLS} \bm{T}) \right)
        \\ \del_{T_{2n-1}} W(\bm{T})
        &= \sum_{m=0}^n \binom{n-\frac12}{n-m} \left(\frac{2}{\epsilon^2}\right)^{n-m} 
        \left( \begin{pmatrix}
            \frac12 \del_X \frac{\delta}{\delta W_+} \mc{E}^\KdV_{m-1}(W_+(\bm{T}))
            \\ - \frac12 \del_X \frac{\delta}{\delta (- W_-)} \mc{E}^\KdV_{m-1}(- W_-(\bm{T}))
        \end{pmatrix} + \epsilon^2 \bm{R}_{2m-1}(\bm{c}^{\rtGP \mapsto \NLS} \bm{T}) \right) \,.
    \end{align*}
\end{theorem}
\begin{proof}
    By the chain rule we have
    \begin{align*}
        \nabla_{\bm{T}} (W(\bm{c}^{\NLS \mapsto \rtGP} \bm{T})) = \bm{c}_{\NLS \mapsto \rtGP} (\nabla_{\bm{T}} W)(\bm{c}^{\NLS \mapsto \rtGP} \bm{T}) \,.
    \end{align*}
    Together with \eqref{eqn:NLS-tGP-3}--\eqref{eqn:NLS-tGP-4} and the definition of $\bm{c}_{\NLS \mapsto \rtGP}$, this implies that $W(\bm{c}^{\NLS \mapsto \rtGP} \bm{T})$ solves the $N$-truncated \rtGP hierarchy.
    Then the claimed formulas follow from the elaborations above the theorem.
\end{proof}

\begin{remark}[Conjecture on the next-to-leading order of the \rtGP hierarchy] \label{rem:conjecture}
    We have verified the following formula to hold for $n \in \{1, 2, 3, 4\}$:
    \begin{align*}
        4 (\bm{R}_{2n} + \bm{R}_{2n-1}) &= \begin{pmatrix}
            - \frac12 \del_X \frac{\delta}{\delta W_+} \mc{E}^\KdV_n(W_+) + \del_{X\!X} \left( \frac12 \del_X \frac{\delta}{\delta (- W_-)} \mc{E}^\KdV_{n-1}(- W_-) \right) + \text{interaction terms} + O(\epsilon^2)
            \\ O(\epsilon^0)
        \end{pmatrix}
        \\  4 (\bm{R}_{2n} - \bm{R}_{2n-1}) &= \begin{pmatrix}
            O(\epsilon^0)
            \\ - \frac12 \del_X \frac{\delta}{\delta W_-} \mc{E}^\KdV_n(W_-) + \del_{X\!X} \left( \frac12 \del_X \frac{\delta}{\delta W_+} \mc{E}^\KdV_{n-1}(W_+) \right) + \text{interaction terms} + O(\epsilon^2)
        \end{pmatrix} \,.
    \end{align*}
    Here ``interaction terms'' are those which contain as factors both $W_+$ and $W_-$, or, respectively, some derivatives thereof.
    The significance of the above formulas is the following: suppose $W(\bm{T})$ is a solution to the \rtGP hierarchy. 
    Then Theorem \ref{thm:2} implies
    \begin{align*}
        & \frac12 \left( \frac{\dd}{\dd T_{2n}} + \frac{\dd}{\dd T_{2n-1}} \right) W(\bm{T}) 
        = \begin{pmatrix}
            0
            \\ - \frac12 \del_X \frac{\delta}{\delta (- W_-)} \mc{E}^\KdV_{n-1}(- W_-(\bm{T}))
        \end{pmatrix} 
        \\ &+ \frac{\epsilon^2}{8} \begin{pmatrix}
            - \frac12 \del_X \frac{\delta}{\delta W_+} \mc{E}^\KdV_n(W_+(\bm{T})) + \del_{X\!X} \left( \frac12 \del_X \frac{\delta}{\delta (- W_-)} \mc{E}^\KdV_{n-1}(- W_-(\bm{T})) \right) + \text{interaction terms} + O(\epsilon^2)
            \\ O(\epsilon^0)
        \end{pmatrix}
        \\ & \frac12 \left( \frac{\dd}{\dd T_{2n}} - \frac{\dd}{\dd T_{2n-1}} \right) W(\bm{T}) 
        = \begin{pmatrix}
            - \frac12 \del_X \frac{\delta}{\delta W_+} \mc{E}^\KdV_{n-1}(W_+(\bm{T}))
            \\ 0
        \end{pmatrix} 
        \\ &+ \frac{\epsilon^2}{8} \begin{pmatrix}
            O(\epsilon^0)
            \\ - \frac12 \del_X \frac{\delta}{\delta W_-} \mc{E}^\KdV_n(W_-(\bm{T})) + \del_{X\!X} \left( \frac12 \del_X \frac{\delta}{\delta W_+} \mc{E}^\KdV_{n-1}(W_+(\bm{T})) \right) + \text{interaction terms} + O(\epsilon^2)
        \end{pmatrix} \,.
    \end{align*}
    This suggests that an extension of Theorem \ref{thm:2} and Theorem \ref{thm:1} below to the next-to-leading order is plausible: while \hyperref[eqn:rtGP-n]{$(\rtGP_{2n})$} has its leading order behavior given by \hyperref[eqn:KdV-n]{$(\KdV_{n-1})$}, switching to a ``co-\hyperref[eqn:rtGP-n]{$(\rtGP_{2n-1})$}-moving'' frame allows us to see, on a longer/slower timescale, the dynamics of \hyperref[eqn:KdV-n]{$(\KdV_n)$} in one component of our solution $W$.
    In the case $n = 1$, i.~e. that of \hyperref[eqn:rtGP-n]{$(\rGP_2)$}, this corresponds to the utilization of right- or left-moving frames in order to see the dynamics of \hyperref[eqn:KdV-n]{$(\KdV_1)$} in one component of the solution.
    This is the content of \cite{BethuelGravejatSautSmets2010}. Interestingly, this suggests the co-moving frame used in this result should be thought of as a renormalization using \hyperref[eqn:rtGP-n]{$(\rtGP_1)$}, instead of \hyperref[eqn:KdV-n]{$(\KdV_0)$}, although they are both transport equations.
    
    In the general case, supposing the above formulas can be proven to hold, the term wrapped in two derivatives and the interaction terms, which are a priori not negligible, may still turn out to be of order $O(\epsilon^2)$ in the energy estimates by the same method that is used in \cite{BethuelGravejatSautSmets2010}: 
    one component of the solution travels rapidly to infinity due to the dynamics of \hyperref[eqn:KdV-n]{$(\KdV_{n-1})$} happening on a faster timescale.
\end{remark}

\subsection{Review of well-posedness results for the hierarchies}
The introduction of \cite{LiaoWegner2025} contains an overview of well-posedness results for \NLS, \GP, \hGP, \KdV, and the associated hierarchies.
Since this work concerns the hierarchies, we focus our review of the literature on well-posedness results specifically for hierarchies.
We start with the \KdV hierarchy, where initial existence and uniqueness results were given by J.-C.~Saut in \cite{Saut1979} and M.~Schwarz in \cite{Schwarz1984}.
Later, C.~E.~Kenig, G.~Ponce and L.~Vega produced a number of works \cite{KenigPonceVega1994-1,KenigPonceVega-BenjaminOno,KenigPonceVega1994-2,KenigPonceVega1991-1,KenigPonceVega-SmallSolutions,KenigPonceVega1993,KenigPonceVega1991-2}, developing dispersive estimates and using them to prove the local well-posedness of various generalized \KdV-type hierarchies in weighted Sobolev spaces.
Another local well-posedness result in weighted Sobolev spaces for small initial data is given in \cite{Pilod}. 
This work also contains an ill-posedness result for all of the higher equations in the \KdV hierarchy, in the sense of failure of the flow map to be $C^2$ in the Sobolev spaces $H^s(\R)$.
Nevertheless, by employing subtle energy estimates and a parabolic regularization argument, C.~E.~Kenig and D.~Pilod \cite{KenigPilod} proved local well-posedness in $H^s(\R)$ at high regularity.
We combine this with the conserved quantities that H.~Koch and D.~Tataru constructed in \cite{KochTataru2018}, which prevent the Sobolev norm of a local solution from blowing up in finite time, to formulate the following global well-posedness result.
Here and thereafter we use the notation $\bm{0} = (0, \dots, 0) \in \R^{N+1}$.
\begin{proposition}[{\cite[Theorem 1.1]{KenigPilod}} and {\cite[Theorem 9.1]{KochTataru2018}}] \label{prop:1}
    Let $N \in \N$, $s \geq 4N-\frac92$.
    For every $U_0 \in H^s$ there exists a unique $U \in C(\R^{N+1}; H^s)$ with initial data $U(\bm{0}) = U_0$ that satisfies \eqref{eqn:KdV-n} in the sense of distributions for all $n \in \{0, \dots, N\}$.
    For every $s' \in [0, s]$ we have
    \begin{align} \label{eqn:KdV-bound}
        \|U\|_{L^\infty_{\bm{T}} H^{s'}_X} &\leq C(s') (1 + \|U_0\|_{H^{-1}}^{6 s'}) \|U_0\|_{H^{s'}} \,.
    \end{align}
\end{proposition}
Since we shall work with high regularity this formulation is all that we need, but it is far from the state of the art in the literature.
In the low regularity regime A.~Rybkin \cite{Rybkin2010} and also R.~Killip, M.~Vi\c san, and X.~Zhang \cite{KillipVisan,KillipVisanZhang2018} achieved a breakthrough by using a perturbation determinant approach to construct conserved quantities that control the $H^{-1}$-norm of the solution.
This allowed them to prove the global well-posedness of \KdV in $H^{-1}$ on the line and the torus, and their ``method of commuting flows'' has since been applied to several other integrable PDEs (see e.~g. \cite{BringmannKillipVisan2021,KillipVisanHarropGriffiths2024,KillipLaurensVisan2024}).
We use the uniform estimates they give for solutions to the \KdV hierarchy in the proof of \eqref{eqn:KdV-bound}.
Recently, the method of commuting flows has been used to prove the global well-posedness of the \KdV hierarchy in $H^{-1}(\R)$ by F.~Klaus, H.~Koch, and B.~Liu \cite{KlausKochLiu2023}.

Moving on to the \NLS hierarchy, we mention the work \cite{Grünrock2010} by A.~Grünrock where the Fourier restriction norm method is used to prove global well-posedness and ill-posedness results for the flows in the \NLS hierarchy, and also the \KdV hierarchy, in Fourier-Lebesgue spaces at low regularity.
This was extended by J.~Adams in \cite{Adams2024} to the even flows in the \NLS hierarchy, and recently in \cite{Adams2025} to the \dNLS hierarchy.
Note also \cite{OzawaTianWu2026}, where orbital stability of the dark soliton solutions to the the odd flows \hyperref[eqn:NLS-n]{$(\NLS_{2n+1})$} of the \NLS hierarchy, with real-valued $q$, is proven.

Of direct relevance is the preceding work \cite{LiaoWegner2025}, where the global well-posedness of the \NLS hierarchy with \NZBC is proven, i.~e. precisely the hierarchy we study here.
While the local well-posedness is based on the aforementioned methods from C.~E.~Kenig, G.~Ponce, and L.~Vega, the globalization argument uses certain conserved energies constructed in \cite{KochLiao,KochLiao2022} by H.~Koch and X.~Liao.
This can be understood as the analogue of \cite{KochTataru2018} in the setting of \NZBC.
Since this is essential for Proposition \ref{prop:2} below, we recall now some of their definitions.

We define for $s, s' \in \N$ and $\tau > 0$ the (semi-classical/weighted) Sobolev spaces
\begin{align}
    \big( H^s, \|\cdot\|_{H^s} \big) &= \big( \big\{ u \in L^2: (1 + |\xi|^2)^{\frac{s}{2}} \ha{u} \in L^2 \big\}, \|(1 + |\xi|^2)^{\frac{s}{2}} \, \ha{\cdot} \,\|_{L^2} \big)
    \\ \label{eqn:Hs-tau}\big( H^s_\tau, \|\cdot\|_{H^s_\tau} \big) &= \big( \big\{ u \in L^2: (\tau^2 + |\xi|^2)^{\frac{s}{2}} \ha{u} \in L^2 \big\}, \|(\tau^2 + |\xi|^2)^{\frac{s}{2}} \, \ha{\cdot} \,\|_{L^2} \big)
    \\ \big( H^{s',1}, \|\cdot\|_{H^{s',1}} \big) &= \big( \big\{ u \in H^{s'}: x u \in H^{s'} \big\}, \|x \cdot\|_{H^{s'}} + \|\cdot\|_{H^{s'}} \big)
\end{align}
and the energy functionals
\begin{align*}
    E^s(q) &= \||q|^2 - 1\|_{H^{s-1}} + \|q_x\|_{H^{s-1}}
    & E^{s',1}(q) &= \||q|^2 - 1\|_{H^{s'-1,1}} + \|q_x\|_{H^{s'-1,1}} \,.
\end{align*}
We recall the complete metric space $(X^s, d^s)$ whose definition is
\begin{align} \label{eqn:def-Xs}
    X^s &= \left\{ q \in H^s_{loc}: E^{s-1}(q) < \infty \right\} \big/ \SSS^1 
    & d^s(q, p) &= \left( \int_{\R} \inf_{\lambda \in \SSS^1} \|\sech(\cdot - y) (\lambda p - q)\|_{H^s}^2 \dd y \right)^{\frac12} \,.
\end{align}
This space was used by H.~Koch and X.~Liao in \cite{KochLiao,KochLiao2022} to prove the global well-posedness of \GP for $s > 0$.
Their globalization argument relies on the aforementioned conserved energies, as they can be used to control the energy functionals $E^s(q)$.

We state now the well-posedness result for the \NLS, \GP, \tGP and \rtGP hierarchies that we use. 
It is a consequence of \cite[Theorem 1.2]{LiaoWegner2025} and the theory used in its proof.
\begin{proposition}[Well-posedness of the \NLS, \GP, and \tGP hierarchies] \label{prop:2}
    Let $N \geq 2$ and define $\omega = \frac{N-1}{2}$. Let $s, s' \in \N$ such that $s \geq 3 N - 1 + \omega$ and $\frac{s+3\omega}{2} \leq s' \leq s - N$.
    For every $q_0: \R \rightarrow \C$ of the form $q_0 = q_\ast + p_0$, where $q_\ast \in H^{s+1-\ceil{\omega}+N}_{\loc}$ with $E^{s+1-\ceil{\omega}+N,1}(q_\ast) < \infty$ and $p_0 \in H^s \cap H^{s',1}$,
    there exists a unique function $q: \R^{N+1} \times \R \rightarrow \C$ with $q(\bm{0}) = q_0$ satisfying the following:
    \begin{enumerate}[(i)]
        \item For every $n \in \{0, \dots, N\}$ the functions $q(\bm{t}, x)$, $q(\bm{c}^{\NLS \mapsto \GP} \bm{t}, x)$, and $q(\bm{c}^{\NLS \mapsto \tGP} \bm{t}, x)$ solve respectively \eqref{eqn:NLS-n}, \eqref{eqn:GP-n}, and \eqref{eqn:tGP-n} in the sense of distributions.
        \item We have $q \in C_b(\R^{N+1}; X^s)$ and
        \begin{align} \label{eqn:global-GP}
            \sup_{\bm{t} \in \R^{N+1}} E^s(q(\bm{t})) &\leq C(N, s, E^s(q_0)) \,.
        \end{align}
        \item The function $p = q - e^{- i (\bm{c}^{\GP \mapsto \NLS} \bm{t})_0} q_\ast$ satisfies $p \in C(\R^{N+1}; H^s \cap H^{s',1})$.
        Furthermore, $p(\bm{t})$, $p(\bm{c}^{\NLS \mapsto \GP} \bm{t})$, and $p(\bm{c}^{\NLS \mapsto \tGP} \bm{t})$ solve 
        the corresponding perturbative formulations of respectively \eqref{eqn:NLS-n}, \eqref{eqn:GP-n}, and \eqref{eqn:tGP-n}, for the respective fronts $e^{- i (\bm{c}^{\GP \mapsto \NLS} \bm{t})_0} q_\ast$, $e^{- i t_0} q_\ast$, and $e^{- i t_0} q_\ast$, strongly in $H^{s-n} \cap H^{s'-n,1}$.
        \item If $\mc{H}^\GP_2(q_0) < b < \frac83$ for some $b > 0$, then there exists $\delta = \delta(b) > 0$ such that
        \begin{align*}
            \inf_{\bm{t} \in \R^{N+1}, x \in \R} |q(\bm{t}, x)| > \delta \,.
        \end{align*}
        In this case we can define for some $\epsilon > 0$ the function $W$ via the ansatz \eqref{eqn:ansatz}, and find that $W(\bm{c}^{\NLS \mapsto \rtGP} \bm{T}) \in C(\R^{N+1}; H^{s-1})$ solves \eqref{eqn:rtGP-n} strongly in $H^{s-2\floor{\frac{n}{2}}-2}$.
        Furthermore, for all $h \in \{1, \dots, s - 2\}$ there exists some $\epsilon_0 = \epsilon_0(h, \|W(\bm{0})\|_{H^{h+1}}) > 0$, $c = c(h) > 0$, and $C = C(h) > 0$ such that the following holds.
        Define \begin{align*}
            \tau = \max\{1, c \|W(\bm{0})\|_{H^{h+1}}^{\frac23}\} \,.
        \end{align*}
        Then $\epsilon \in (0, \epsilon_0)$ implies
        \begin{align} \label{eqn:approx-global-explicit}
            \|W_\pm(\bm{T})\|_{H^h_\tau}^2 \!+\! \frac{\epsilon^2}{8} \|\del_X (W_+ - W_-)(\bm{T})\|_{H^h_\tau}^2
            &\!\leq\! C \left( \|W_\pm(\bm{0})\|_{H^h_\tau}^2 \!+\! \frac{\epsilon^2}{8} \|\del_X (W_+ - W_-)(\bm{0})\|_{H^h_\tau}^2 \right)
            \\\nonumber &+ \epsilon^2 \tau^2 C \left( \|W(\bm{0})\|_{H^h_\tau}^2 \!+\! \frac{\epsilon^2}{4} \|\del_X (W_+ - W_-)(\bm{0})\|_{H^h_\tau}^2 \right)
        \end{align}
        for all $\bm{T} \in \R^{N+1}$.
    \end{enumerate}
\end{proposition}

\subsection{Quantitative approximation of hierarchies}

Let us give some justification for why the study of integrable hierarchies is interesting.
Due to Noether's theorem, the higher flows in an integrable hierarchy are symmetries of each equation in the hierarchy.
Any statement about the $\NLS/\KdV/\GP/\dots$ hierarchy is therefore also a statement about the $\NLS/\KdV/\GP/\dots$ equation.
These equations are all central to the study of nonlinear PDEs and ubiquitous in the physical modelling of nonlinear waves (see e.~g. \cite{Bao,Berge,Hasegawa} for applications of \NLS and \cite{Feynman,Dalfovo1998,Gardner,Grant1973,Landau,Loffredo1993} for applications of \GP).
Besides, integrable hierarchies directly appear in mathematical physics in connection with algebraic geometry \cite{AratynSorin2001,London-1,London-2}.
For the reader interested in more algebraic approaches to these integrable hierarchies, we refer to \cite{Dickey1991,GesztesyHolden2003,GesztesyHoldenMichorTeschl2008}.
We are particularly interested in the relevance of these hierarchies in the theory of amplitude and modulation equations.

The transport equation \hyperref[eqn:KdV-n]{$(\KdV_0)$} and the Korteweg-de Vries equation \hyperref[eqn:KdV-n]{$(\KdV_1)$} play an important role as long-wave amplitude equations in the study of one-dimensional nonlinear waves.
For example, they have been used in the description of long-wave solutions to the water wave problem \cite{Craig1985,CraigGroves1994,SchneiderWayne2000}, a Boussinesq equation \cite{Schneider1998-1}, 
a capillary-gravity wave system \cite{SchneiderWayne2002}, the \NLS equation \cite{ChironRousset2010,Schneider1998-2,Chirilus-BrucknerDüllSchneider2014}, the \GP equation \cite{BethuelGravejatSautSmets2010,BethuelGravejatSautSmets2009}, and the Euler-Poisson system \cite{LiuYang2020}.
For particular classes of dispersive nonlinear equations, \KdV acts as a \textit{universal} long-wave amplitude equation \cite{Bridges2013} (see also \cite[§12.2]{SchneiderUecker2017}).
Other examples of long-wave amplitude and modulation equations are \NLS, the Ginzburg-Landau equation, and the Whitham modulation equations.
For a broad perspective on the theory of such amplitude and modulation equations, we refer to the books \cite{SchneiderUecker2017,Whitham1999,Whitham1965} and the survey \cite{ElHoefer2016}.
Many of these amplitude equations are completely integrable, an aspect which has been studied in \cite{Kalyakin1989} by L.~A.~Kalyakin.
Another perspective on this phenomenon was given in \cite{Calogero1991} by F. Calogero, who argued that for certain large classes of PDEs the process of passing to an effective limit equation preserves integrability, and hence the resulting limit equations ought to be integrable if the original class contained any integrable equations.
This idea was further developed toward the classification of integrable equations.

Returning specifically to \hyperref[eqn:KdV-n]{$(\KdV_1)$} as a long-wave amplitude equation, we note that it is common for the (trivial transport equation) \hyperref[eqn:KdV-n]{$(\KdV_0)$} to predict the leading order behavior, and \hyperref[eqn:KdV-n]{$(\KdV_1)$} the next-to-leading order behavior.
One is therefore tempted to conjecture that the higher equations \hyperref[eqn:KdV-n]{$(\KdV_n)$} in the \KdV hierarchy may be used in the description of the higher order asymptotic behavior of long-wave solutions to dispersive nonlinear equations.
Yet, such results, and in general direct physical applications of the \KdV hierarchy in the modelling of nonlinear waves, are few and far between.
In this direction, there exist results by Y.~Kodama, T.~Taniuti, and Y.~Hiraoka \cite{HiraokaKodama2009,KodamaTaniuti1979}. 
They were used by D.~Bambusi \cite{Bambusi} to show that, indeed, \hyperref[eqn:KdV-n]{$(\KdV_2)$} shows up after \hyperref[eqn:KdV-n]{$(\KdV_0)$} and \hyperref[eqn:KdV-n]{$(\KdV_1)$} in the aforementioned long-wave approximation of water waves.
Furthermore, using the reductive perturbation method of T.~Taniuti \cite{Taniuti1974}, R.~A.~Kraenkel, M.~A.~Manna, and J.~G.~Pereira \cite{KraenkelMannaPereira1995} discovered that the full \KdV hierarchy appears in the study of the asymptotic behavior of long, shallow water waves.
Other results in a similar direction are \cite{Demiray2008,Ivanov2024}. 
In \cite[Remark 4.6.10]{Sarah2024} S.~Hofbauer noted that a result of her dissertation, using the inverse scattering transform method to prove an approximation result for \KdV by a linear Schrödinger equation, can be extended to the higher equations in the \KdV hierarchy.

We state now our main result, which shows rigorously that the equations in the \KdV hierarchy act as long-wave amplitude equations for the \tGP hierarchy.
Recall our notations $\bm{0} = (0, \dots, 0) \in \R^{N+1}$ and $M = \floor{\frac{N - 1}{2}}$.
Define matrices $\bm{b}^+, \bm{b}^- \in \R^{(M+1) \times (N+1)}$ by
\begin{align*}
    \bm{b}^+_{m,n} &= - (-1)^n \mathds{1}_{\{m=\floor{\frac{n-1}{2}}\}}
    & \bm{b}^-_{m,n} &= \mathds{1}_{\{m=\floor{\frac{n-1}{2}}\}} \,.
\end{align*}
The purpose of this matrix is to match the correct flows in the \NLS hierarchy with the corresponding flows in the \KdV hierarchy, as described in Theorem \ref{thm:2}.
\begin{theorem}[Quantitative approximation of the \rtGP hierarchy by the \KdV hierarchy] \label{thm:1}
    Let $N \geq 2$, $s \in \N$, and $q_0: \R \rightarrow \C$ satisfy the assumptions of Proposition \ref{prop:2} for some $s' \in \N$ and $q_\ast, p_0: \R \rightarrow \C$.
    Assume in addition, as in (iv) of Proposition \ref{prop:2}, that $\mc{H}^\GP_2(q_0) < \frac83$ so that $|q(\bm{t}, x)| > \delta > 0$ and hence we can define for some $\epsilon > 0$ the function $W$ according to the ansatz \eqref{eqn:ansatz}.
    We have $W \in C(\R^{N+1}; H^{s-1})$, and if $\epsilon < \epsilon_0(s, \|W(\bm{0})\|_{H^{s-1}})$ then
    \begin{align} \label{eqn:approx-global-general}
        \|W\|_{L^\infty_{\bm{T}} H^{h-1}_X} &\leq C(h, \|W(\bm{0})\|_{H^h}) \|W(\bm{0})\|_{H^h} \qquad \qquad \forall\, h \in \{2, \dots, s - 1\} \,.
    \end{align}
    Furthermore, $W(\bm{c}^{\NLS \mapsto \rtGP} \bm{T}) \in C(\R^{N+1}; H^{s-1})$ solves \eqref{eqn:rtGP-n} in the sense of distributions.

    Let $\pm U_\pm \in C(\R^{M+1}; H^{s-1})$ be the solutions to the $M$-truncated \KdV hierarchy, given by Proposition \ref{prop:1}, with initial data $U_\pm(\bm{0}) = W_\pm(\bm{0})$.
    They satisfy
    \begin{align}
        \|U_\pm\|_{L^\infty_{\bm{T}} H^h_X} &\leq C(h, \|U_\pm(\bm{0})\|_{H^{-1}}) \|U_\pm(\bm{0})\|_{H^h} \qquad \qquad \forall\, h \in \{0, \dots, s - 1\} \,.
    \end{align}
    
    Define $[\bm{T}] = |T_2| + \dots + |T_N|$.
    For all $h \in \{0, \dots, \floor{\frac{s-N}{2}}\}$ there exists $C = C(N, h, \delta) > 0$ such that
    \begin{align} \label{eqn:approx-main}
        \left\|\begin{pmatrix}
            U_+(\bm{b}^+ \bm{T})
            \\ U_-(\bm{b}^- \bm{T})
        \end{pmatrix}
        - W(\bm{c}^{\NLS \mapsto \rtGP} \bm{T})\right\|_{H^h} 
        &\leq \epsilon^2 [\bm{T}]^{\frac12} \exp\left( [\bm{T}] C \big(1 + \|(W, U)\|_{L^\infty_{\bm{T}} H^{2M}_X}^M\big) \right) 
        \\ \nonumber &\times C \big(1 + \|W\|_{L^\infty_{\bm{T}} H^{N+1+h}_X}^{N+4}\big) 
        \big(1 + [\bm{T}] \|(W, U)\|_{L^\infty_{\bm{T}} H^{2(M+h)}_X}^{N-1}\big)^{\frac{h}{2}} \,.
    \end{align}
\end{theorem}
This result represents a collection of approximations, and the timescales involved have several edge cases.
Note, for example, that $\bm{c}^{\NLS \mapsto \rtGP}$ contains negative powers of $\epsilon$.
We shall therefore take some time now to explain this result.
It is convenient to introduce for a \textit{time-independent} function $q: \R \rightarrow \C$ with $|q| > \delta > 0$ the notation
\begin{align*}
    [q \mapsto W](q)(X) &= \frac{1}{\epsilon^2} \left(\frac12\Imag\left[\frac{q_x}{q}\right] + (|q| - 1), \frac12\Imag\left[\frac{q_x}{q}\right] - (|q| - 1)\right)(x) \,.
\end{align*}
This means that if $q: \R^{N+1} \times \R \rightarrow \C$ is a solution to the $N$-truncated \NLS hierarchy, and we define $W: \R^{N+1} \times \R \rightarrow \R^2$ according to \eqref{eqn:ansatz}, then the following hold:
\begin{align*}
    q(\bm{t}) &\text{ solves the $N$-truncated \NLS hierarchy} & &\text{ and } & [q \mapsto W](q(\bm{t})) &= W(\bm{T})
    \\ q(\bm{c}^{\NLS \mapsto \GP} \bm{t}) &\text{ solves the $N$-truncated \GP hierarchy} & &\text{ and } & [q \mapsto W](q(\bm{c}^{\NLS \mapsto \GP} \bm{t})) &= W(\bm{c}^{\NLS \mapsto \rGP} \bm{T}) 
    \\ q(\bm{c}^{\NLS \mapsto \tGP} \bm{t}) &\text{ solves the $N$-truncated \tGP hierarchy} & &\text{ and } & [q \mapsto W](q(\bm{c}^{\NLS \mapsto \tGP} \bm{t})) &= W(\bm{c}^{\NLS \mapsto \rtGP} \bm{T}) \,.
\end{align*}
We see now that Theorem \ref{thm:1} is an approximation result written in terms of the \tGP/\rtGP hierarchy.
We want to sketch the consequences of Theorem \ref{thm:1} for the equations \hyperref[eqn:NLS-n]{$(\NLS_n)$}, \hyperref[eqn:GP-n]{$(\GP_n)$}, and \hyperref[eqn:tGP-n]{$(\tGP_n)$} for $n \in \{0, \dots, 7\}$.
To keep the notation simple, we shall use the formal statement
\begin{align*}
    W(\bm{c}^{\NLS \mapsto \rtGP} \bm{T}) &= 
    \begin{pmatrix}
        U_+(\bm{b}^+ \bm{T})
        \\ U_-(\bm{b}^- \bm{T})
    \end{pmatrix}
    + \begin{cases} 
        O(\epsilon^2) \quad \text{ as long as } [\bm{T}] \lesssim 1 &, [\bm{T}] > 0
        \\ 0 \quad \text{ for all } \bm{T} \in \R^{N+1} &, [\bm{T}] = 0
    \end{cases}
\end{align*}
to represent the precise statement \eqref{eqn:approx-main}, which as explained above is an approximation result for the \tGP hierarchy.
A corresponding approximation result for the \NLS hierarchy reads
\begin{align*}
    W(\bm{T}) &=
    \begin{pmatrix}
        U_+(\bm{b}^+ \bm{c}^{\rtGP \mapsto \NLS} \bm{T})
        \\ U_-(\bm{b}^- \bm{c}^{\rtGP \mapsto \NLS} \bm{T})
    \end{pmatrix}
    + \begin{cases} 
        O(\epsilon^2) \quad \text{ as long as } [\bm{c}^{\rtGP \mapsto \NLS} \bm{T}] \lesssim 1 &, [\bm{c}^{\rtGP \mapsto \NLS} \bm{T}] > 0
        \\ 0 \quad \text{ for all } \bm{T} \in \R^{N+1} &, [\bm{c}^{\rtGP \mapsto \NLS} \bm{T}] = 0
    \end{cases} \,,
\end{align*}
and for the \GP hierarchy it is
\begin{align*}
    W(\bm{c}^{\NLS \mapsto \rGP} \bm{T}) &= 
    \begin{pmatrix}
        U_+(\bm{b}^+ \bm{c}^{\rtGP \mapsto \rGP} \bm{T})
        \\ U_-(\bm{b}^- \bm{c}^{\rtGP \mapsto \rGP} \bm{T})
    \end{pmatrix}
    + \begin{cases} 
        O(\epsilon^2) \quad \text{ as long as } [\bm{c}^{\rtGP \mapsto \rGP} \bm{T}] \lesssim 1 &, [\bm{c}^{\rtGP \mapsto \rGP} \bm{T}] > 0
        \\ 0 \quad \text{ for all } \bm{T} \in \R^{N+1} &, [\bm{c}^{\rtGP \mapsto \rGP} \bm{T}] = 0
    \end{cases} \,.
\end{align*}
Here the dependence of various time-related quantities on $\epsilon$ is not clearly visible, so let us write it in a more explicit form.
We assume that $N$ is even. The approximation results can be expressed as
\begin{align*}
    [q \mapsto W](q(\bm{t})) &= \begin{pmatrix}
        U_+\left( \left((\sqrt{2} \epsilon)^{2m+1} \big( (\bm{c}^{\tGP \mapsto \NLS} \bm{t})_{2m+1} - 2 (\bm{c}^{\tGP \mapsto \NLS} \bm{t})_{2m+2} \big) \right)_{0 \leq m \leq M} \right)
        \\ U_-\left( \left((\sqrt{2} \epsilon)^{2m+1} \big( (\bm{c}^{\tGP \mapsto \NLS} \bm{t})_{2m+1} + 2 (\bm{c}^{\tGP \mapsto \NLS} \bm{t})_{2m+2} \big) \right)_{0 \leq m \leq M} \right)
    \end{pmatrix}
    \\ &+ \begin{cases} 
        O(\epsilon^2) \quad \text{ as long as } [(\nabla_{\bm{t}} \bm{T}) \bm{c}^{\tGP \mapsto \NLS} \bm{t}] \lesssim 1 &, [(\nabla_{\bm{t}} \bm{T}) \bm{c}^{\tGP \mapsto \NLS} \bm{t}] > 0
        \\ 0 \quad \text{ for all } \bm{T} \in \R^{N+1} &, [(\nabla_{\bm{t}} \bm{T}) \bm{c}^{\tGP \mapsto \NLS} \bm{t}] = 0
    \end{cases}
\end{align*}
for the \NLS hierarchy,
\begin{align*}
    [q \mapsto W](q(\bm{c}^{\NLS \mapsto \GP} \bm{t})) &= \begin{pmatrix}
        U_+\left( \left((\sqrt{2} \epsilon)^{2m+1} \big( (\bm{c}^{\tGP \mapsto \GP} \bm{t})_{2m+1} - 2 (\bm{c}^{\tGP \mapsto \GP} \bm{t})_{2m+2} \big) \right)_{0 \leq m \leq M} \right)
        \\ U_-\left( \left((\sqrt{2} \epsilon)^{2m+1} \big( (\bm{c}^{\tGP \mapsto \GP} \bm{t})_{2m+1} + 2 (\bm{c}^{\tGP \mapsto \GP} \bm{t})_{2m+2} \big) \right)_{0 \leq m \leq M} \right)
    \end{pmatrix}
    \\ &+ \begin{cases} 
        O(\epsilon^2) \quad \text{ as long as } [(\nabla_{\bm{t}} \bm{T}) \bm{c}^{\tGP \mapsto \GP} \bm{t}] \lesssim 1 &, [(\nabla_{\bm{t}} \bm{T}) \bm{c}^{\tGP \mapsto \GP} \bm{t}] > 0
        \\ 0 \quad \text{ for all } \bm{T} \in \R^{N+1} &, [(\nabla_{\bm{t}} \bm{T}) \bm{c}^{\tGP \mapsto \GP} \bm{t}] = 0
    \end{cases}
\end{align*}
for the \GP hierarchy, and
\begin{align*}
    [q \mapsto W](q(\bm{c}^{\NLS \mapsto \tGP} \bm{t})) &= \begin{pmatrix}
        U_+\left( \left((\sqrt{2} \epsilon)^{2m+1} \big( t_{2m+1} - 2 t_{2m+2} \big) \right)_{0 \leq m \leq M} \right)
        \\ U_-\left( \left((\sqrt{2} \epsilon)^{2m+1} \big( t_{2m+1} + 2 t_{2m+2} \big) \right)_{0 \leq m \leq M} \right)
    \end{pmatrix}
    \\ &+ \begin{cases} 
        O(\epsilon^2) \quad \text{ as long as } [(\nabla_{\bm{t}} \bm{T}) \bm{t}] \lesssim 1 &, [(\nabla_{\bm{t}} \bm{T}) \bm{t}] > 0
        \\ 0 \quad \text{ for all } \bm{T} \in \R^{N+1} &, [(\nabla_{\bm{t}} \bm{T}) \bm{t}] = 0
    \end{cases}
\end{align*}
for the \tGP hierarchy.

\subsection{Examples}
We use the above formulations to give explicit approximation results along one-dimensional time-rays in $\R^{N+1}$ for the cases $N = 0, \dots, 7$.
\begin{example}[$(\NLS_0) = (\GP_0) = (\tGP_0) =$ ``phase rotation''] \label{ex:0}
    Take
    \begin{align*}
        N &= 0 & M &= - 1 & \bm{T} &= (2 (\sqrt{2} \epsilon)^{-1} t_0) \,.
    \end{align*}
    Then
    \begin{align*}
        [q \mapsto W](q(\bm{t})) = [q \mapsto W](q(\bm{c}^{\NLS \mapsto \GP} \bm{t})) = [q \mapsto W](q(\bm{c}^{\NLS \mapsto \tGP} \bm{t})) 
        &= \begin{pmatrix}
            U_+(0)
            \\ U_-(0)
        \end{pmatrix}
        \quad \text{ for all } t_0 \in \R \,.
    \end{align*}
    Here we use the convention $\R^{M+1} = \R^0 = \{0\}$.
\end{example}
\begin{example}[$(\NLS_1) = (\GP_1) = (\tGP_1) =$ ``transport''] \label{ex:1}
    Take
    \begin{align*}
        N &= 1 & M &= 0 & \bm{T} &= (0, \sqrt{2} \epsilon t_1) \,.
    \end{align*}
    Then
    \begin{align*}
        [q \mapsto W](q(\bm{t})) = [q \mapsto W](q(\bm{c}^{\NLS \mapsto \GP} \bm{t})) = [q \mapsto W](q(\bm{c}^{\NLS \mapsto \tGP} \bm{t})) 
        &= \begin{pmatrix}
            U_+(\sqrt{2} \epsilon t_1)
            \\ U_-(\sqrt{2} \epsilon t_1)
        \end{pmatrix}
        \quad \text{ for all } t_1 \in \R \,.
    \end{align*}
\end{example}
\begin{example}[$(\NLS_2) = \NLS$ and $(\GP_2) = (\tGP_2) = \GP$] \label{ex:2}
    Take
    \begin{align*}
        N &= 2 & M &= 0 & \bm{T} &= (0, 0, 2 \sqrt{2} \epsilon t_2) \,.
    \end{align*}
    Then
    \begin{align*}
        [q \mapsto W](q(\bm{t})) &= \begin{pmatrix}
            U_+(- 2 \sqrt{2} \epsilon t_2)
            \\ U_-(2 \sqrt{2} \epsilon t_2)
        \end{pmatrix} + O(\epsilon^2)
        \quad \text{ as long as } |t_2| \lesssim \epsilon^{-1} 
        \\ [q \mapsto W](q(\bm{c}^{\NLS \mapsto \GP} \bm{t})) &= \begin{pmatrix}
            U_+(- 2 \sqrt{2} \epsilon t_2)
            \\ U_-(2 \sqrt{2} \epsilon t_2)
        \end{pmatrix} + O(\epsilon^2)
        \quad \text{ as long as } |t_2| \lesssim \epsilon^{-1}
        \\ [q \mapsto W](q(\bm{c}^{\NLS \mapsto \tGP} \bm{t})) &= \begin{pmatrix}
            U_+(- 2 \sqrt{2} \epsilon t_2)
            \\ U_-(2 \sqrt{2} \epsilon t_2)
        \end{pmatrix} + O(\epsilon^2)
        \quad \text{ as long as } |t_2| \lesssim \epsilon^{-1} \,.
    \end{align*}
    In fact, we have $[q \mapsto W](q(\bm{t})) = [q \mapsto W](q(\bm{c}^{\NLS \mapsto \GP} \bm{t})) = [q \mapsto W](q(\bm{c}^{\NLS \mapsto \tGP} \bm{t}))$.
    This is a consequence of the fact that the difference between these flows for $q$ is one of phase rotation, which does not show up in the variable $W$.

    This result corresponds to \cite[Theorem 2]{BethuelDanchinSmets2010}, i.~e. it is an approximation of \GP by two transport equations.
    As discussed above, it is known from \cite{BethuelGravejatSautSmets2009,BethuelGravejatSautSmets2010} that in the subsequent order \KdV = \hyperref[eqn:KdV-n]{$(\KdV_1)$} shows up as the approximating equation.
    It is natural to conjecture that this may be generalized for the cases $N \geq 3$, i.~e. to characterize the effect of the terms $O(\epsilon^2)$ up to $O(\epsilon^4)$.
    Remark \ref{rem:conjecture} above provides some evidence towards this.
\end{example}
\begin{example}[$(\NLS_3) = \mKdV$, $(\GP_3)$, and $(\tGP_3)$] \label{ex:3}
    Take
    \begin{align*}
        N &= 3 & M &= 1 & \bm{T} &= (0, 0, 0, (\sqrt{2} \epsilon)^3 t_3) \,.
    \end{align*}
    Then
    \begin{align*}
        [q \mapsto W](q(\bm{t})) &= \begin{pmatrix}
            U_+(\sqrt{2} \epsilon \bm{c}^{\tGP \mapsto \NLS}_{1,3} t_3, (\sqrt{2} \epsilon)^3 t_3)
            \\ U_-(\sqrt{2} \epsilon \bm{c}^{\tGP \mapsto \NLS}_{1,3} t_3, (\sqrt{2} \epsilon)^3 t_3)
        \end{pmatrix} + O(\epsilon^2)
        \quad \text{ as long as } |t_3| \lesssim \epsilon^{-3}
        \\ [q \mapsto W](q(\bm{c}^{\NLS \mapsto \GP} \bm{t})) &= \begin{pmatrix}
            U_+(\sqrt{2} \epsilon \bm{c}^{\tGP \mapsto \GP}_{1,3} t_3, (\sqrt{2} \epsilon)^3 t_3)
            \\ U_-(\sqrt{2} \epsilon \bm{c}^{\tGP \mapsto \GP}_{1,3} t_3, (\sqrt{2} \epsilon)^3 t_3)
        \end{pmatrix} + O(\epsilon^2)
        \quad \text{ as long as } |t_3| \lesssim \epsilon^{-3}
        \\ [q \mapsto W](q(\bm{c}^{\NLS \mapsto \tGP} \bm{t})) &= \begin{pmatrix}
            U_+(0, (\sqrt{2} \epsilon)^3 t_3)
            \\ U_-(0, (\sqrt{2} \epsilon)^3 t_3)
        \end{pmatrix} + O(\epsilon^2)
        \quad \text{ as long as } |t_3| \lesssim \epsilon^{-3} \,.
    \end{align*}
    This is the first new result that Theorem \ref{thm:1} yields. 
    Here \mKdV and renormalized versions of \mKdV are approximated by two transport equations in leading order, and then by two \KdV equations at next-to-leading order.
    Observe that for the \NLS and \GP hierarchies the timescale is $\epsilon^{-3}$, which is better than the expected $\epsilon^{-1}$. 
    This is because the approximation of transport equations by transport equations has zero error.
\end{example}
\begin{example}[$(\NLS_4) =$ ``$4$-th order \NLS'', $(\GP_4) = (\tGP_4) =$ ``$4$-th order \GP''] \label{ex:4}
    Take
    \begin{align*}
        N &= 4 & M &= 1 & \bm{T} &= (0, 0, 0, 0, 2 (\sqrt{2} \epsilon)^3 t_4) \,.
    \end{align*}
    Then
    \begin{align*}
        [q \mapsto W](q(\bm{t})) &= \begin{pmatrix}
            U_+(- 2 \sqrt{2} \epsilon \bm{c}^{\tGP \mapsto \NLS}_{2,4} t_4, - 2 (\sqrt{2} \epsilon)^3 t_4)
            \\ U_-(2 \sqrt{2} \epsilon \bm{c}^{\tGP \mapsto \NLS}_{2,4} t_4, 2 (\sqrt{2} \epsilon)^3 t_4)
        \end{pmatrix} + O(\epsilon^2)
        \quad \text{ as long as } |t_4| \lesssim \epsilon^{-1}
        \\ [q \mapsto W](q(\bm{c}^{\NLS \mapsto \GP} \bm{t})) &= \begin{pmatrix}
            U_+(0, - 2 (\sqrt{2} \epsilon)^3 t_4)
            \\ U_-(0, 2 (\sqrt{2} \epsilon)^3 t_4)
        \end{pmatrix} + O(\epsilon^2)
        \quad \text{ as long as } |t_4| \lesssim \epsilon^{-3}
        \\ [q \mapsto W](q(\bm{c}^{\NLS \mapsto \tGP} \bm{t})) &= \begin{pmatrix}
            U_+(0, - 2 (\sqrt{2} \epsilon)^3 t_4)
            \\ U_-(0, 2 (\sqrt{2} \epsilon)^3 t_4)
        \end{pmatrix} + O(\epsilon^2)
        \quad \text{ as long as } |t_4| \lesssim \epsilon^{-3} \,.
    \end{align*}
    For the \NLS hierarchy the case $n = 2$ ``dominates'', leading to the shorter timescale $\epsilon^{-1}$. 
    For the \tGP hierarchy the dynamics are cleanly separated, leading to a better timescale.
    Remark \ref{rem:conjecture} suggests that studying \hyperref[eqn:rtGP-n]{$(\rtGP_4)$} in a ``co-\hyperref[eqn:rtGP-n]{$(\rtGP_3)$}-moving'' frame allows us to see, on a longer timescale, the dynamics of \hyperref[eqn:KdV-n]{$(\KdV_2)$} in one component of our solution $W$.
\end{example}
\begin{example}[$(\NLS_5) =$ ``$5$-th order \mKdV'', $(\GP_5)$, and $(\tGP_5)$] \label{ex:5}
    Take
    \begin{align*}
        N &= 5 & M &= 2 & \bm{T} &= (0, 0, 0, 0, 0, (\sqrt{2} \epsilon)^5 t_5) \,.
    \end{align*}
    Then
    \begin{align*}
        [q \mapsto W](q(\bm{t})) &= \begin{pmatrix}
            U_+(\sqrt{2} \epsilon \bm{c}^{\tGP \mapsto \NLS}_{1,5} t_5, (\sqrt{2} \epsilon)^3 \bm{c}^{\tGP \mapsto \NLS}_{3,5} t_5, (\sqrt{2} \epsilon)^5 t_5)
            \\ U_-(\sqrt{2} \epsilon \bm{c}^{\tGP \mapsto \NLS}_{1,5} t_5, (\sqrt{2} \epsilon)^3 \bm{c}^{\tGP \mapsto \NLS}_{3,5} t_5, (\sqrt{2} \epsilon)^5 t_5)
        \end{pmatrix} 
        \\ &+ O(\epsilon^2) \quad \text{ as long as } |t_5| \lesssim \epsilon^{-3}
        \\ [q \mapsto W](q(\bm{c}^{\NLS \mapsto \GP} \bm{t})) &= \begin{pmatrix}
            U_+(\sqrt{2} \epsilon \bm{c}^{\tGP \mapsto \GP}_{1,5} t_5, (\sqrt{2} \epsilon)^3 \bm{c}^{\tGP \mapsto \GP}_{3,5} t_5, (\sqrt{2} \epsilon)^5 t_5)
            \\ U_-(\sqrt{2} \epsilon \bm{c}^{\tGP \mapsto \GP}_{1,5} t_5, (\sqrt{2} \epsilon)^3 \bm{c}^{\tGP \mapsto \GP}_{3,5} t_5, (\sqrt{2} \epsilon)^5 t_5)
        \end{pmatrix} 
        \\ &+ O(\epsilon^2) \quad \text{ as long as } |t_5| \lesssim \epsilon^{-3}
        \\ [q \mapsto W](q(\bm{c}^{\NLS \mapsto \tGP} \bm{t})) &= \begin{pmatrix}
            U_+(0, 0, (\sqrt{2} \epsilon)^5 t_5)
            \\ U_-(0, 0, (\sqrt{2} \epsilon)^5 t_5)
        \end{pmatrix} 
        \\ &+ O(\epsilon^2) \quad \text{ as long as } |t_5| \lesssim \epsilon^{-5} \,.
    \end{align*}
    Here \hyperref[eqn:KdV-n]{$(\KdV_2)$} shows up for the first time. 
    Note that once again the approximation of transport equations by transport equations has no error, yielding the better-than-expected timescale $\epsilon^{-3}$ for the \NLS and \GP hierarchies.
\end{example}
\begin{example}[$(\NLS_6) =$ ``$6$-th order \NLS'', $(\GP_6) = (\tGP_6) =$ ``$6$-th order \GP''] \label{ex:6}
    Take
    \begin{align*}
        N &= 6 & M &= 2 & \bm{T} &= (0, 0, 0, 0, 0, 0, 2 (\sqrt{2} \epsilon)^5 t_6) \,.
    \end{align*}
    Then
    \begin{align*}
        [q \mapsto W](q(\bm{t})) &= \begin{pmatrix}
            U_+(- 2 \sqrt{2} \epsilon \bm{c}^{\tGP \mapsto \NLS}_{2,6} t_6, - 2 (\sqrt{2} \epsilon)^3 \bm{c}^{\tGP \mapsto \NLS}_{4,6} t_6, - 2 (\sqrt{2} \epsilon)^5 t_6)
            \\ U_-(2 \sqrt{2} \epsilon \bm{c}^{\tGP \mapsto \NLS}_{2,6} t_6, 2 (\sqrt{2} \epsilon)^3 \bm{c}^{\tGP \mapsto \NLS}_{4,6} t_6, 2 (\sqrt{2} \epsilon)^5 t_6)
        \end{pmatrix} 
        \\ &+ O(\epsilon^2) \quad \text{ as long as } |t_6| \lesssim \epsilon^{-1}
        \\ [q \mapsto W](q(\bm{c}^{\NLS \mapsto \GP} \bm{t})) &= \begin{pmatrix}
            U_+(0, 0, - 2 (\sqrt{2} \epsilon)^5 t_6)
            \\ U_-(0, 0, 2 (\sqrt{2} \epsilon)^5 t_6)
        \end{pmatrix} 
        \\ &+ O(\epsilon^2) \quad \text{ as long as } |t_6| \lesssim \epsilon^{-5}
        \\ [q \mapsto W](q(\bm{c}^{\NLS \mapsto \tGP} \bm{t})) &= \begin{pmatrix}
            U_+(0, 0, - 2 (\sqrt{2} \epsilon)^5 t_6)
            \\ U_-(0, 0, 2 (\sqrt{2} \epsilon)^5 t_6)
        \end{pmatrix} 
        \\ &+ O(\epsilon^2) \quad \text{ as long as } |t_6| \lesssim \epsilon^{-5} \,.
    \end{align*}
\end{example}
\begin{example}[$(\NLS_7) =$ ``$7$-th order \mKdV'', $(\GP_7)$, and $(\tGP_7)$] \label{ex:7}
    Take
    \begin{align*}
        N &= 7 & M &= 3 & \bm{T} &= (0, 0, 0, 0, 0, 0, 0, (\sqrt{2} \epsilon)^7 t_7) \,.
    \end{align*}
    Then
    \begin{align*}
        [q \mapsto W](q(\bm{t})) &= \begin{pmatrix}
            U_+(\sqrt{2} \epsilon \bm{c}^{\tGP \mapsto \NLS}_{1,7} t_7, (\sqrt{2} \epsilon)^3 \bm{c}^{\tGP \mapsto \NLS}_{3,7} t_7, (\sqrt{2} \epsilon)^5 \bm{c}^{\tGP \mapsto \NLS}_{5,7} t_7, (\sqrt{2} \epsilon)^7 t_7)
            \\ U_-(\sqrt{2} \epsilon \bm{c}^{\tGP \mapsto \NLS}_{1,7} t_7, (\sqrt{2} \epsilon)^3 \bm{c}^{\tGP \mapsto \NLS}_{3,7} t_7, (\sqrt{2} \epsilon)^5 \bm{c}^{\tGP \mapsto \NLS}_{5,7} t_7, (\sqrt{2} \epsilon)^7 t_7)
        \end{pmatrix} 
        \\ &+ O(\epsilon^2) \quad \text{ as long as } |t_7| \lesssim \epsilon^{-3}
        \\ [q \mapsto W](q(\bm{c}^{\NLS \mapsto \GP} \bm{t})) &= \begin{pmatrix}
            U_+(\sqrt{2} \epsilon \bm{c}^{\tGP \mapsto \GP}_{1,7} t_7, (\sqrt{2} \epsilon)^3 \bm{c}^{\tGP \mapsto \GP}_{3,7} t_7, (\sqrt{2} \epsilon)^5 \bm{c}^{\tGP \mapsto \GP}_{5,7} t_7, (\sqrt{2} \epsilon)^7 t_7)
            \\ U_-(\sqrt{2} \epsilon \bm{c}^{\tGP \mapsto \GP}_{1,7} t_7, (\sqrt{2} \epsilon)^3 \bm{c}^{\tGP \mapsto \GP}_{3,7} t_7, (\sqrt{2} \epsilon)^5 \bm{c}^{\tGP \mapsto \GP}_{5,7} t_7, (\sqrt{2} \epsilon)^7 t_7)
        \end{pmatrix} 
        \\ &+ O(\epsilon^2) \quad \text{ as long as } |t_7| \lesssim \epsilon^{-3}
        \\ [q \mapsto W](q(\bm{c}^{\NLS \mapsto \tGP} \bm{t})) &= \begin{pmatrix}
            U_+(0, 0, 0, (\sqrt{2} \epsilon)^7 t_7)
            \\ U_-(0, 0, 0, (\sqrt{2} \epsilon)^7 t_7)
        \end{pmatrix} 
        \\ &+ O(\epsilon^2) \quad \text{ as long as } |t_7| \lesssim \epsilon^{-7} \,.
    \end{align*}
\end{example}
At this point we stop because all the edge cases have been exhibited and the pattern continues as expected.

\subsection{Notations and definitions}
We use $\N = \{n \in \Z: n \geq 0\}$ but note that in some cases where we write $n \in \N$, and when the intended interpretation is clear from the context, $n = 0$ will be an inadmissible value.

We often write $A \leq C(\text{things}) B$ to indicate that there exists a constant, depending only on the given parameters, for which the estimate holds.
Such constants may change from one line to the next without being given a new name.

Whenever we consider a function space such as $H^s$ without explicit domain, the domain is implicitly assumed to be $\R$. 
The target space may be $\R$, $\C$, or corresponding spaces of vectors or matrices, and must be determined from context. 

We denote by $\bm{t} = (t_0, \dots, t_N), \bm{T} = (T_0, \dots, T_N), \bm{0} = (0, \dots, 0) \in \R^{N+1}$ certain vectors of times.
For matrices $\bm{c} \in \C^{(N+1) \times (N+1)}$ we denote the entry in the $n$-th row and $m$-th column by $\bm{c}_{n,m}$.
The only exception is the matrix of Jost solutions $\Phi^{\Eqn,\pm} = (\Phi^{\Eqn,\pm}_1 |\Phi^{\Eqn,\pm}_2)$.

We define the Pauli matrices
\begin{align*}
    \sigma_3 &= \begin{pmatrix} 1 & 0 \\ 0 & - 1 \end{pmatrix}
    & \sigma_1 &= \begin{pmatrix} 0 & 1 \\ 1 & 0 \end{pmatrix} \,.
\end{align*}

For real numbers $x, y \in \R$, we define $x \land y = \min\{x, y\}$ and $x \lor y = \max\{x, y\}$, as well as the Japanese bracket $\langle x \rangle = \sqrt{x^2 + 1}$ and the floor and ceiling functions $\floor{x} = \max\{n \in \Z: n \leq x\}$ and $\ceil{x} = \min\{n \in \Z: n \geq x\}$.

For functions $f: \R \rightarrow \C^n$ we use the following convention for the Fourier transform:
\begin{align*}
    \ha{f}(\xi) &= \frac{1}{\sqrt{2 \pi}} \int_{\R} e^{- i x \xi} f(x) \dd x \,.
\end{align*}
We extend the Fourier transform to the tempered distributions $\mc{S}'(\R; \C^n)$ as usual.

We denote by $\mc{Q}_j = i^{j-1} \{z \in \C: \Re z > 0, \Imag z > 0\}$, $j \in \{1, 2, 3, 4\}$ the four open quadrants of the complex plane.

\textbf{Acknowledgements.}
I thank my supervisor Xian Liao for her support and guidance over the past few years, and for many enjoyable and insightful discussions.

\section{\texorpdfstring{Definition and approximation of the Hamiltonians $\mc{H}^\rGP_n$ by $\mc{E}^\KdV_n$}
{Definition and approximation of the Hamiltonians 𝓗ᴳᴾₙ by 𝓔ᴷᵈⱽₙ}}

\label{section:transmissioncoefficient}
\subsection{Lax pairs}
\label{section:lax-pairs}
Formally, a function $q: \R^2 \rightarrow \C$ is a solution to \GP if and only if the time-dependent operators
\begin{align*}
    L^\GP &= \begin{pmatrix} i \del_x & - i q \\ i \conj{q} & - i \del_x \end{pmatrix}
\end{align*}
and
\begin{align*}
    P^\GP &=  i \begin{pmatrix}
        2 \del_x^2 - q \conj{q} + 1 & - 2 q \del_x - q_x \\ 2 \conj{q} \del_x + \conj{q}_x & - 2 \del_x^2 + q \conj{q} - 1
    \end{pmatrix}
\end{align*}
solve the so-called Lax equation
\begin{align*}
    \del_t L^\GP &= [P^\GP, L^\GP] \,.
\end{align*}
Here $L^\GP$ is called the Lax operator of \GP and $(L^\GP, P^\GP)$ the Lax pair.
Setting $L^\NLS = L^\GP$ and
\begin{align*}
    P^\NLS &= i \begin{pmatrix}
        2 \del_x^2 - q \conj{q} & - 2 q \del_x - q_x \\ 2 \conj{q} \del_x + \conj{q}_x & - 2 \del_x^2 + q \conj{q}
    \end{pmatrix}
\end{align*}
yields a Lax pair for \NLS.
A direct consequence of the Lax pair formalism is that eigenvalues of the Lax operator are conserved quantities. 
A deeper consequence is that the potential $q$ can be reconstructed from the scattering data associated to the Lax operator $L^\GP$,
and this scattering data undergoes a linear evolution in time. 
As such, one can formally linearize \NLS or \GP by conjugating the flow with the direct and inverse scattering transforms.
This is called the inverse scattering transform (IST) method, and it is the reason why \NLS and \GP are considered to be completely integrable PDEs.
Note that understanding the mapping properties of the IST on standard function spaces is non-trivial, especially for nonzero boundary conditions such as in the case of \GP.
For \NLS and \GP the scattering data is obtained from the Zakharov-Shabat scattering problem, which is given as the cases $\Eqn \in \{\NLS, \GP\}$ in Definition \ref{def:jost} below.
Since the IST method is not the focus of this work, we shall merely highlight here the seminal papers \cite{AblowitzKaupNewellSegur1974,ZakharovShabat} and refer to the recent review \cite{Prinari} regarding the development of the IST method, with a focus on \GP.
We emphasize the connection to \cite[Section 2]{LiaoWegner2025} in a preceding work, where the theory that we use here is set up, with matching notation, and more background on complete integrability and the IST method is given.

We shall now give Lax pairs for \hGP, \rGP and \KdV.
Set $q = a e^{i \phi}$ and $\theta = \frac{i}{2} (\phi - \frac{\pi}{2})$.
In \cite{LiaoPlum2023} X.~Liao and M.~Plum note that the explicit unitary matrix
\begin{align*}
    M &= \frac{1}{\sqrt{2}} \begin{pmatrix}
        e^{- \theta} & e^\theta
        \\ e^{- \theta} &  - e^\theta
    \end{pmatrix}
    & M^{-1} &= \frac{1}{\sqrt{2}} \begin{pmatrix}
        e^\theta & e^{- \theta}
        \\ e^\theta & - e^{- \theta}
    \end{pmatrix}
\end{align*}
yields the unitary equivalence
\begin{align*}
    L^\GP &= M^{-1} L^\hGP M \,,
\end{align*}
where
\begin{align*}
    L^\hGP &= \begin{pmatrix}
        - w_- + 1 & i \del_x \\ i \del_x & - w_+ - 1
    \end{pmatrix} \,.
\end{align*}
Together with
\begin{align*}
    P^\hGP &= \begin{pmatrix}
        - \frac12 (3 w_- + w_+)_x - (3 w_- + w_+ - 2) \del_x
        & 2 i \del_x^2 - \frac{i}{2} \frac{(w_+ - w_-)_{xx}}{w_+ - w_- + 2}
        \\ 2 i \del_x^2 - \frac{i}{2} \frac{(w_+ - w_-)_{xx}}{w_+ - w_- + 2}
        & - \frac12 (w_- + 3 w_+)_x - (w_- + 3 w_+ + 2) \del_x
    \end{pmatrix}
\end{align*}
this forms a Lax pair for the hydrodynamic formulation \hGP of \GP.

By rescaling, we can construct from $(L^\hGP, P^\hGP)$ the Lax pair
\begin{align*}
    L^\rGP &= \begin{pmatrix}
        - \epsilon^2 W_- + 1 & i \sqrt{2} \epsilon \del_X \\ i \sqrt{2} \epsilon \del_X & - \epsilon^2 W_+ - 1
    \end{pmatrix}
\end{align*}
and
\begin{align*}
    P^\rGP &= \begin{pmatrix} P^\rGP_{1,1} & P^\rGP_{1,2} \\ P^\rGP_{2,1} & P^\rGP_{2,2}
    \end{pmatrix}
    \text{ where } \begin{pmatrix} P^\rGP_{1,1} \\ P^\rGP_{1,2} \\ P^\rGP_{2,1} \\ P^\rGP_{2,2}
    \end{pmatrix}
    = \begin{pmatrix}
        - \frac{\epsilon}{\sqrt{2}} (3 \epsilon^2 W_- + \epsilon^2 W_+)_X - (3 \epsilon^2 W_- + \epsilon^2 W_+ - 2) \sqrt{2} \epsilon \del_X
        \\ 4 i \epsilon^2 \del_X^2 - i \epsilon^2 \frac{(\epsilon^2 W_+ - \epsilon^2 W_-)_{X\!X}}{\epsilon^2 W_+ - \epsilon^2 W_- + 2}
        \\ 4 i \epsilon^2 \del_X^2 - i \epsilon^2 \frac{(\epsilon^2 W_+ - \epsilon^2 W_-)_{X\!X}}{\epsilon^2 W_+ - \epsilon^2 W_- + 2}
        \\ - \frac{\epsilon}{\sqrt{2}} (\epsilon^2 W_- + 3 \epsilon^2 W_+)_X - (\epsilon^2 W_- + 3 \epsilon^2 W_+ + 2) \sqrt{2} \epsilon \del_X
    \end{pmatrix}
\end{align*}
for \rGP.
Lastly, we recall the well-known Lax pair
\begin{align*}
    \ti{L}^\KdV &= - \del_X^2 + U
    & P^\KdV &= - 4 \del_X^3 + 3 U_X + 6 U \del_X
\end{align*}
for \KdV.
For the sake of uniformity we treat the scattering problem for $\KdV$ as a first-order system, i.~e. we use instead of $\ti{L}^\KdV$ the Lax operator
\begin{align*}
    L^\KdV &= \begin{pmatrix}
        - \del_X^2 + U & 0
        \\ U_X & - \del_X^2 + U
    \end{pmatrix} \,.
\end{align*}

We are interested in these Lax pairs because the conservation of the spectral data of the Lax operator implies the existence of infinitely many conserved quantities. 
More specifically, we shall construct in the subsequent section the Hamiltonians $\mc{H}^\NLS_n, \mc{H}^\GP_n, \mc{H}^\hGP_n, \mc{H}^\rGP_n$, and $\mc{E}^\KdV_n$ from the transmission coefficients associated to each scattering problem.

\subsection{Jost solutions and their asymptotic expansion}
The Jost solutions for each scattering problem depend on a complex parameter $\lambda \in \C$.
We prefer to use for \GP and \hGP the alternative variable $z = \sqrt{\lambda^2 - 1}$, and for \rGP and \KdV the variable $k = \frac{z}{\sqrt{2} \epsilon}$. 
Due to the multivaluedness of the relation $\lambda \leftrightarrow z$, it is common to consider pairs $(\lambda, z) \in \C^2$ as lying on a Riemann surface.
Since we also consider the additional parameter $k$, we define
\begin{align} \label{eqn:surface}
    \K = \{(\lambda, z, k) \in \C^3: \lambda^2 - z^2 = 1 \text{ and } z =  \sqrt{2} \epsilon k\} \,.
\end{align}
We decompose $\K$ into two sheets $\K = \conj{\K_+} \cup \conj{\K_-}$ where
\begin{align*}
    \K_\pm &= \{(\lambda, z, k) \in \K: \lambda \in \C \setminus ((-\infty, -1] \cup [1, \infty)), \pm \Imag z > 0\} \,.
\end{align*}
By an abuse of notation, we may write $\lambda \in \conj{\K_\pm}$, $z \in \conj{\K_\pm}$, or $k \in \conj{\K_\pm}$ to mean $(\lambda, z, k) \in \conj{\K_\pm}$.
Similarly, for functions $f$ with domain $\conj{\K_\pm}$ we write $f(\lambda) = f(z) = f(k) = f(\lambda, z, k)$.

In order to give a uniform treatment of all the scattering problems under consideration, we define
\begin{align*}
    \Eqns &= \{\NLS, \GP, \hGP, \rGP, \KdV\}
\end{align*}
and write $\Eqn \in \Eqns$ when we wish for $\Eqn$ to refer to any of the equations under consideration.
Using this notation, we shall characterize all five scattering problems under consideration using a single set of variables:
\begin{align*}
    (q^\Eqn, q^\Eqn_\pm, x^\Eqn, \lambda^\Eqn, z^\Eqn) &= \begin{cases}
        (q, 0, x, \lambda, \lambda) &, \Eqn = \NLS
        \\ (q, q_\pm, x, \lambda, z) &, \Eqn = \GP
        \\ (w, 0, x, \lambda, z) &, \Eqn = \hGP
        \\ (W, 0, X, \lambda, k) &, \Eqn = \rGP
        \\ (U, 0, X, k^2, k) &, \Eqn = \KdV
    \end{cases} \,.
\end{align*}
In the following we always take $j \in \{1, 2\}$ and set $\gamma_1 = 1$ and $\gamma_2 = - 1$.

\begin{definition}[Jost solutions of \Eqn] \label{def:jost}
    Given a potential $q^\Eqn$ with $\lim_{x \rightarrow \pm \infty} q^\Eqn = q^\Eqn_\pm$, we define the Jost solution $\Phi^{\Eqn,\pm}_j$ as the solution to the scattering problem
    \begin{align}
        \nonumber \Phi^{\Eqn,\pm}_j: \R \times \conj{\K_{\mp \gamma_j}} &\longrightarrow \C^2
        \\ \label{eqn:jost-eigenvalue} L^\Eqn \;\Phi^{\Eqn,\pm}_j(x^\Eqn, z^\Eqn; q^\Eqn) &= \lambda^\Eqn \; \Phi^{\Eqn,\pm}_j(x^\Eqn, z^\Eqn; q^\Eqn) 
        \\ \label{eqn:jost-bdry}\lim_{x \rightarrow \pm \infty} \Phi^{\Eqn,\pm}_j(x^\Eqn, z^\Eqn; q^\Eqn) e^{x^\Eqn i z^\Eqn \gamma_j} &= E^{\Eqn,\pm}_j(z^\Eqn) \,.
    \end{align}
    We define in addition the modified Jost solutions as $\Psi^{\Eqn,\pm}_j = \Phi_j^{\Eqn,\pm} e^{x^\Eqn i z^\Eqn \gamma_j}$.
    The eigenvalue problem \eqref{eqn:jost-eigenvalue} is equivalent to
    \begin{align}
        \label{eqn:jost-ode} \del_{x^\Eqn} \Phi^{\Eqn,\pm}_j(x^\Eqn, z^\Eqn; q^\Eqn) &= A^\Eqn(x^\Eqn, z^\Eqn; q^\Eqn) \; \Phi^{\Eqn,\pm}_j(x^\Eqn, z^\Eqn; q^\Eqn) 
        \\ \label{eqn:mod-jost-ode} \del_{x^\Eqn} \Psi^{\Eqn,\pm}_j(x^\Eqn, z^\Eqn; q^\Eqn) &= (A^\Eqn(x^\Eqn, z^\Eqn; q^\Eqn) + i z^\Eqn \gamma_j) \; \Psi^{\Eqn,\pm}_j(x^\Eqn, z^\Eqn; q^\Eqn) 
    \end{align}
    The matrices $A^\Eqn(x^\Eqn, z^\Eqn; q^\Eqn), E^{\Eqn,\pm}(z^\Eqn) \in \C^{2 \times 2}$ are defined in the following table:
    {\setlength{\arraycolsep}{2pt}
    \setlength{\tabcolsep}{2.3pt}
    \begin{equation*}
        \begin{array}{|l|c|c|}
            \hline
            \Eqn & A^{\Eqn} & E^{\Eqn,\pm}
            \\ \hline
            \NLS 
            & \begin{pmatrix}
                - i \lambda & q
                \\ \conj{q} & i \lambda
            \end{pmatrix}
            & \begin{pmatrix}
                1 & 0
                \\ 0 & 1
            \end{pmatrix}
            \\ \hline
            \GP
            & \begin{pmatrix}
                - i \lambda & q
                \\ \conj{q} & i \lambda
            \end{pmatrix}
            & \begin{pmatrix}
                q_\pm & i (z - \lambda)
                \\ i (\lambda - z) & \conj{q}_\pm
            \end{pmatrix}
            \\ \hline
            \hGP
            & \begin{pmatrix}
                    0 & - i (\lambda + 1 + w_+)
                    \\ - i (\lambda - 1 + w_-) & 0
                \end{pmatrix}
            & \frac{i}{\sqrt{2}} \begin{pmatrix}
                \lambda - z + 1 & - \lambda + z - 1
                \\ - \lambda + z + 1 & - \lambda + z + 1
            \end{pmatrix} \begin{pmatrix}
                e^{\theta_\pm} & 0 
                \\ 0 & e^{- \theta_\pm}
            \end{pmatrix}
            \\ \hline
            \rGP 
            & \begin{pmatrix}
                0 & - \frac{i \epsilon}{\sqrt{2}} \big(\frac{\lambda + 1}{\epsilon^2} + W_+\big)
                \\ - \frac{i \epsilon}{\sqrt{2}} \big(\frac{\lambda - 1}{\epsilon^2} + W_-\big) & 0
            \end{pmatrix}
            & \frac{i}{\sqrt{2}} \begin{pmatrix}
                \lambda - \sqrt{2} \epsilon k + 1 & - \lambda + \sqrt{2} \epsilon k - 1
                \\ - \lambda + \sqrt{2} \epsilon k + 1 & - \lambda + \sqrt{2} \epsilon k + 1
            \end{pmatrix} \begin{pmatrix}
                e^{\theta_\pm} & 0 
                \\ 0 & e^{- \theta_\pm}
            \end{pmatrix}
            \\ \hline
            \KdV
            & \begin{pmatrix}
                0 & 1
                \\ U - k^2 & 0
            \end{pmatrix}
            & \begin{pmatrix}
                1 & 1
                \\ - i k & i k
            \end{pmatrix}
            \\ \hline
        \end{array}
    \end{equation*}}
    Here $\theta_\pm = \frac{i}{2} (\phi_\pm - \frac{\pi}{2})$ where $e^{i \phi_\pm} = q_\pm$.
    We also define $A^{\Eqn,\pm}(z^\Eqn; q^\Eqn) = \lim_{x^\Eqn \rightarrow \pm \infty} A^{\Eqn}(x^\Eqn, z^\Eqn; q^\Eqn)$.
    Observe that $E^{\Eqn,\pm}$ is indeed a solution to \eqref{eqn:mod-jost-ode} at $x^\Eqn = \pm \infty$.
\end{definition}
Our choice of normalization for these eigenvalue problems yields the identity
\begin{align} \label{eqn:jost-identity}
    \Phi^{\GP,\pm}_j(x, z) &= M(x) \Phi^{\hGP,\pm}_j(x, z) = M\left(\frac{X}{\sqrt{2} \epsilon}\right) \Phi^{\rGP,\pm}_j(X, k) \,.
\end{align}

\subsection{Transmission coefficients and their asymptotic expansion}
We recall \cite[Definition 2.4]{LiaoWegner2025}.
\begin{definition} \label{def:expansion}
    Let $D \subset \C$, $d \in \N$ and $1 \leq p \leq \infty$.
    \begin{enumerate}[(i)]
        \item We say that a function $f = f(z) \in C^\infty(D; \C^d)$ has an 
        \textbf{asymptotic expansion in powers of $2 i z$ at infinity on $D$} if there exist $(f_n)_{n \in \N} \subset \C^d$ such that
        \begin{align*}
            \forall\, N \in \N \quad
            \lim_{|z| \rightarrow \infty} (2 i z)^N \left( f(z) - \sum_{n=0}^N \frac{f_n}{(2 i z)^n} \right) = 0 \,.
        \end{align*}
        We call $f_n$ the \textbf{expansion coefficients} of $f$ and note that they are unique.
        If $f$ and $g$ have such an expansion then $2 i z f$ and $f g$ do as well. 
        If $g(z) \neq 0$ for $|z|$ sufficiently large, then $f g^{-1}$ also has such an expansion.
        
        \item We say that a function $f = f(x, z) \in C^\infty(\R \times D; \C^d)$ has an 
        \textbf{$L^p$-smooth asymptotic expansion in powers of $2 i z$ at infinity on $D$} if there exist $(f_n)_{n \in \N} \subset (C^\infty \cap L^p)(\R; \C^d)$ such that
        \begin{align*}
            \forall\, k, N \in \N \quad
            \lim_{|z| \rightarrow \infty} \left\| (2 i z)^N \left( \del_x^k f(x, z) - \sum_{n=0}^N \frac{\del_x^k f_n(x)}{(2 i z)^n} \right) \right\|_{L^p_x(\R)} = 0 \,.
        \end{align*}
        In this case $\del_x f$ has such an expansion as well. 
        If $g$ has such an $L^\infty$-smooth expansion and $\inf_{x \in \R} |g(x, z)| > 0$ for $|z|$ sufficiently large, then $f g^{-1}$ also has an $L^p$-smooth asymptotic expansion.
        If $p = 1$ then $\int_{-\infty}^\infty f(x, z) \dd x$ has an asymptotic expansion in powers of $2 i z$ at infinity on $D$ with expansion coefficients $\int_{-\infty}^\infty f_n(x) \dd x$.
    \end{enumerate}
    When $f$ has an asymptotic expansion of some kind with expansion coefficients $(f_n)_{n \in \N}$, we write
    \begin{align*}
        f &\sim \sum_{n=0}^\infty \frac{f_n}{(2 i z)^n} \,.
    \end{align*}
\end{definition}
We collect now a number of results on the Jost solutions which are known in the literature, although rarely arranged in this form.
For the rest of the section we focus only on $\conj{\K_+}$.
\begin{lemma} \label{lem:jost}
    Let $q^\Eqn \in C^\infty(\R)$ with $q^\Eqn - q^\Eqn_\pm \in \mc{S}(\R_\pm)$.
    \begin{enumerate}[(i)]
        \item The Jost solutions $\Phi^{\Eqn,\pm}_j$ exist and are smooth in $x^\Eqn$ and analytic in $z^\Eqn$ up to a finite set of choices for $z^\Eqn$.
        \item There exists a constant $c(q^\Eqn) > 0$ such that $\Imag z^\Eqn > c(q^\Eqn)$ implies
        \begin{align}
            \label{eqn:ass-limit} &\text{the limits } \lim_{x^\Eqn \rightarrow \infty} \Psi^{\Eqn,-}_1(x^\Eqn, z^\Eqn), \lim_{x^\Eqn \rightarrow - \infty} \Psi^{\Eqn,+}_2(x^\Eqn, z^\Eqn) 
            \text{ exist}
            \\ &\lim_{x^\Eqn \rightarrow \infty} \del_{x^\Eqn} \Psi^{\Eqn,-}_1(x^\Eqn, z^\Eqn) = \lim_{x^\Eqn \rightarrow - \infty} \del_{x^\Eqn} \Psi^{\Eqn,+}_2(x^\Eqn, z^\Eqn) = 0 
            \\ \label{eqn:sup-jost-1} &\sup_{x^\Eqn \in \R} \left| \Psi^{\Eqn,-}_1(x^\Eqn, z^\Eqn) - E^{\Eqn,-}_1(z^\Eqn) \right| < \frac{|E^{\Eqn,-}_1(z^\Eqn)|}{c(q^\Eqn) |z^\Eqn|}
            \\ \label{eqn:sup-jost-2} &\sup_{x^\Eqn \in \R} \left| \Psi^{\Eqn,+}_2(x^\Eqn, z^\Eqn) - E^{\Eqn,+}_2(z^\Eqn) \right| < \frac{|E^{\Eqn,+}_2(z^\Eqn)|}{c(q^\Eqn) |z^\Eqn|} \,.
        \end{align}
        \item Define on $\conj{\K_+} \cap \{\Imag z^\Eqn > c, \pm \Imag \lambda > 0\}$ the quantities
        \begin{align}
            \label{eqn:sig-0} \sigma^\Eqn &= \frac{\del_{x^\Eqn} \Psi^{\Eqn,-}_{1,1}}{\Psi^{\Eqn,-}_{1,1}} 
            & \ti{\sigma}^\Eqn &= - \frac{\del_{x^\Eqn} \Psi^{\Eqn,+}_{2,2}}{\Psi^{\Eqn,+}_{2,2}} 
            \\ \label{eqn:r-0} r^\Eqn &= \frac{\Psi^{\Eqn,-}_{1,2}}{\Psi^{\Eqn,-}_{1,1}} 
            & \ti{r}^\Eqn &= \frac{\Psi^{\Eqn,+}_{2,1}}{\Psi^{\Eqn,+}_{2,2}} \,.
        \end{align}
        They are bounded and have $L^\infty$-smooth asymptotic expansions in powers of $2 i z^\Eqn$ at infinity.
        The densities $\sigma^\Eqn$ and $\ti{\sigma}^\Eqn$ are in addition integrable, and their asymptotic expansions are also $L^1$-smooth.
    \end{enumerate}
\end{lemma}
\begin{proof}
    Let us first discuss the case of \GP. 
    Here (i) and (ii) can be proven by the well-known method of writing the boundary value problem as a Volterra integral equation and using a Neumann series ansatz. 
    For a demonstration we refer to \cite[Proposition 3]{DemontisPrinariMeeVitale}, and also the theory in \cite[Chapter 1]{FaddeevTakhtajan}. 
    We refer also to \cite[Section 2]{LiaoWegner2025} for a treatment of these matters with a different approach, and specifically \cite[Lemma 2.5]{LiaoWegner2025} as a reference for (iii).
    These same methods are strictly easier to apply for \NLS and \KdV because the potential under consideration has zero boundary data and the structure of the scattering problem is similar.
    For more information on the direct scattering theory of \KdV we refer to \cite{AblowitzClarkson1991}.
    Using \eqref{eqn:jost-identity}, the cases \hGP and \rGP can be obtained directly from \GP.
\end{proof}
If $\Imag z^\Eqn = 0$ then $\Phi^{\Eqn,-} = (\Phi^{\Eqn,-}_1, \Phi^{\Eqn,-}_2)$ and $\Phi^{\Eqn,+} = (\Phi^{\Eqn,+}_1, \Phi^{\Eqn,+}_2)$ are two fundamental solution matrices to \eqref{eqn:jost-ode}, 
hence there exist $a^\Eqn(z^\Eqn; q^\Eqn), b^\Eqn(z^\Eqn; q^\Eqn) \in \C$ such that
\begin{align}
    \label{eqn:trans-1} \Phi^{\Eqn,-}_1(x^\Eqn, z^\Eqn) &= a^\Eqn(z^\Eqn; q^\Eqn) \Phi^{\Eqn,+}_1(x^\Eqn, z^\Eqn) + b^\Eqn(z^\Eqn; q^\Eqn) \Phi^{\Eqn,+}_2(x^\Eqn, z^\Eqn) \,.
\end{align}
We call $a^\Eqn(z^\Eqn; q^\Eqn)$ the transmission coefficient of $\Eqn$ with potential $q^\Eqn$. It satisfies
\begin{align} \label{eqn:trans-3}
    a^\Eqn &= \frac{\det(\Phi^{\Eqn,-}_1|\Phi^{\Eqn,+}_2)}{\det E^{\Eqn,+}}
    = \frac{\Phi^{\Eqn,-}_{1,1} \Phi^{\Eqn,+}_{2,2} - \Phi^{\Eqn,-}_{1,2} \Phi^{\Eqn,+}_{2,1}}{\Phi^{\Eqn,+}_{1,1} \Phi^{\Eqn,+}_{2,2} - \Phi^{\Eqn,+}_{1,2} \Phi^{\Eqn,+}_{2,1}} 
    = \frac{\Psi^{\Eqn,-}_{1,1} \Psi^{\Eqn,+}_{2,2} - \Psi^{\Eqn,-}_{1,2} \Psi^{\Eqn,+}_{2,1}}{\Psi^{\Eqn,+}_{1,1} \Psi^{\Eqn,+}_{2,2} - \Psi^{\Eqn,+}_{1,2} \Psi^{\Eqn,+}_{2,1}} \,.
\end{align}
Using this formula, we extend $a^\Eqn(z^\Eqn)$ analytically to $z^\Eqn \in \conj{\K_+}$ up to a finite set of points.
Since the determinant is multiplicative, and due to \eqref{eqn:jost-identity}, we have
\begin{align*}
    a^\GP(z; q) &= a^\hGP(z; w) = a^\rGP(k; W) \,.
\end{align*}
Note that
\begin{align*}
    \det E^{\Eqn, \pm} &= \begin{cases}
        1 &, \Eqn = \NLS
        \\ 2 z (\lambda - z) &, \Eqn = \GP
        \\ - 2 z (\lambda - z) &, \Eqn = \hGP
        \\ - 2 \sqrt{2} \epsilon k (\lambda - \sqrt{2} \epsilon k) &, \Eqn = \rGP
        \\ 2 i k &, \Eqn = \KdV
    \end{cases} \,.
\end{align*}
Using \eqref{eqn:mod-jost-ode} to substitute $\Phi^{\Eqn,-}_{1,2}$ in \eqref{eqn:trans-3} and taking the limit $x \rightarrow \infty$, or, respectively, by substituting $\Phi^+_{2,1}$ and taking the limit $x \rightarrow - \infty$, we find that
\begin{align*}
    \lim_{x^\Eqn \rightarrow \infty} \Psi^{\Eqn,-}_{1,1} &= a^\Eqn \beta^\Eqn_1
    & \lim_{x^\Eqn \rightarrow - \infty} \Psi^{\Eqn,+}_{2,2} &= a^\Eqn \beta^\Eqn_2
\end{align*}
where
\begin{align*}
    \beta^\Eqn_1 &= \frac{\det E^{\Eqn,+}}{E^{\Eqn,+}_{2,2} + (A^{\Eqn,+}_{2,2} + i z^\Eqn)^{-1} A^{\Eqn,+}_{1,2} E^{\Eqn,+}_{2,1}}
    = \begin{cases}
        1 &, \Eqn = \NLS
        \\ q_+ &, \Eqn = \GP
        \\ e^{\theta_+} \frac{i}{\sqrt{2}} (1 + \lambda - z) &, \Eqn = \hGP
        \\ e^{\theta_+} \frac{i}{\sqrt{2}} (1 + \lambda - \sqrt{2} \epsilon k) &, \Eqn = \rGP
        \\ 1 &, \Eqn = \KdV
    \end{cases}
    \\ \beta^\Eqn_2 &= \frac{\det E^{\Eqn,-}}{E^{\Eqn,-}_{1,1} + (A^{\Eqn,-}_{1,1} - i z^\Eqn)^{-1} A^{\Eqn,-}_{2,1} E^{\Eqn,-}_{1,2}}
    = \begin{cases}
        1 &, \Eqn = \NLS
        \\ \conj{q}_- &, \Eqn = \GP
        \\ e^{- \theta_-} \frac{i}{\sqrt{2}} (1 + z - \lambda) &, \Eqn = \hGP
        \\ e^{- \theta_-} \frac{i}{\sqrt{2}} (1 + \sqrt{2} \epsilon k - \lambda) &, \Eqn = \rGP
        \\ i k &, \Eqn = \KdV
    \end{cases} \,.
\end{align*}
On the other hand, \eqref{eqn:jost-bdry} implies
\begin{align*}
    \lim_{x^\Eqn \rightarrow - \infty} \Psi^{\Eqn,-}_{1,1} &= E^{\Eqn,-}_{1,1}
    = \begin{cases}
        1 &, \Eqn = \NLS
        \\ q_- &, \Eqn = \GP
        \\ e^{\theta_-} \frac{i}{\sqrt{2}} (1 + \lambda - z) &, \Eqn = \hGP
        \\ e^{\theta_-} \frac{i}{\sqrt{2}} (1 + \lambda - \sqrt{2} \epsilon k) &, \Eqn = \rGP
        \\ 1 &, \Eqn = \KdV \,.
    \end{cases}
    \\ \lim_{x^\Eqn \rightarrow \infty} \Psi^{\Eqn,+}_{2,2} &= E^{\Eqn,+}_{2,2}
    = \begin{cases}
        1 &, \Eqn = \NLS
        \\ \conj{q}_+ &, \Eqn = \GP
        \\ e^{- \theta_+} \frac{i}{\sqrt{2}} (1 + z - \lambda) &, \Eqn = \hGP
        \\ e^{- \theta_+} \frac{i}{\sqrt{2}} (1 + \sqrt{2} \epsilon k - \lambda) &, \Eqn = \rGP
        \\ i k &, \Eqn = \KdV \,.
    \end{cases}
\end{align*}
Due to \eqref{eqn:sup-jost-1}--\eqref{eqn:sup-jost-2}, there exists for sufficiently large $\Imag z > c = c(q)$ a branch $\log^\Eqn$ of the logarithm, only depending on $q^\Eqn_\pm$, 
for which $\log(\Psi^{\Eqn,-}_{1,1} (\beta^\Eqn_1)^{-1})$ and $\log(\Psi^{\Eqn,+}_{2,2} (\beta^\Eqn_2)^{-1})$ are continuously differentiable in $x^\Eqn$ on all of $\R$.
Specifically, writing $\log_{\text{pr}}$ for the principal logarithm, we choose
\begin{align*}
    \log^\Eqn(\omega) &= \begin{cases}
        \log_{\text{pr}}(\omega) &, \Eqn \in \{\NLS, \KdV\}
        \\ \log_{\text{pr}}(\omega e^{i (\phi_+ - \phi_-)}) - \log_{\text{pr}}(e^{i (\phi_+ - \phi_-)}) &, \Eqn = \GP
        \\ \log_{\text{pr}}(\omega e^{\frac{i}{2} (\phi_+ - \phi_-)}) - \log_{\text{pr}}(e^{\frac{i}{2} (\phi_+ - \phi_-)}) &, \Eqn \in \{\hGP, \rGP\}
    \end{cases} \qquad \qquad \forall\, \omega \in \C \,.
\end{align*}
Then 
\begin{align}
    \log^\Eqn a^\Eqn - \begin{cases}
        0 &, \Eqn = \NLS
        \\ i (\phi_- - \phi_+) &, \Eqn = \GP
        \\ \frac{i}{2} (\phi_- - \phi_+) &, \Eqn \in \{\hGP, \rGP\}
        \\ 0 &, \Eqn = \KdV
    \end{cases} &= \int_{\R} \sigma^\Eqn \dd x^\Eqn = \int_{\R} \ti{\sigma}^\Eqn \dd x^\Eqn \,.
\end{align}
We are interested in the asymptotic expansions
\begin{align}
    \label{eqn:a-sig-r-expand} \sigma^\Eqn(x^\Eqn, z^\Eqn; q^\Eqn) &\sim \sum_{n=0}^\infty \frac{\sigma^\Eqn_n(x^\Eqn; q^\Eqn)}{(2 i z^\Eqn)^n}
    & r^\Eqn(x^\Eqn, z^\Eqn; q^\Eqn) &\sim \sum_{n=0}^\infty \frac{r^\Eqn_n(x^\Eqn; q^\Eqn)}{(2 i z^\Eqn)^n} \,.
\end{align}
By (iii) from Lemma \ref{lem:jost} these expansions indeed exist, and the expansion coefficients are bounded and integrable in $x^\Eqn$.
Furthermore, the asymptotic expansion of $\sigma^\Eqn$ is $L^1$-smooth, meaning that we can take the integral and thereby obtain an asymptotic expansion in powers of $2 i z^\Eqn$ for the transmission coefficient of the form
\begin{align} \label{eqn:a-sig-sigt-expand}
    \log^\Eqn a^\Eqn - \begin{cases}
        0 &, \Eqn = \NLS
        \\ i (\phi_- - \phi_+) &, \Eqn = \GP
        \\ \frac{i}{2} (\phi_- - \phi_+) &, \Eqn \in \{\hGP, \rGP\}
        \\ 0 &, \Eqn = \KdV
    \end{cases} 
    &\sim \sum_{n=0}^\infty \frac{\int_{\R} \sigma^\Eqn_n \dd x^\Eqn}{(2 i z^\Eqn)^n} 
    \sim \sum_{n=0}^\infty \frac{\int_{\R} \ti{\sigma}^\Eqn_n \dd x^\Eqn}{(2 i z^\Eqn)^n} \,.
\end{align}
For the cases $\Eqn \in \{\GP, \hGP, \rGP\}$ clarification is necessary, as the expansion coefficients both in \eqref{eqn:a-sig-r-expand} and \eqref{eqn:a-sig-sigt-expand} are not the same on all of $\conj{\K_+}$.
Instead, they depend on the sign of $\Imag \lambda$, or alternatively a choice of bijection $\lambda \leftrightarrow (z, k)$.
We shall choose the expansion coefficients so that \eqref{eqn:a-sig-r-expand}--\eqref{eqn:a-sig-sigt-expand} hold true for $\Imag \lambda > 0$.

Using these expansion coefficients, we define for $\Eqn \neq \KdV$ the Hamiltonians
\begin{align} \label{eqn:def-hamil}
    \mc{H}^\Eqn_n(q^\Eqn) &= - (- i)^n \big((\sqrt{2} \epsilon)^{(n+1)}\big)^{\mathds{1}_{\{\Eqn = \rGP\}}} \int_{\R} \sigma^\Eqn_{n+1}(x^\Eqn; q^\Eqn) \dd x^\Eqn + \mathds{1}_{\{n=-1\}} \begin{cases}
        0 &, \Eqn = \NLS
        \\ \phi_- - \phi_+ &, \Eqn = \GP
        \\ \frac{\phi_- - \phi_+}{2} &, \Eqn \in \{\hGP, \rGP\}
    \end{cases} \,,
\end{align}
and for $\Eqn = \KdV$ the energies
\begin{align} \label{eqn:def-energy}
    \mc{E}^\KdV_n(U) &= (-1)^n \int_{\R} \sigma^\KdV_{2n+3}(X; U) \dd X \,.
\end{align}
We do not define conserved quantities corresponding to the even orders because they vanish (see Lemma \ref{lem:vanish} below).
These Hamiltonians act as expansion coefficients for the logarithm of the transmission coefficient. 
Specifically, for $(\lambda, z, k) \in \K_+ \cap \conj{\mc{Q}_1}$ we have as $\Imag z \rightarrow \infty$ the asymptotic expansions
\begin{align*}
    \log a^\Eqn(z^\Eqn; q^\Eqn) &\sim i \sum_{n=-1}^\infty \frac{\mc{H}^\Eqn_n(q^\Eqn)}{(2 z)^{n+1}} \qquad \qquad \forall\, \Eqn \in \{\NLS, \GP, \hGP, \rGP\}
    \\ \log a^\KdV(k; U) &\sim i \sum_{n=-1}^\infty \frac{\mc{E}^\KdV_n(U)}{(2 k)^{2n+3}} \,.
\end{align*}
In particular
\begin{align} \label{eqn:hamil-equal}
    \mc{H}^\GP_n(q) &= \mc{H}^\hGP_n(w) = \mc{H}^\rGP_n(W) \,.
\end{align}
Note that we can reformulate \eqref{eqn:GP-tGP} as
\begin{align*}
    \log a^\GP(z; q) &\sim i \sum_{n=0}^\infty \frac{\mc{H}^\tGP_{2n}(q)}{(2 z)^{2n+1}} + i \frac{\lambda}{z} \sum_{n=0}^\infty \frac{2^{\mathds{1}_{\{n=0\}}} \mc{H}^\tGP_{2n-1}(q)}{(2 z)^{2n}} \,,
\end{align*} 
which gives an alternative view on the definition of the \tGP hierarchy and its relation to the \GP hierarchy.

As mentioned above, for the cases $\Eqn \in \{\GP, \hGP, \rGP\}$ these expansion coefficients are specific to our choice $\Imag \lambda > 0$.
A crucial property of the Hamiltonians $\mc{H}^\Eqn_n$ is that they are pairwise Poisson commuting.
This is a well-known result and it implies that each Hamiltonian is a conserved quantity for every Hamiltonian flow.
\begin{lemma}
    Let $\Eqn \in \Eqns$ and $q^\Eqn \in C^\infty(\R)$ with $q^\Eqn - q^\Eqn_\pm \in \mc{S}(\R_\pm)$. Below, we write $\mc{H}^\KdV_n = \mc{E}^\KdV_n$.
    For all $n, m \geq 0$ and $z^\Eqn_1, z^\Eqn_2 \in \conj{\K_+}$ with $\Imag z^\Eqn_1, \Imag z^\Eqn_2 > c(q^\Eqn)$ we have
    \begin{align*}
        \{\mc{H}^\Eqn_n, \mc{H}^\Eqn_m\}^\Eqn(q^\Eqn) &= 0 
        & \{\log a^\Eqn(z^\Eqn_1), \log a^\Eqn(z^\Eqn_2)\}(q^\Eqn) &= 0 \,.
    \end{align*}
    The Poisson brackets $\{\dots, \dots\}^\Eqn$ are given in Appendix \ref{appendix:poissonstructures}.
\end{lemma}
\begin{proof}
    For \NLS, \GP and \KdV, which are all part of the \AKNS hierarchy, we refer to \cite[Theorem B.7]{KlausKochLiu2023}. 
    Specifically for \GP we note also a different proof given in \cite[III.§2]{FaddeevTakhtajan}.
    With \eqref{eqn:hamil-equal}, we can derive from the case \GP the cases \hGP and \rGP.
\end{proof}

\subsection{Recurrence relations for the densities of the Hamiltonians}

From \eqref{eqn:mod-jost-ode} we derive Riccati equations for $\sigma^\Eqn$ and $\ti{\sigma}^\Eqn$ with $\Eqn \in \{\NLS, \GP\}$:
\begin{align}
    \label{eqn:NLS-sig} \del_x \sigma^\NLS &= (2 i \lambda) \sigma^\NLS + \frac{q_x}{q} \sigma^\NLS - (\sigma^\NLS)^2 + q \conj{q}
    \\ \label{eqn:NLS-sigt} - \del_x \ti{\sigma}^\NLS &= (2 i \lambda) \ti{\sigma}^\NLS - \frac{\conj{q}_x}{\conj{q}} \ti{\sigma}^\NLS - (\ti{\sigma}^\NLS)^2 + q \conj{q}
    \\ \label{eqn:GP-sig} \del_x \sigma^\GP &= (2 i z) \sigma^\GP - (\sigma^\GP)^2 + q \conj{q} - 1 + q_x \frac{\sigma^\GP + i (\lambda - z)}{q}
    \\ \label{eqn:GP-sigt} - \del_x \ti{\sigma}^\GP &= (2 i z) \ti{\sigma}^\GP - (\ti{\sigma}^\GP)^2 + q \conj{q} - 1 - \conj{q}_x \frac{\ti{\sigma}^\GP + i (\lambda - z)}{\conj{q}}
\end{align}
They were first derived in \cite{ZakharovShabat}.
For $\Eqn \in \{\hGP, \rGP\}$ we prefer to work with the systems for the pairs of variables $(\sigma^\Eqn, r^\Eqn)$ and $(\ti{\sigma}^\Eqn, \ti{r}^\Eqn)$, which are specifically
\begin{align}
    \\ \label{eqn:hGP-sig} \del_x \sigma^\hGP &= (2 i z) \sigma^\hGP - (\sigma^\hGP)^2 - w_- w_+ - w_-(1 + \lambda) + w_+ (1 - \lambda) - i (w_+)_x r^\hGP
    \\ \label{eqn:hGP-sigt} - \del_x \ti{\sigma}^\hGP &= (2 i z) \ti{\sigma}^\hGP - (\ti{\sigma}^\hGP)^2 - w_- w_+ - w_-(1 + \lambda) + w_+ (1 - \lambda) - i (w_-)_x \ti{r}^\hGP
    \\ \label{eqn:hGP-r} \del_x r^\hGP &= i (1 + \lambda + w_+) (r^\hGP)^2 + i (1 - \lambda - w_-)
    \\ \label{eqn:hGP-rt} - \del_x \ti{r}^\hGP &= i (1 - \lambda - w_-) (\ti{r}^\hGP)^2 + i (1 + \lambda + w_+)
    \\ \label{eqn:rGP-sig} \del_X \sigma^\rGP\! &= (2 i k) \sigma^\rGP\! - (\sigma^\rGP)^2 + \frac12\left( - \epsilon^2 W_- W_+ - W_-(1 \!+\! \lambda) + W_+ (1 \!-\! \lambda) - i \sqrt{2} \epsilon (W_+)_X r^\rGP\! \right)
    \\ \label{eqn:rGP-sigt} - \del_X \ti{\sigma}^\rGP\! &= (2 i k) \ti{\sigma}^\rGP\! - (\ti{\sigma}^\rGP)^2 + \frac12\left( - \epsilon^2 W_- W_+ - W_-(1 \!+\! \lambda) + W_+ (1 \!-\! \lambda) - i \sqrt{2} \epsilon (W_-)_X \ti{r}^\rGP\! \right)
    \\ \label{eqn:rGP-r} \del_X r^\rGP &= \frac{i \epsilon}{\sqrt{2}} (\epsilon^{-2}(1 + \lambda) + W_+) (r^\rGP)^2 + \frac{i \epsilon}{\sqrt{2}} (\epsilon^{-2}(1 - \lambda) - W_-)
    \\ \label{eqn:rGP-rt} - \del_X \ti{r}^\rGP &= \frac{i \epsilon}{\sqrt{2}} (\epsilon^{-2}(1 - \lambda) - W_-) (\ti{r}^\rGP)^2 + \frac{i \epsilon}{\sqrt{2}} (\epsilon^{-2}(1 + \lambda) + W_+) \,.
\end{align}
The case \KdV is distinct in that we consider a single equation for $\sigma^\KdV$, but also the system $(\ti{\sigma}^\KdV, \ti{r}^\KdV)$:
\begin{align}
    \label{eqn:KdV-sig} \del_X \sigma^\KdV &= (2 i k) \sigma^\KdV - (\sigma^\KdV)^2 + U
    \\ \label{eqn:KdV-sigt} - \del_X \ti{\sigma}^\KdV &= (2 i k) \ti{\sigma}^\KdV - (\ti{\sigma}^\KdV)^2 + U + U_X \ti{r}^\KdV
    \\ \label{eqn:KdV-rt} - \del_X \ti{r}^\KdV &= \left( U + \frac14 (2 i k)^2 \right) (\ti{r}^\KdV)^2 - 1 \,.
\end{align}
These equations prescribe recurrence relations for the expansion coefficients $\sigma^\Eqn_n$.
In order to avoid an abuse of notation, we shall use bold variables $\bm{\sigma}^\Eqn$, $\bm{\ti{\sigma}}^\Eqn$, $\bm{r}^\Eqn$, $\bm{\ti{r}}^\Eqn$, and so on to denote the formal power series corresponding to each function.
By substituting the formal power series $\bm{\sigma}^\NLS, \ti{\bm{\sigma}}^\NLS, \bm{\sigma}^\GP$ and $\ti{\bm{\sigma}}^\GP$ into the Riccati equations \eqref{eqn:NLS-sig}--\eqref{eqn:GP-sigt}, we can derive recurrence relations for the expansion coefficients $\sigma^\NLS_n$, $\ti{\sigma}^\NLS_n$, $\sigma^\GP_n$ and $\ti{\sigma}^\GP_n$.
For the case $\NLS$ this yields
\begin{align*}
    \sigma^\NLS_0 &= 0
    & \sigma^\NLS_1 &= - q \conj{q}
    & \sigma^\NLS_{n+1} &= \del_x \sigma^\NLS_n - \frac{q_x}{q} \sigma^\NLS_n + \sum_{k=0}^n \sigma^\NLS_k \sigma^\NLS_{n-k} \,.
\end{align*}
Due to the presence of $\lambda$ on the right-hand side of \eqref{eqn:GP-sig}, the recurrence relation for $\sigma^\GP_n$ is awkward, so we use instead the relation
\begin{align*}
    \bm{\sigma}^\NLS(x, \lambda) &= \bm{\sigma}^\GP(x, z) + i (\lambda - z) \,,
\end{align*}
which is proven by using it to substitute $\bm{\sigma}^\NLS$ in \eqref{eqn:NLS-sig} and observing that this yields \eqref{eqn:GP-sig}.
We can then derive the expansion coefficients $\sigma^\GP_n$ from $\sigma^\NLS_n$ and vice versa. 
We assume now that $\lambda, z \in \conj{\mc{Q}_1}$.
Then we can write $z = \sqrt{\lambda^2 - 1}$ and $\lambda = \sqrt{z^2 + 1}$ using the principal square root, and obtain the relations
\begin{align}
    \label{eqn:NLS-GP-1} \sigma^\NLS_{2m} &= \sum_{k=0}^m \binom{m-1}{m-k} (-4)^{m-k} \sigma^\GP_{2k}
    & \sigma^\NLS_{2m+1} &= \sum_{k=0}^m \binom{m-\frac12}{m-k} (-4)^{m-k} \sigma^\GP_{2k+1} + (-1)^{m+1} C_m
    \\ \label{eqn:NLS-GP-2} \sigma^\GP_{2m} &= \sum_{k=0}^m \binom{m-1}{m-k} 4^{m-k} \sigma^\NLS_{2k}
    & \sigma^\GP_{2m+1} &= \sum_{k=0}^m \binom{m-\frac12}{m-k} 4^{m-k} \sigma^\NLS_{2k+1} + C_m \,.
\end{align}
Here $C_n$ are the Catalan numbers. They can be defined either by the recurrence relation they solve, or by their generating function:
\begin{align}
    \label{eqn:catalan} C_0 &= 1 & C_{n+1} &= \sum_{k=0}^n C_k C_{n-k} 
    & \sum_{n=0}^\infty X^n C_n &= \frac{1 - \sqrt{1 - 4 X}}{2 X} \,.
\end{align}
If $\lambda \not\in \conj{\mc{Q}_1}$ or $z \not\in \conj{\mc{Q}_1}$ then the relations \eqref{eqn:NLS-GP-1}--\eqref{eqn:NLS-GP-2} may need to be corrected by sign changes.

Defining $q^\sharp(x) = q(- x)$, we can observe that $\bm{\sigma}^\NLS(x, \lambda; q)$ and $\bm{\ti{\sigma}}^\NLS(- x, \lambda; \conj{q}^\sharp)$ both solve \eqref{eqn:NLS-sig}, and $\bm{\sigma}^\GP(x, z; q)$ and $\bm{\ti{\sigma}}^\GP(- x, z; \conj{q}^\sharp)$ both solve \eqref{eqn:GP-sig}. 
A consequence of the recurrence relations discussed above is that $\sigma^\NLS_n$ and $\sigma^\GP_n$ consist of monomials whose total number of derivatives has the opposite parity of $n$. 
Then $(\del_x^k q^\sharp)(x) = (-1)^k (\del_x^k q)(- x)$ together with \eqref{eqn:a-sig-sigt-expand} yields
\begin{align} \label{eqn:sig-sigt-relation}
    \int \sigma^\Eqn_n(x, z; q) \dd x &= (-1)^{n+1} \int \conj{\sigma^\Eqn_n(x, z; q)} \dd x
\end{align}
for $\Eqn \in \{\NLS, \GP\}$, implying that the functionals $\mc{H}^\NLS_n, \mc{H}^\GP_n, \mc{H}^\hGP_n$ and $\mc{H}^\rGP_n$ are real-valued.
We assume again that $\lambda, k \in \conj{\mc{Q}_1}$ are in the first quadrant of the complex plane in order to have an explicit expansion
\begin{align} \label{eqn:lam-expand}
    \lambda = \sqrt{1 + 2 \epsilon^2 k^2} &= \sum_{n=-1}^\infty \frac{c_n}{(2 i k)^n}
    \quad \text{ where } \quad
    c_n = \begin{cases}
        - \frac{i \epsilon}{\sqrt{2}} &, n = -1 \\
        \frac{i}{(\sqrt{2} \epsilon)^{n}} \frac{1}{n} \binom{n}{\frac{n-1}{2}} &, n \geq 0 \text{ odd} \\
        0 &, \text{else}
    \end{cases} \,.
\end{align}
Then we can derive from \eqref{eqn:rGP-sig}--\eqref{eqn:rGP-rt} the recurrence relations
\begin{align}
    \label{eqn:rGP-sig-rec-0} \sigma^\rGP_0 &= \frac{c_{-1}}{2} (W_+ + W_-)
    \\ \label{eqn:rGP-sig-rec-1} \sigma^\rGP_1 &= \frac{c_{-1}}{2} \del_X (W_+ + W_-) + \frac{c_{-1}^2}{4} (W_+ + W_-)^2 + \frac{\epsilon^2}{2} W_- W_+ + \frac12 (W_- - W_+) + \frac{i \epsilon}{\sqrt{2}} \del_X W_+ r^\rGP_0
    \\ \label{eqn:rGP-sig-rec-n+1} \sigma^\rGP_{n+1} &= \del_X \sigma^\rGP_n + \sum_{k=0}^n \sigma^\rGP_{n-k} \sigma^\rGP_k + \frac12 c_n (W_- + W_+) + \frac{i \epsilon}{\sqrt{2}} \del_X W_+ r^\rGP_n \,.
\end{align}
and
\begin{align}
    \label{eqn:rGP-r-rec-0} r^\rGP_0 &= 1
    \\ \label{eqn:rGP-r-rec-1} r^\rGP_1 &= - \frac{i \epsilon}{\sqrt{2}} \left( W_+ - W_- + \frac{2}{\epsilon^2} \right)
    \\ \label{eqn:rGP-r-rec-n+1} r^\rGP_{n+1} &= \del_X r^\rGP_n + \frac{i c_n}{\sqrt{2} \epsilon} + \frac{i \epsilon}{\sqrt{2}} \left( - W_+ - \frac{1}{\epsilon^2} \right) \sum_{k=0}^n r^\rGP_{n-k} r^\rGP_k 
    \\\nonumber &- \frac{i}{\sqrt{2} \epsilon} \sum_{m=0}^n \sum_{k=0}^m c_{n-m} r^\rGP_{m-k} r^\rGP_k - \frac{i c_{-1}}{\sqrt{2} \epsilon} \sum_{k=1}^n r^\rGP_{n+1-k} r^\rGP_k \,.
\end{align}
In the same way recurrence relations can be derived for $\ti{\sigma}^\rGP$, $\ti{r}^\rGP$, $\sigma^\hGP$, $\ti{\sigma}^\hGP$, $r^\hGP$, and $\ti{r}^\hGP$, 
which we omit here for brevity.

Lastly, we derive from \eqref{eqn:KdV-sig}--\eqref{eqn:KdV-rt} the recurrence relations
\begin{align} \label{eqn:KdV-sig-rec} 
    \sigma^\KdV_0 &= 0
    & \sigma^\KdV_1 &= - U
    & \sigma^\KdV_{n+1} &= \del_X \sigma^\KdV_n + \sum_{k=0}^n \sigma^\KdV_k \sigma^\KdV_{n-k} \,.
\end{align}
and
\begin{align*} 
    \ti{\sigma}^\KdV_0 &= 0
    & \ti{\sigma}^\KdV_1 &= - U
    & \ti{\sigma}^\KdV_{n+1} &= - \del_X \ti{\sigma}^\KdV_n + \sum_{k=0}^n \ti{\sigma}^\KdV_k \ti{\sigma}^\KdV_{n-k} - U_X \ti{r}^\KdV_n
\end{align*}
as well as
\begin{align*}
    \ti{r}^\KdV_0 &= 0
    & \ti{r}^\KdV_1 &= 2
    & \ti{r}^\KdV_{n+1} &= - \del_X \ti{r}^\KdV_n - U \sum_{k=0}^n \ti{r}^\KdV_k \ti{r}^\KdV_{n-k} - \frac14 \sum_{k=2}^n \ti{r}^\KdV_k \ti{r}^\KdV_{n+2-k} \,. 
\end{align*}

Note that something subtle has happened in \eqref{eqn:rGP-r-rec-0}: the equation \eqref{eqn:rGP-sig} only implies $(r^\rGP_0)^2 = 1$, so we need to use the additional information 
\begin{align*}
    r^\rGP_0 &= \lim_{k \rightarrow i \infty} \frac{E^{\rGP,-}_{1,2}(k)}{E^{\rGP,-}_{1,1}(k)} 
    = \lim_{k, \lambda \rightarrow i \infty} \frac{1 + \sqrt{2} \epsilon k - \lambda}{1 - \sqrt{2} \epsilon k + \lambda} 
    = 1 \,.
\end{align*}
On the other hand, we have $\ti{r}^\rGP_0 = - 1$. 
These statements are true only in the case $\Imag \lambda > 0$, which we have assumed by considering the asymptotic expansion for $\lambda, k \in \conj{\mc{Q}_1}$.
On other regions of $\conj{\K_+}$, where instead $\Imag \lambda < 0$ while of course $\Imag k > 0$, we have $r^\rGP_0 = - 1$, $\ti{r}^\rGP_0 = 1$.
In addition we must also replace $c_n$ by $- c_n$, hence the expansion coefficients change significantly.
This is relevant for the following lemma, which uses the operators $\Even_\lambda$ and $\Odd_\lambda$ defined in \eqref{eqn:def-even-odd}.
\begin{lemma} \label{lem:a-ev-od-expansion}
    For $\lambda, k \in \conj{\mc{Q}_1}$ we have the following asymptotic expansions in powers of $2 i k$ at infinity:
    \begin{align} \label{eqn:a-ev-od-expansion}
        \Even_\lambda[\log a^\rGP(k)] &\sim \sum_{n=0}^\infty \frac{\int_{\R} \Real \sigma^\rGP_{2n+1} \dd X}{(2 i k)^{2n+1}}
        & \Odd_\lambda[\log a^\rGP(k)] &\sim \sum_{n=0}^\infty \frac{\int_{\R} i \Imag \sigma^\rGP_{2n} \dd X}{(2 i k)^{2n}} \,.
    \end{align}
\end{lemma}
\begin{proof}
    Due to \eqref{eqn:sig-sigt-relation}, the real and imaginary parts on the right-hand side are superfluous.
    Define now $W_\pm^\sharp(X) = W_\pm(- X)$. An analysis of \eqref{eqn:rGP-sig}--\eqref{eqn:rGP-rt} reveals that
    \begin{align*}
        r^\rGP(X, (- \lambda, k); (W_+, W_-)) &= \ti{r}^\rGP(- X, (\lambda, k); (- W_-^\sharp, - W_+^\sharp))
        \\ \sigma^\rGP(X, (- \lambda, k); (W_+, W_-)) &= \ti{\sigma}^\rGP(- X, (\lambda, k); (- W_-^\sharp, - W_+^\sharp)) \,.
    \end{align*}
    Observe that $(W_+, W_-) \mapsto (- W_-^\sharp, - W_+^\sharp)$ corresponds to $q \mapsto q^\sharp$. 
    Then one can proceed as in the proof of \eqref{eqn:sig-sigt-relation} to obtain
    \begin{align*}
        \int_{\R} \ti{\sigma}^\rGP_n(- X, (\lambda, k); (- W_-^\sharp, - W_+^\sharp)) \dd X &= (-1)^{n+1} \int_{\R} \sigma^\rGP_n(X, (\lambda, k); (W_+, W_-)) \dd X \,,
    \end{align*}
    which yields the claimed expansions.
\end{proof}

\subsection{Proofs of Theorem \ref{thm:approx-2} and Corollary \ref{cor:approx-global}} \label{section:proof-of-thm-approx-2}
Let us emphasize that these recurrence relations are convenient for the explicit computation of the expansion coefficients, but difficult to work with in proofs by induction.
Instead, it is preferable to work with their generating functions, which solve the Riccati equations derived above, as operations on these generating functions elegantly represent complicated steps of a proof by induction.
As an example, consider the following lemma.
\begin{lemma} \label{lem:vanish}
    For all $U \in \mc{S}(\R; \R)$ and $n \geq 0$
    \begin{align*}
        \int_{\R} \sigma^\KdV_{2n}(X; U) \dd X &= 0 \,.
    \end{align*}
\end{lemma}
\begin{proof}
    We decompose $\bm{\sigma}^\KdV = E + O$ into the even and odd orders of the formal power series.
    Then \eqref{eqn:KdV-sig} implies
    \begin{align*}
        O_X &= (2 i k) E - 2 E O \,.
    \end{align*}
    In particular
    \begin{align*}
        E &= \frac{O_X}{2 i k - 2 O} = - \frac{(2 i k - 2 O)_X}{2 i k - 2 O} = - \frac12 (\log(2 i k - 2 O))_X \,.
    \end{align*}
    Since the right-hand side is a formal power series where each coefficient is a derivative of a polynomial in $U$ and its derivatives, its integral vanishes.
\end{proof}
This is essentially a proof by induction, but all the work of the induction step is hidden in the calculus of formal power series.
In order to transform it into an actual proof by induction one needs to compare the coefficients at each order, which yields recurrence relations. 
The calculations with generating functions then represent the inductive step for these recurrence relations. 
We generally skip the base case, which sometimes also falls out of the generating function calculations, but otherwise can always be checked manually.

Using the generating-function approach, we prove our approximation result for the densities of the Hamiltonians.
In the following one should understand $k \in i \R_+$ to be sufficiently large so that $\lambda \in i \R$.
\begin{lemma} \label{lem:approx-1}
    Define $W_\pm^\sharp(X) = W_\pm(-X)$. We assume that $k \in i \R$ for the definition of the real and imaginary part of formal power series in the variable $2 i k$. 
    Let $\bm{\lambda}$ be a purely imaginary formal power series in the variable $2 i k$ which satisfies $\bm{\lambda}^2 = 1 - \frac{\epsilon^2}{2} (2 i k)^2$.
    Then
    \begin{align} 
        \label{eqn:sig-approx-1} \Real[\bm{\sigma}^\rGP](X, k; W) - \frac{1}{i \bm{\lambda}} \Imag[\bm{\sigma}^\rGP](X, k; W) 
        &= \bm{\sigma}^\KdV(X, k; - W_-) + O(\epsilon^2)
        \\ \label{eqn:sig-approx-2} \Real[\bm{\sigma}^\rGP](X, k; W) + \frac{1}{i \bm{\lambda}} \Imag[\bm{\sigma}^\rGP](X, k; W) 
        &= \bm{\ti{\sigma}}^\KdV(- X, k; W_+^\sharp) + O(\epsilon^2)
        \\ \label{eqn:sig-approx-3} \Real[\bm{\ti{\sigma}}^\rGP](X, k; W) - \frac{1}{i \bm{\lambda}} \Imag[\bm{\ti{\sigma}}^\rGP](X, k; W) 
        &= \bm{\ti{\sigma}}^\KdV(X, k; - W_-) + O(\epsilon^2)
        \\ \label{eqn:sig-approx-4} \Real[\bm{\ti{\sigma}}^\rGP](X, k; W) + \frac{1}{i \bm{\lambda}} \Imag[\bm{\ti{\sigma}}^\rGP](X, k; W) 
        &= \bm{\sigma}^\KdV(- X, k; W_+^\sharp) + O(\epsilon^2)
    \end{align}
    and
    \begin{align*}
        \Real\left[\frac{i \epsilon}{\sqrt{2}} \bm{r}^\rGP\right](X, k; W) - \frac{1}{i \bm{\lambda}} \Imag\left[\frac{i \epsilon}{\sqrt{2}} \bm{r}^\rGP\right](X, k; W) 
        &= O(\epsilon^2)
        \\ \Real\left[\frac{i \epsilon}{\sqrt{2}} \bm{r}^\rGP\right](X, k; W) + \frac{1}{i \bm{\lambda}} \Imag\left[\frac{i \epsilon}{\sqrt{2}} \bm{r}^\rGP\right](X, k; W) 
        &= \bm{\ti{r}}^\KdV(- X, k; W_+^\sharp) + O(\epsilon^2)
        \\ \Real\left[\frac{i \epsilon}{\sqrt{2}} \bm{\ti{r}}^\rGP\right](X, k; W) - \frac{1}{i \bm{\lambda}} \Imag\left[\frac{i \epsilon}{\sqrt{2}} \bm{\ti{r}}^\rGP\right](X, k; W) 
        &= \bm{\ti{r}}^\KdV(X, k; - W_-) + O(\epsilon^2)
        \\ \Real\left[\frac{i \epsilon}{\sqrt{2}} \bm{\ti{r}}^\rGP\right](X, k; W) + \frac{1}{i \bm{\lambda}} \Imag\left[\frac{i \epsilon}{\sqrt{2}} \bm{\ti{r}}^\rGP\right](X, k; W) 
        &= O(\epsilon^2) \,.
    \end{align*}
\end{lemma}
\begin{proof}
    We define
    \begin{align*}
        \bm{\mu} &= \Real[\bm{\sigma}^\rGP] + \frac{1}{i \bm{\lambda}} \Imag[\bm{\sigma}^\rGP] 
        & \bm{\nu} &= \Real[\bm{\sigma}^\rGP] - \frac{1}{i \bm{\lambda}} \Imag[\bm{\sigma}^\rGP]
        \\ \bm{\ti{\mu}} &= \Real[\bm{\ti{\sigma}}^\rGP] + \frac{1}{i \bm{\lambda}} \Imag[\bm{\ti{\sigma}}^\rGP]                 
        & \bm{\ti{\nu}} &= \Real[\bm{\ti{\sigma}}^\rGP] - \frac{1}{i \bm{\lambda}} \Imag[\bm{\ti{\sigma}}^\rGP]
        \\ \bm{\alpha} &= \Real\left[\frac{i \epsilon}{\sqrt{2}} \bm{r}^\rGP\right] + \frac{1}{i \bm{\lambda}} \Imag\left[\frac{i \epsilon}{\sqrt{2}} \bm{r}^\rGP\right] 
        & \bm{\beta} &= \Real\left[\frac{i \epsilon}{\sqrt{2}} \bm{r}^\rGP\right] - \frac{1}{i \bm{\lambda}} \Imag\left[\frac{i \epsilon}{\sqrt{2}} \bm{r}^\rGP\right]
        \\ \bm{\ti{\alpha}} &= \Real\left[\frac{i \epsilon}{\sqrt{2}} \bm{\ti{r}}^\rGP\right] + \frac{1}{i \bm{\lambda}} \Imag\left[\frac{i \epsilon}{\sqrt{2}} \bm{\ti{r}}^\rGP\right] 
        & \bm{\ti{\beta}} &= \Real\left[\frac{i \epsilon}{\sqrt{2}} \bm{\ti{r}}^\rGP\right] - \frac{1}{i \bm{\lambda}} \Imag\left[\frac{i \epsilon}{\sqrt{2}} \bm{\ti{r}}^\rGP\right]
    \end{align*}
    From \eqref{eqn:rGP-r}--\eqref{eqn:rGP-rt} we derive the system
    \begin{align*}
        \\ \bm{\alpha}_X - W_+ \bm{\alpha}^2 + (2 i k)^2 \left( \frac{\epsilon^2}{2} W_+ \left(\frac{\bm{\alpha} - \bm{\beta}}{2}\right)^2 - \left(\frac{\bm{\alpha} + \bm{\beta}}{2}\right) \left(\frac{\bm{\alpha} - \bm{\beta}}{2}\right) \right) + 1 - \frac{\epsilon^2}{2} W_- &= 0
        \\ \bm{\beta}_X - (W_+ + 2 \epsilon^{-2}) \bm{\beta}^2 + (2 i k)^2 \left( \frac{\epsilon^2}{2} W_+ \left(\frac{\bm{\alpha} - \bm{\beta}}{2}\right)^2 - \bm{\beta} \left(\frac{\bm{\alpha} - \bm{\beta}}{2}\right) \right) - \frac{\epsilon^2}{2} W_- &= 0
        \\ - \bm{\ti{\alpha}}_X + (W_- - 2 \epsilon^{-2}) \bm{\ti{\alpha}}^2 - (2 i k)^2 \left( \frac{\epsilon^2}{2} W_- \left(\frac{\bm{\ti{\alpha}} - \bm{\ti{\beta}}}{2}\right)^2 - \bm{\ti{\alpha}} \left(\frac{\bm{\ti{\alpha}} - \bm{\ti{\beta}}}{2}\right) \right) + \frac{\epsilon^2}{2} W_+ &= 0
        \\ - \bm{\ti{\beta}}_X + W_- \bm{\ti{\alpha}}^2 - (2 i k)^2 \left( \frac{\epsilon^2}{2} W_- \left(\frac{\bm{\ti{\alpha}} - \bm{\ti{\beta}}}{2}\right)^2 - \left(\frac{\bm{\ti{\alpha}} + \bm{\ti{\beta}}}{2}\right) \left(\frac{\bm{\ti{\alpha}} - \bm{\ti{\beta}}}{2}\right) \right) + 1 + \frac{\epsilon^2}{2} W_+ &= 0 \,.
    \end{align*}
    This implies
    \begin{align} \label{eqn:ass-r}
        \bm{\alpha}, \bm{\ti{\beta}} &= O(\epsilon^0)
        & \bm{\ti{\alpha}}, \bm{\beta} &= O(\epsilon^2)
    \end{align}
    by induction.
    More precisely, we have
    \begin{align*}
        \bm{\alpha}_X - \left(W_+ + \frac14 (2 i k)^2\right) \bm{\alpha}^2 + 1 &= O(\epsilon^2)
        \\ - \bm{\ti{\beta}}_X - \left(- W_- + \frac14 (2 i k)^2\right) \bm{\ti{\beta}}^2 + 1 &= O(\epsilon^2) \,.
    \end{align*}
    As a result, the claimed approximations for $\bm{r}^\rGP$ and $\bm{\ti{r}}^\rGP$ hold.
    On the other hand, we obtain from \eqref{eqn:rGP-sig}--\eqref{eqn:rGP-sigt} the system
    \begin{align*}
        \bm{\mu}_X - (2 i k) \bm{\mu} + \bm{\mu}^2 - \frac{\epsilon^2}{2} (2 i k)^2 \left(\frac{\bm{\mu} - \bm{\nu}}{2}\right)^2 - W_+ + \epsilon^2\frac{W_- W_+}{2} + \del_X W_+ \bm{\alpha} &= 0
        \\ \bm{\nu}_X - (2 i k) \bm{\nu} + \bm{\nu}^2 - \frac{\epsilon^2}{2} (2 i k)^2 \left(\frac{\bm{\mu} - \bm{\nu}}{2}\right)^2 + W_- + \epsilon^2\frac{W_- W_+}{2} + \del_X W_+ \bm{\beta}  &= 0
        \\ - \bm{\ti{\mu}}_X - (2 i k) \bm{\ti{\mu}} + \bm{\ti{\mu}}^2 - \frac{\epsilon^2}{2} (2 i k)^2 \left(\frac{\bm{\ti{\mu}} - \bm{\ti{\nu}}}{2}\right)^2 - W_+ + \epsilon^2\frac{W_- W_+}{2} + \del_X W_- \bm{\ti{\alpha}} &= 0
        \\ - \bm{\ti{\nu}}_X - (2 i k) \bm{\ti{\nu}} + \bm{\ti{\nu}}^2 - \frac{\epsilon^2}{2} (2 i k)^2 \left(\frac{\bm{\ti{\mu}} - \bm{\ti{\nu}}}{2}\right)^2 + W_- + \epsilon^2\frac{W_- W_+}{2} + \del_X W_- \bm{\ti{\beta}}  &= 0 \,.
    \end{align*}
    Again with induction, we find that
    \begin{align*}
        \bm{\mu}_X - (2 i k) \bm{\mu} + \bm{\mu}^2 - W_+ + (W_+)_X \bm{\alpha} &= O(\epsilon^2)
        \\ \bm{\nu}_X - (2 i k) \bm{\nu} + \bm{\nu}^2 + W_- &= O(\epsilon^2)
        \\ - \bm{\ti{\mu}}_X - (2 i k) \bm{\ti{\mu}} + \bm{\ti{\mu}}^2 - W_+ &= O(\epsilon^2)
        \\ - \bm{\ti{\nu}}_X - (2 i k) \bm{\ti{\nu}} + \bm{\ti{\nu}}^2 + W_- + (W_-)_X \bm{\ti{\beta}} &= O(\epsilon^2) \,.
    \end{align*}
    This concludes the proof of the approximations for $\bm{\sigma}^\rGP$ and $\bm{\ti{\sigma}}^\rGP$.
\end{proof}

\begin{proof}[Proof of Theorem \ref{thm:approx-2}]
    We assume again a setting where $\lambda, k \in \conj{\mc{Q}_1}$.
    From the definition \eqref{eqn:def-hamil} and the fact that the Hamiltonians are real, we obtain
    \begin{align*}
        \mc{H}^\rGP_{2n}(W) &= (-1)^{n+1} (\sqrt{2} \epsilon)^{2n+1} \int_{\R} \Real \sigma^\rGP_{2n+1}(X; W) \dd X
        \\ \mc{H}^\rGP_{2n-1}(W) &= (-1)^n (\sqrt{2} \epsilon)^{2n} \int_{\R} \Imag \sigma^\rGP_{2n}(X; W) \dd X \,.
    \end{align*}
    Now combining \eqref{eqn:GP-tGP} with \eqref{eqn:def-energy} and Lemma \ref{lem:approx-1} yields \eqref{eqn:approx-both}.
    Note that $\sigma^\KdV_n(\pm W_\pm)$ is a polynomial in $U$ and its derivatives with an even number of derivatives present in each monomial.
    Therefore $\int_{\R} \sigma^\KdV_n(\mp X; \pm W_\pm^\sharp) \dd X = \int_{\R} \sigma^\KdV_n(X; \pm W_\pm) \dd X$.

    The main work is done, but we would like to gain a precise understanding of the number of derivatives present in $\bm{P}_n$ and prove \eqref{eqn:extra}.
    Below we prove results about $\mc{H}^\rGP$, which then directly imply the claims about $\mc{H}^\rtGP$.
    For $k \in \N$ and a polynomial $P$ in $W_+$, $W_-$, and their derivatives, we define $\pi_k P$ to be the sum of all monomials in $P$ whose total number of derivatives is $k$.
    In the following we identify such polynomials $P$ up to formal integration by parts, as the claims of Theorem \ref{thm:approx-2} concern only integrals $\int_{\R} P \dd X$.
    
    We set now $\Eqn = \rGP$ in order to simplify the notation.
    From \eqref{eqn:rGP-sig-rec-0}--\eqref{eqn:rGP-sig-rec-n+1} we derive
    \begin{align*}
        \Real \sigma^\Eqn_0 &= 0
        \\ \Real \sigma^\Eqn_1 &= - \frac{\epsilon^2}{8} (W_+ - W_-)^2 + \frac12 (W_- - W_+)
        \\ \Real \sigma^\Eqn_{n+1} &= \del_X \Real \sigma^\Eqn_n + \sum_{k=0}^n (\Real \sigma^\Eqn_{n-k} \Real \sigma^\Eqn_k - \Imag \sigma^\Eqn_{n-k} \Imag \sigma^\Eqn_k) - \frac{\epsilon}{\sqrt{2}} \del_X W_+ \Imag r^\Eqn_n
    \end{align*}
    and
    \begin{align*}
        \Imag \sigma^\Eqn_0 &= - \frac{\epsilon}{2 \sqrt{2}} (W_+ + W_-)
        \\ \Imag \sigma^\Eqn_1 &= \frac{\epsilon}{2 \sqrt{2}} \del_X (W_+ - W_-)
        \\ \Imag \sigma^\Eqn_{n+1} &= \del_X \Imag \sigma^\Eqn_n + \sum_{k=0}^n 2 \Real \sigma^\Eqn_{n-k} \Imag \sigma^\Eqn_k + \frac12 \Imag c_n (W_- + W_+) + \frac{\epsilon}{\sqrt{2}} \del_X W_+ \Real r^\Eqn_n
    \end{align*}
    as well as
    \begin{align*}
        \Imag r^\Eqn_0 &= 0
        \\ \Imag r^\Eqn_1 &= - \frac{\epsilon}{\sqrt{2}} \left( W_+ - W_- + \frac{2}{\epsilon^2} \right)
        \\ \Imag r^\Eqn_{n+1} &= \del_X \Imag r^\Eqn_n + \frac{\epsilon}{\sqrt{2}} \left( - W_+ - \frac{1}{\epsilon^2} \right) \sum_{k=0}^n (\Real r^\Eqn_{n-k} \Real r^\Eqn_k - \Imag r^\Eqn_{n-k} \Imag r^\Eqn_k) 
        \\ &+ \frac{1}{\sqrt{2} \epsilon} \sum_{m=0}^n \sum_{k=0}^m \Imag c_{n-m} 2 \Real r^\Eqn_{m-k} \Imag r^\Eqn_k + \frac{\Imag c_{-1}}{\sqrt{2} \epsilon} \sum_{k=1}^n 2 \Real r^\Eqn_{n+1-k} \Imag r^\Eqn_k \,.
    \end{align*}
    We do not write out the recurrence relation for $\Real r^\Eqn_n$ since it plays no significant role. 
    The first fact that we read from these recurrence relations is that $\Real \sigma^\Eqn_{2n+1}$ and $\Imag \sigma^\Eqn_{2n}$ are sums of monomials with an even number of derivatives, while all monomials in $\Real \sigma^\Eqn_{2n}$ and $\Imag \sigma^\Eqn_{2n+1}$ have an odd number of derivatives.
    Studying the recurrence relation for $\Imag r^\Eqn_n$ explicitly, we see that $\pi_k r^\Eqn_n = 0$ for all $k \geq n - 1$ and
    \begin{align*}
        \pi_0 \Imag r^\Eqn_1 &= - \frac{\epsilon}{\sqrt{2}} \left( W_+ - W_- + \frac{2}{\epsilon^2} \right)
        \\ \pi_n \Imag r^\Eqn_{n+1} &= \del_X \pi_{n-1} \Imag r^\Eqn_n \,,
    \end{align*}
    from which we deduce for all $n \geq 1$ that
    \begin{align*}
        \pi_{n-1} \Imag r^\Eqn_n &= - \frac{\epsilon}{\sqrt{2}} \del_X^{n-1} \left( W_+ - W_- + \frac{2}{\epsilon^2} \right) \,.
    \end{align*}
    Similarly, we find that $\pi_k \Imag \sigma^\Eqn_n = 0$ for all $k \geq n$ and
    \begin{align*}
        \pi_0 \Imag \sigma^\Eqn_0 &= - \frac{\epsilon}{2 \sqrt{2}} (W_+ + W_-)
        \\ \pi_1 \Imag \sigma^\Eqn_1 &= \frac{\epsilon}{2 \sqrt{2}} \del_X (W_+ - W_-)
        \\ \pi_{n+1} \Imag \sigma^\Eqn_{n+1} &= \pi_n \del_X \Imag \sigma^\Eqn_n \,,
    \end{align*}
    which implies
    \begin{align*}
        \pi_n \Imag \sigma^\Eqn_n &= \begin{cases}
            - \frac{\epsilon}{2 \sqrt{2}} (W_+ + W_-) &, n = 0
            \\ \frac{\epsilon}{2 \sqrt{2}} \del_X^n (W_+ - W_-) &, n \geq 1
        \end{cases} \,.
    \end{align*}
    Lastly, we have $\pi_k \Real \sigma^\Eqn_n = 0$ for $k \geq n$ and
    \begin{align*}
        \\ \pi_0 \Real \sigma^\Eqn_1 &= - \frac{\epsilon^2}{8} (W_+ - W_-)^2 + \frac12 (W_- - W_+)
        \\ \pi_n \Real \sigma^\Eqn_{n+1} &= \del_X \pi_{n-1} \Real \sigma^\Eqn_n - \sum_{k=0}^n (\pi_k \Imag \sigma^\Eqn_k) (\pi_{n-k} \Imag \sigma^\Eqn_{n-k}) - \frac{\epsilon}{\sqrt{2}} \del_X W_+ \pi_{n-1} \Imag r^\Eqn_n 
        \\ &= \del_X \pi_{n-1} \Real \sigma^\Eqn_n 
        - \sum_{k=1}^{n-1} \frac{\epsilon}{2 \sqrt{2}} \del_X^k (W_+ - W_-) \frac{\epsilon}{2 \sqrt{2}} \del_X^{n-k} (W_+ - W_-)
        \\ &+ \frac{\epsilon}{\sqrt{2}} (W_+ + W_-) \frac{\epsilon}{2 \sqrt{2}} \del_X^n (W_+ - W_-) 
        + \frac{\epsilon}{\sqrt{2}} \del_X W_+ \frac{\epsilon}{\sqrt{2}} \del_X^{n-1} \left( W_+ - W_- + \frac{2}{\epsilon^2} \right) 
        \\ &= \del_X \pi_{n-1} \Real \sigma^\Eqn_n + \frac{\epsilon^2}{8} \Bigg( 
            - \sum_{k=1}^{n-1} \del_X^k (W_+ - W_-) \del_X^{n-k} (W_+ - W_-) 
        \\ &+ 2 (W_+ + W_-) \del_X^n (W_+ - W_-)
            + 4 \del_X W_+ \del_X^{n-1} (W_+ - W_-)
        \Bigg)
        \\ &+ \mathds{1}_{\{n=1\}} \del_X W_+ \,.
    \end{align*}
    This implies
    \begin{align*}
        \int_{\R} \pi_{2n} \Real \sigma^\Eqn_{2n+1} \dd X &= (-1)^{n+1} \frac{\epsilon^2}{8} \int_{\R} (\del_X^n (W_+ - W_-))^2 \dd X \,,
    \end{align*}
    from which we obtain \eqref{eqn:extra}.
    So far we have shown that $\Real \sigma^\Eqn_{2n+1}$ has up to $2 n$ derivatives present in each monomial. 
    Via repeated integration by parts, we can reduce the number of derivatives on the factor which has the most, and thereby ensure that each monomial in $\bm{P}_{2n}$ has no more than $n$ derivatives on any factor.
    For $\Imag \sigma^\Eqn_{2n}$ we observe that the terms $\pi_{2n} \Imag \sigma^\Eqn_{2n}$ with the most derivatives vanish under integration by parts, so we obtain at most $n - 1$ derivatives on any factor in $\bm{P}_{2n-1}$.
    As a result of these elaborations, we also know that $\bm{P}_n$ has at most $\floor{\frac{n}{2}}$ derivatives in each monomial.

    It remains to prove the absence of constant or linear monomials in $\bm{P}_n$ and $\bm{Q}_n$. This follows by direct computation for $n \in \{-1, 0\}$.
    For $n \geq 1$ we shall obtain it from the absence of constant or linear monomials in $\mc{H}^\rtGP_n$ and $\mc{E}^\KdV_{n-1}$, which for the energies $\mc{E}^\KdV_{n-1}$ follows directly from \eqref{eqn:KdV-sig-rec}.
    If $P$ is a polynomial as above, we denote by $\tau_0 P$ its constant part and by $\tau_1 P$ its linear part.
    First, we note that \eqref{eqn:rGP-sig-rec-0}--\eqref{eqn:rGP-sig-rec-n+1} directly imply $\tau_0 \sigma^\Eqn_n = 0$. 
    Since $\pi_k \tau_1 \sigma^\Eqn_n$ vanishes under integration by parts for all $k \geq 1$, it suffices to determine $\pi_0 \tau_1 \sigma^\Eqn_n$.
    Here \eqref{eqn:rGP-sig-rec-0}--\eqref{eqn:rGP-sig-rec-n+1} imply
    \begin{align*}
        \pi_0 \tau_1 \sigma^\Eqn_0 &= \frac{c_{-1}}{2} (W_+ + W_-)
        & \pi_0 \tau_1 \sigma^\Eqn_1 &= \frac12 (W_- - W_+)
        & \pi_0 \tau_1 \sigma^\Eqn_{n+1} &= \frac{c_n}{2} (W_+ + W_-) \,.
    \end{align*}
    Since $c_n = 0$ for even $n$ it remains to prove the claim for the odd Hamiltonians
    \begin{align*}
        \mc{H}^\rtGP_{2n-1}(W) &= \int_{\R} \sum_{m=0}^n \binom{-\frac12}{n-m} 4^{n-m} 
        (-1)^m (\sqrt{2} \epsilon)^{2m} \Imag \sigma^\rGP_{2m}(X; W) \dd X \,.
    \end{align*}
    Here we have
    \begin{align*}
        \pi_0 \tau_1 \sum_{m=0}^n \binom{-\frac12}{n-m} 4^{n-m} (-1)^m (\sqrt{2} \epsilon)^{2m} \Imag \sigma^\rGP_{2m}
        &= \sum_{m=0}^n \binom{-\frac12}{n-m} 4^{n-m} (-1)^m (\sqrt{2} \epsilon)^{2m} \frac12 \Imag c_{2m-1} (W_+ + W_-)
        \\ &= - \mathds{1}_{\{n=0\}} \frac{\epsilon}{2 \sqrt{2}} (W_+ + W_-) \,.
    \end{align*}

\end{proof}

\begin{proof}[Proof of Corollary \ref{cor:approx-global}]
    Define for $n \in \N$ the energies
    \begin{align*}
        \ti{E}^n(U) &= \sum_{l=0}^n \binom{n}{l} \mc{E}^\KdV_l(U) \,.
    \end{align*}
    According to \cite[Theorem 9.1]{KochTataru2018} there exists $\delta \in (0, 1)$ such that for every $n \in \N$ the smallness condition $\|U\|_{H^{-1}} \leq \delta$ implies
    \begin{align*}
        |\ti{E}^n(U) - \|U\|_{H^n}^2| &\leq c(n) \|U\|_{H^{-1}} \|U\|_{H^n}^2 \,.
    \end{align*}
    We calculate
    \begin{align*}
        \mc{E}^{\rtGP,\pm}_n(W) &= \sum_{l=0}^n \binom{n}{l} (\sqrt{2} \epsilon)^{-2l-3} \big(\mc{H}^\rtGP_{2l+2}(W) \mp 2 \mc{H}^\rtGP_{2l+1}(W)\big) 
        \\ &= \sum_{l=0}^n \binom{n}{l} \left( \mc{E}^\KdV_{l}(\pm W_\pm) + \frac{\epsilon^2}{2} \int_{\R} \bm{P}_{2l+2} \mp \bm{P}_{2l+1} \dd X \right) 
        \\ &= \sum_{l=0}^n \binom{n}{l} \left( \mc{E}^\KdV_{l}(\pm W_\pm) + \frac{\epsilon^2}{8} \|\del_X^{l+1} (W_+ - W_-)\|_{L^2}^2 + \frac{\epsilon^2}{2} \int_{\R} \bm{Q}_{2l+2} \mp \bm{P}_{2l+1} \dd X \right) 
        \\ &= \ti{E}^n(\pm W_\pm) + \frac{\epsilon^2}{8} \|\del_X (W_+ - W_-)\|_{H^n}^2 + \frac{\epsilon^2}{2} \int_{\R} \sum_{l=0}^n \binom{n}{l} (\bm{Q}_{2l+2} \mp \bm{P}_{2l+1}) \dd X
        \\ &= \|W_\pm\|_{H^n}^2 + \frac{\epsilon^2}{8} \|\del_X (W_+ - W_-)\|_{H^n}^2 
        \\ &+ \ti{E}^n(\pm W_\pm) - \|W_\pm\|_{H^n}^2 + \frac{\epsilon^2}{2} \int_{\R} \sum_{l=0}^n \binom{n}{l} (\bm{Q}_{2l+2} \mp \bm{P}_{2l+1}) \dd X
    \end{align*}
    An analysis of the recurrence relations \eqref{eqn:rGP-sig-rec-0}--\eqref{eqn:rGP-sig-rec-n+1}, \eqref{eqn:rGP-r-rec-0}--\eqref{eqn:rGP-r-rec-n+1} reveals that the polynomials $\bm{Q}_{2l+2}$ and $\bm{P}_{2l+1}$ for $0 \leq l \leq n$ have degree less than or equal to $n + 3$, and at most $n$ derivatives on a single factor.
    Furthermore, the total number of derivatives in any monomial is strictly less than $n + 1$, and we know from Theorem \ref{thm:approx-2} that they contain no constant or linear monomials. 
    For $n \geq 1$ this allows us to combine the Cauchy--Schwarz inequality and the Sobolev embedding $H^1 \xhookrightarrow{\quad} L^\infty$ to estimate
    \begin{align*}
        \left|\int_{\R} \sum_{l=0}^n (\bm{Q}_{2l+2} \mp \bm{P}_{2l+1}) \dd X\right| &\leq C \|W\|_{H^n}^2 (1 + \|W\|_{H^n})^{n+1} \,.
    \end{align*}
    By a similar analysis of $\mc{E}^\KdV_n$ we find in particular that $\mc{H}^\rtGP_n(W) < \infty$ if $W \in H^{\floor{\frac{n}{2}}}$.
\end{proof}

\section{\texorpdfstring{Approximation of solutions to the \rtGP hierarchy by solutions to the \KdV hierarchy}
{Approximation of solutions to the εhG̃P̃ hierarchy by solutions to the KdV hierarchy}}

\label{section:approximation}
\subsection{Proof of Proposition \ref{prop:1}}

Proposition \ref{prop:1} is an assembly of known results. 
Most importantly, \cite[Theorem 1.1]{KenigPilod} by C.~E.~Kenig and D.~Pilod states local well-posedness of \eqref{eqn:KdV-n}.
For any initial data, the local solution to the $n$-th flow \hyperref[eqn:KdV-n]{$(\KdV_n)$} can be globalized by the standard argument of combining a blow-up alternative with the existence of conserved energies that control the Sobolev norm of the solution.
Then the solution to the $N$-truncated \KdV hierarchy is constructed by using this well-posedness result to move for a time $t_0$ along the flow of $(\KdV_0)$, then for a time $t_1$ along the flow of $(\KdV_1)$, and so on.
Due to the commutativity of the flows, this is independent of the choice of ordering. 
The commutativity of the flows is a well-known result which can be proven, for example, by proceeding as in \cite[Section 4]{LiaoWegner2025}, using the weighted Sobolev spaces $H^s \cap H^{s',1}$.
Then by density the commutativity of the flows on $H^s$ follows.

It remains to use certain conserved quantities to prove the uniform bound \eqref{eqn:KdV-bound}.
While the energies $\mc{E}^\KdV_n$ are of course conserved, they do not have the desired coercivity property. 
A first result controlling the Sobolev norm of the solution by certain combinations of $\mc{E}^\KdV_n$ in the periodic case was given in \cite[Theorem 3.1]{Lax1975} by P.~D.~Lax.
We shall use the recent results \cite[Theorem 1.1]{KillipVisanZhang2018} by R.~Killip, M.~Vi\c san, and X.~Zhang, and \cite[Theorem 9.1]{KochTataru2018} by H.~Koch and D.~Tataru.
The former result is only stated for Schwartz solutions, but the continuity of the local flow maps on Sobolev spaces allows us to approximate.

\begin{proof}[Proof of Proposition \ref{prop:1}]
    Since the low regularity conserved quantities used in \cite[Theorem 1.1]{KillipVisanZhang2018} are conserved for every flow in the \KdV hierarchy, we know that
    \begin{align*}
        \sup_{\bm{T} \in \R^{N+1}} \|U(\bm{T})\|_{H^{-1}} &\leq C (1 + \|U_0\|_{H^{-1}}^2) \|U_0\|_{H^{-1}} \,,
    \end{align*}
    so it suffices to prove for every $T > 0$ that
    \begin{align*}
        \sup_{|\bm{T}| \leq T} \|U(\bm{T})\|_{H^{s'}} &\leq C(s') \sup_{|\bm{T}| \leq T} \big(1 + \|U(\bm{T})\|_{H^{-1}}^{2 s'}\big) \|U_0\|_{H^{s'}} \,.
    \end{align*}
    To show this we use \cite[Theorem 9.1]{KochTataru2018}, which provides a smallness constant $\delta \in (0, 1)$ and for every $s' > - 1$ a functional $\ti{E}^{s'}: H^{s'} \rightarrow \R$ such that $\|U\|_{H^{-1}} \leq \delta$ implies
    \begin{align*}
        |\ti{E}^{s'}(U) - \|U\|_{H^{s'}}^2| &\leq c(s') \|U\|_{H^{-1}} \|U\|_{H^{s'}}^2 \,.
    \end{align*}
    Crucially, the functionals $\ti{E}^{s'}$ are conserved along every flow of the \KdV hierarchy.
    In particular, if
    \begin{align*}
        \sup_{|\bm{T}| \leq T} \|U(\bm{T})\|_{H^{-1}} \leq \delta_0 = \min\left\{\delta, \frac{1}{2 c(s')}\right\} \,,
    \end{align*}
    then
    \begin{align*}
        \|U(\bm{T})\|_{H^{s'}}^2 &\leq \ti{E}^{s'}(U(\bm{T})) + c(s') \|U(\bm{T})\|_{H^{-1}} \|U(\bm{T})\|_{H^{s'}}^2 
        \leq \ti{E}^{s'}(U(\bm{T})) + \frac12 \|U(\bm{T})\|_{H^{s'}}^2 \,,
    \end{align*}
    which implies
    \begin{align*}
        \sup_{|\bm{T}| \leq T} \|U(\bm{T})\|_{H^{s'}}^2 &\leq 2 \ti{E}^{s'}(U_0)
        \leq 2 \big(\|U_0\|_{H^{s'}}^2 + c(s') \|U_0\|_{H^{-1}} \|U_0\|_{H^{s'}}^2\big)
        \leq 3 \|U_0\|_{H^{s'}}^2 \,.
    \end{align*}
    We therefore assume $\sup_{|\bm{T}| \leq T} \|U(\bm{T})\|_{H^{-1}} > \delta_0$ and use the scaling invariance of the \KdV hierarchy to regain the required smallness.
    Specifically, we define $u(\bm{T}, X) = \omega^2 U(\Omega \bm{T}, \omega X)$ where 
    \begin{align*}
        \omega &= \frac{\delta_0^2}{\sup_{|\bm{T}| \leq T} \|U(\bm{T})\|_{H^{-1}}^2} \in (0, 1)
        \text{ and } \Omega \in \R^{(N+1) \times (N+1)} \text{ with } \Omega_{n,m} = \mathds{1}_{\{n=m\}} \omega^{2n+1} \,.
    \end{align*}
    By the representation \eqref{eqn:KdV-structure} below, $u$ is still a solution to the \KdV hierarchy in the sense of this theorem.
    It satisfies
    \begin{align*}
        \sup_{|\Omega \bm{T}| \leq T} \|u(\bm{T})\|_{H^{-1}}^2 
        &= \sup_{|\Omega \bm{T}| \leq T} \int_{\R} \frac{\omega}{\frac{1}{\omega^2} + \xi^2} |\ha{U}(\Omega \bm{T}, \xi)|^2 \dd \xi 
        \leq \omega \sup_{|\bm{T}| \leq T} \|U(\bm{T})\|_{H^{-1}}^2 = \delta_0^2 \,.
    \end{align*}
    The above reasoning now implies
    \begin{align*}
        \sup_{|\Omega \bm{T}| \leq T} \|u(\bm{T})\|_{H^{s'}}^2 &\leq 3 \|u(\bm{0})\|_{H^{s'}}^2 \,,
    \end{align*}
    which after scaling yields
    \begin{align*}
        \sup_{|\bm{T}| \leq T} \|U(\bm{T})\|_{H^{s'}}^2 &\leq 3 \omega^{-2s'} \|U_0\|_{H^{s'}}^2 \,.
    \end{align*}
\end{proof}

\subsection{Proof of Proposition \ref{prop:2}}

\begin{lemma}[Continuity argument] \label{lem:cont}
    For every $C_0 > 0$ and $h \in \N$ there exists some $\delta_0, \ti{\epsilon}_0, C_1 > 0$ such that the following holds.
    Let $r^\ast \in (0, \infty) \cup \{\infty\}$ and consider two functions $g_\pm \in C([0, r^\ast); [0, \infty))$ that satisfy for some $\ti{\epsilon} \in (0, \ti{\epsilon}_0)$ and all $r \in [0, r^\ast)$
    \begin{align*}
        |g_\pm(r) - g_\pm(0)| &\leq C_0 (g_\pm(r) + g_\pm(0)) (g_+(r) + g_-(r))^{\frac12}
        \\ &+ \ti{\epsilon}^2 C_0 (g_+(r) + g_-(r) + g_+(0) + g_-(0)) (1 + g_+(r) + g_-(r) + g_+(0) + g_-(0))^{h+1} \,.
    \end{align*}
    Then $g_+(0) + g_-(0) < \delta_0$ implies 
    \begin{align*}
        g_\pm(r) &\leq \begin{cases}
            3 g_\pm(0) &, g_\pm(0) > 0
            \\ 3 g_\mp(0) &, g_\pm(0) = 0
        \end{cases} + \ti{\epsilon}^2 C_1 (g_+(0) + g_-(0))
    \end{align*}
    for all $r \in [0, r^\ast)$.
\end{lemma}
\begin{proof}
    Set $E = \ti{\epsilon}^2 C_1 (g_+(0) + g_-(0))$.
    We prove by contradiction, so we assume that the estimate fails before $r^\ast$.
    More specifically, we may assume without loss of generality that the estimate for $g_+$ becomes saturated at some time $r_+ \in (0, r^\ast)$, and holds for both $g_+$ and $g_-$ on the interval $(0, r_+)$.
    This means either $g_+(r_+) = 3 g_+(0) + E$ and $g_+(0) > 0$, or $g_+(r_+) = 3 g_-(0) + E$ and $g_+(0) = 0$,
    and furthermore for all $r \in (0, r_+)$ we have $g_\pm(r) \leq 3 g_\pm(0) + E$ if $g_\pm(0) > 0$, and $g_\pm(r) \leq 3 g_\mp(0) + E$ if $g_\pm(0) = 0$.
    In the case $g_+(0) > 0$ this implies
    \begin{align*}
        2 g_+(0) + E &\leq C_0 (8 g_+(0) + E) (6 g_+(0) + 6 g_-(0) + 2 E)^{\frac12}
        \\ &+ \ti{\epsilon}^2 C_0 (8 g_+(0) + 8 g_-(0) + 2 E) (1 + 8 g_+(0) + 8 g_-(0) + 2 E)^{h+1} \,,
    \end{align*}
    which leads to a contradiction if $\delta_0$ and $\ti{\epsilon}_0$ are sufficiently small and $g_+(0) + g_-(0) < \delta_0$.
    The next case is $g_+(0) = 0$ and $g_-(0) > 0$.
    Here we have
    \begin{align*}
        g_+(r) &\leq C_0 g_+(r) (g_+(r) + g_-(r))^{\frac12}
        \\ &+ \ti{\epsilon}^2 C_0 (g_+(r) + g_-(r) + g_-(0)) (1 + g_+(r) + g_-(r) + g_-(0))^{h+1} \,.
    \end{align*}
    and $\max\{g_+(r), g_-(r)\} \leq 3 g_-(0) + E$ for all $r \in (0, r_+)$.
    In particular
    \begin{align*}
        3 g_-(0) + E = g_+(r_+) = \lim_{r \nearrow r_+} g_+(r) 
        &\leq C_0 (3 g_-(0) + E) (6 g_-(0) + 2 E)^{\frac12}
        \\ &+ \ti{\epsilon}^2 C_0 (8 g_-(0) + E) (1 + 8 g_-(0) + E)^{h+1} \,,
    \end{align*}
    which again leads to a contradiction for sufficiently small $\delta_0$ and $\ti{\epsilon}_0$.
    The final case is $g_+(0) = g_-(0) = 0$, where
    \begin{align*}
        g_\pm(r) &\leq C_0 g_\pm(r) (g_+(r) + g_-(r))^{\frac12}
        \\ &+ \ti{\epsilon}^2 C_0 (g_+(r) + g_-(r)) (1 + g_+(r) + g_-(r))^{h+1} \,.
    \end{align*}
    If either $g_+$ or $g_-$ are ever positive on $[0, r^\ast)$, then they assume arbitrarily small values for some $r \in (0, r^\ast)$, which is not compatible with the estimate when $\delta_0$ and $\ti{\epsilon}_0$ are sufficiently small.
\end{proof}

We can now prove the proposition.
\begin{proof}[Proof of Proposition \ref{prop:2}]
    This proposition is \cite[Corollary 1.7]{LiaoWegner2025} with the addition of (iv), which we now prove.
    A lower bound for the distance to vacuum using the energy is given, for example, in \cite[Lemma 1]{BethuelGravejatSaut2008} and \cite[Lemma 1.2]{Wegner2023}.
    Then continuity and boundedness of $W$ in $H^{s-1}$ follow from the local bilipschitz property of the Madelung transform $q \mapsto \big(\big(1 + \frac{w_+ - w_-}{2}\big)^2, w_+ + w_-\big)$ established in \cite[Corollary 1.7, Lemma 2.4]{Wegner2023} together with $q \in C(\R^{N+1}; X^s)$.
    For a detailed proof of the fact that \eqref{eqn:rGP-n} is solved in the sense of distributions in the case $n = 2$, we refer to \cite[Section 3]{Wegner2023}. 
    The cases $n \geq 3$ are analogous. 

    In fact the solutions are strong in the sense that $T_n \mapsto W(\bm{c}^{\NLS \mapsto \rtGP} \bm{T}) \in C^1(H^{s-2\floor{\frac{n}{2}}-2})$ satisfies \eqref{eqn:rtGP-n} in that space.
    We obtain this by transferring the differentiability of the solutions from (iii) of Proposition \ref{prop:2} through the Madelung transform.
    Away from vacuum, this boils down to combining product estimates and the boundedness of certain smooth nonlinear mappings on Sobolev spaces, which can be performed as in \cite{Wegner2023}.
    Note that the highest order of derivative appearing in \eqref{eqn:rtGP-n} is $2\floor{\frac{n}{2}}+1$.

    It remains to prove \eqref{eqn:approx-global-explicit}, so we let $h \in \{1, \dots, s - 2\}$. Here we use Corollary \ref{cor:approx-global}.
    In order to obtain an estimate for $\|W_\pm\|_{H^h}$ from this bound we need smallness of $\|W_\pm\|_{H^{-1}}$.
    We shall obtain smallness for $\|W\|_{H^{h+1}}$ globally in time by slightly relaxing our scaling.
    Specifically, we define $\ti{\bm{T}}$, $\ti{X}$, and $\ti{W}$ in analogy to $\bm{T}$, $X$, and $W$, with $\epsilon$ replaced by 
    \begin{align*}
        \ti{\epsilon} = \epsilon \max\{1, \delta^{-\frac13} \|W(\bm{0})\|_{H^{h+1}}^{\frac23}\}
    \end{align*}
    for some $\delta = \delta(h) > 0$ to be chosen later.
    Then $\ti{W}$ has the same properties as described above, except $\epsilon$ is replaced by $\ti{\epsilon}$.
    Since
    \begin{align*}
        \ti{W}(\ti{\bm{T}}, \ti{X}) &= \frac{1}{\ti{\epsilon}^2} w(\bm{t}, x) 
        = \left(\frac{\epsilon}{\ti{\epsilon}}\right)^2 \frac{1}{\epsilon^2} w(\bm{t}, x) 
        = \left(\frac{\epsilon}{\ti{\epsilon}}\right)^2 W(\bm{T}, X)
        = \left(\frac{\epsilon}{\ti{\epsilon}}\right)^2 W\left(\bm{T}, \frac{\epsilon}{\ti{\epsilon}} \ti{X}\right)
    \end{align*}
    we have for some constant $C = C(h) > 0$
    \begin{align*} 
        C^{-1} \left(\frac{\epsilon}{\ti{\epsilon}}\right)^{\frac32+h} \|W(\bm{T})\|_{H^h} 
        &\leq \|\ti{W}(\ti{\bm{T}})\|_{H^h}
        \leq C \left(\frac{\epsilon}{\ti{\epsilon}}\right)^{\frac32} \|W(\bm{T})\|_{H^h} \,,
    \end{align*}
    in particular $\|\ti{W}(\bm{0})\|_{H^{h+1}}^2 < \delta$.
    By Corollary \ref{cor:approx-global} there exists a choice of $\delta$ and $C = C(h) > 0$ so that $\|\ti{W}\|_{H^{-1}}^2 < 4 \delta$ implies
    \begin{align*}
        \left| \mc{E}^{\rtGP,\pm}_h(\ti{W}) - \|\ti{W}_\pm\|_{H^h}^2 - \frac{\ti{\epsilon}^2}{8} \|\del_X (\ti{W}_+ - \ti{W}_-)\|_{H^h}^2 \right| 
        &\leq C \|\ti{W}_\pm\|_{H^h}^2 \|\ti{W}_\pm\|_{H^{-1}} 
        \\ &+ \ti{\epsilon}^2 C \|\ti{W}\|_{H^h}^2 (1 + \|\ti{W}\|_{H^h})^{h+1} \,.
    \end{align*}
    We use this to estimate
    \begin{align*}
        & \left| \|\ti{W}_\pm(\ti{\bm{T}})\|_{H^h}^2 + \frac{\ti{\epsilon}^2}{8} \|\del_X (\ti{W}_+ - \ti{W}_-)(\ti{\bm{T}})\|_{H^h}^2 - \|\ti{W}_\pm(\bm{0})\|_{H^h}^2 - \frac{\ti{\epsilon}^2}{8} \|\del_X (\ti{W}_+ - \ti{W}_-)(\bm{0})\|_{H^h}^2 \right|
        \\ &\leq C \left( \|\ti{W}_\pm(\ti{\bm{T}})\|_{H^h}^2 \|\ti{W}_\pm(\ti{\bm{T}})\|_{H^{-1}} + \|\ti{W}_\pm(\bm{0})\|_{H^h}^2 \|\ti{W}_\pm(\bm{0})\|_{H^{-1}} \right)
        \\ &+ \ti{\epsilon}^2 C \left( \|\ti{W}(\ti{\bm{T}})\|_{H^h}^2 (1 + \|\ti{W}(\ti{\bm{T}})\|_{H^h})^{h+1} + \|\ti{W}(\bm{0})\|_{H^h}^2 (1 + \|\ti{W}(\bm{0})\|_{H^h})^{h+1} \right)
    \end{align*}
    as long as $\|\ti{W}(\bm{0})\|_{H^{-1}}^2, \|\ti{W}(\ti{\bm{T}})\|_{H^{-1}}^2 < 4 \delta$.
    Define
    \begin{align*}
        g_\pm(r) &= \sup_{\{|\ti{\bm{T}}| \leq r\}} \left( \|\ti{W}_\pm(\ti{\bm{T}})\|_{H^h}^2 + \frac{\ti{\epsilon}^2}{8} \|\del_X (\ti{W}_+ - \ti{W}_-)(\ti{\bm{T}})\|_{H^h}^2 \right)
        \\ r^\ast &= \inf\{r > 0: g_+(r) + g_-(r) \geq 4 \delta\}
    \end{align*}
    and note that $g_\pm$ is continuous.

    By choosing $\epsilon < \epsilon_0$ for some appropriate $\epsilon_0 = \epsilon_0(\delta, \|W(\bm{0})\|_{H^{h+1}})$, we can ensure that $\ti{\epsilon} \in (\epsilon, 1)$.
    Then the above inequality implies for $r \in [0, r^\ast)$ that for some $C_0 = C_0(h)$ we have
    \begin{align*}
        |g_\pm(r) - g_\pm(0)| &\leq C_0 (g_\pm(r) + g_\pm(0)) (g_+(r) + g_-(r))^{\frac12}
        \\ &+ \ti{\epsilon}^2 C_0 (g_+(r) + g_-(r) + g_+(0) + g_-(0)) (1 + g_+(r) + g_-(r) + g_+(0) + g_-(0))^{h+1} \,.
    \end{align*}
    By Lemma \ref{lem:cont} there exists some $\delta_0 = \delta_0(h) > 0$, $\ti{\epsilon}_0 = \ti{\epsilon}_0(h) > 0$, and $C_1 = C_1(h) > 0$ such that $g_+(0) + g_-(0) < \delta_0$ and $\ti{\epsilon} < \ti{\epsilon}_0$ imply
    \begin{align*}
        g_\pm(r) &\leq \begin{cases}
            3 g_\pm(0) &, g_\pm(0) > 0
            \\ 3 g_\mp(0) &, g_\pm(0) = 0
        \end{cases} + \ti{\epsilon}^2 C_1 (g_+(0) + g_-(0))
    \end{align*}
    for all $r \in [0, r^\ast)$.
    Since $g_\pm(0) \lesssim \|\ti{W}(\bm{0})\|_{H^{h+1}}^2$ we can obtain $g_+(0) + g_-(0) < \min\{\delta_0, \delta\}$ and also $\ti{\epsilon} < \ti{\epsilon}_0$ from an appropriately small choice of $\epsilon_0 = \epsilon_0(h, \delta, \delta_0, \|W(\bm{0})\|_{H^{h+1}})$.
    Then the estimate just proved implies $g_+(r) + g_-(r) < 4 \delta$, so $r^\ast = \infty$.
    
    If $g_\pm(0) = 0$ then $(\ti{W}_+ - \ti{W}_-)(\bm{0})$ is constant in $\ti{X}$, hence $W_+(\bm{0}) = W_-(\bm{0})$. 
    Therefore $g_\pm(0) = 0$ implies $W(\bm{0}) = 0$, which is a trivial case.
    If $g_\pm(0) > 0$ then we finally obtain \eqref{eqn:approx-global-explicit}.
\end{proof}

\subsection{Proof of Theorem \ref{thm:1}}

We are now ready to give the proof of our main result.
To simplify the notation, we set $\bm{c} = \bm{c}^{\NLS \mapsto \rtGP}$.
\begin{proof}[Proof of Theorem \ref{thm:1}]
    Let $\bm{T} \in \R^{N+1}$ and recall that $\bm{T} = (T_0, \dots, T_N)$. We shall use an energy inequality with Grönwall's lemma.
    Note, however, that we do not wish to apply Grönwall's lemma several times, i.~e. for each flow in the hierarchy, as this would require reusing solutions obtained by our well-posedness theory as initial data and lead to worse estimates.
    Instead, we shall apply Grönwall's lemma only once.

    \textbf{Continuous path from $\bm{0}$ to $\bm{T}$.}
    Let $n \in \{0, \dots, N\}$ and define $m = m(n) = \floor{\frac{n-1}{2}} \in \{-1, \dots, M\}$ where $M = \floor{\frac{N-1}{2}}$.
    In order to simplify the argument below we assume that $|T_0|, \dots, |T_N| > 0$. The degenerate cases only require minor modifications.
    Define $\tau_n = |T_0| + \dots + |T_n|$ as well as $\tau_{-1} = 0$ and set
    \begin{align*}
        (\bm{T}^\tau)_n &= \begin{cases}
            0 &, \tau \in [0, \tau_{n-1})
            \\ \frac{\tau - \tau_{n-1}}{\tau_n - \tau_{n-1}} T_n &, \tau \in [\tau_{n-1}, \tau_n)
            \\ T_n &, \tau \in [\tau_n, \infty)
        \end{cases} \,.
    \end{align*}
    Then $(\bm{T}^\tau)_{\tau \in [0, \infty)}$ is a continuous path with $\bm{T}^0 = \bm{0}$ and $(\bm{T}^\tau)_n = T_n$ if $\tau \geq \tau_n$.
    In particular $\bm{T}^{\tau_N} = \bm{T}$.
    We use Theorem \ref{thm:2} to write
    \begin{align*}
        W(\bm{c} \bm{T}^\tau) &= W(\bm{0}) + \int_0^{\tau_N \land \tau} \frac{\dd}{\dd {\tau'}} W(\bm{c} \bm{T}^{\tau'}) \dd \tau'
        \\ &= W(\bm{0}) + \sum_{n=0}^N \sign(T_n) \int_{\tau_{n-1} \land \tau}^{\tau_n \land \tau} \frac{\dd}{\dd (\bm{T}^{\tau'})_n} W(\bm{c} \bm{T}^{\tau'}) \dd \tau'
        \\ &= W(\bm{0}) + \sum_{n=0}^N \sign(T_n) \int_{\tau_{n-1} \land \tau}^{\tau_n \land \tau}
        \begin{pmatrix}
            - (-1)^n \frac12 \del_X \frac{\delta}{\delta W_+} \mc{E}^\KdV_m(W_+(\bm{c} \bm{T}^{\tau'}))
            \\ - \frac12 \del_X \frac{\delta}{\delta (- W_-)} \mc{E}^\KdV_m(- W_-(\bm{c} \bm{T}^{\tau'}))
        \end{pmatrix} + \epsilon^2 \bm{R}_n(\bm{T}^{\tau'}) \dd \tau' \,.
    \end{align*}
    Similarly, we can write
    \begin{align*}
        \begin{pmatrix}
            U_+(\bm{b}^+ \bm{T}^\tau)
            \\ U_-(\bm{b}^- \bm{T}^\tau)
        \end{pmatrix} &= U(\bm{0}) + \int_0^{\tau_N \land \tau} \frac{\dd}{\dd \tau'} 
        \begin{pmatrix}
            U_+(\bm{b}^+ \bm{T}^{\tau'})
            \\ U_-(\bm{b}^- \bm{T}^{\tau'})
        \end{pmatrix} \dd \tau'
        \\ &= U(\bm{0}) + \sum_{n=0}^N \sign(T_n) \int_{\tau_{n-1} \land \tau}^{\tau_n \land \tau} \frac{\dd}{\dd (\bm{T}^{\tau'})_n} 
        \begin{pmatrix}
            U_+(\bm{b}^+ \bm{T}^{\tau'})
            \\ U_-(\bm{b}^- \bm{T}^{\tau'})
        \end{pmatrix} \dd \tau'
        \\ &= U(\bm{0}) + \sum_{n=0}^N \sign(T_n) \int_{\tau_{n-1} \land \tau}^{\tau_n \land \tau} 
        \begin{pmatrix}
            - (-1)^n \frac12 \del_X \frac{\delta}{\delta U_+} \mc{E}^\KdV_m(U_+(\bm{b}^+ \bm{T}^{\tau'}))
            \\ - \frac12 \del_X \frac{\delta}{\delta (- U_-)} \mc{E}^\KdV_m(- U_-(\bm{b}^- \bm{T}^{\tau'}))
        \end{pmatrix} \dd \tau' \,.
    \end{align*}

    \textbf{Energy estimate.}
    Define
    \begin{align*}
        V(\bm{T}) &= \begin{pmatrix}
            V_+(\bm{T})
            \\ V_-(\bm{T})
        \end{pmatrix} = \begin{pmatrix}
            U_+(\bm{b}^+ \bm{T})
            \\ U_-(\bm{b}^- \bm{T})
        \end{pmatrix}
        - W(\bm{c} \bm{T}) \,.
    \end{align*}
    Since $W(\bm{0}) = U(\bm{0})$ and $\bm{R}_0 = \bm{R}_1 = 0$, we know that $V(\bm{T}^\tau) = 0$ for $\tau \in [0, \tau_1]$.
    We therefore have
    \begin{align*}
        V(\bm{T}^\tau) &= \sum_{n=2}^N \sign(T_n) \int_{\tau_{n-1} \land \tau}^{\tau_n \land \tau} (\bm{F}_n(\tau') - \epsilon^2 \bm{R}_n(\bm{T}^{\tau'})) \dd \tau' \,.
    \end{align*}
    with
    \begin{align*}
        \bm{F}_n(\tau) &= \begin{pmatrix}
            - (-1)^n \frac12 \del_X \frac{\delta}{\delta W_+} \big(\mc{E}^\KdV_m(W_+(\bm{c} \bm{T}^\tau) + V_+(\bm{T}^\tau)) - \mc{E}^\KdV_m(W_+(\bm{c} \bm{T}^\tau))\big)
            \\ - \frac12 \del_X \frac{\delta}{\delta (- W_-)} \big(\mc{E}^\KdV_m(- W_-(\bm{c} \bm{T}^\tau) - V_-(\bm{T}^\tau)) - \mc{E}^\KdV_m(- W_-(\bm{c} \bm{T}^\tau))\big)
        \end{pmatrix} \text{ for } \tau \in [\tau_{n-1}, \tau_n) \,.
    \end{align*}
    In particular, for all $h \in \{0, \dots, \floor{\frac{s-N}{2}}\}$ we have
    \begin{align*}
        \int_{\R} (\del_X^h V(\bm{T}^\tau))^2 \dd X &= \sum_{n=2}^N \sign(T_n) \int_{\tau_{n-1} \land \tau}^{\tau_n \land \tau} \int_{\R} 2 \del_X^h V(\bm{T}^{\tau'}) \del_X^h (\bm{F}_n(\tau') - \epsilon^2 \bm{R}_n(\bm{T}^{\tau'})) \dd X \dd \tau' \,.
    \end{align*}
    For notational convenience, let us define for a set of functions $F$ in the variable $X$ and parameters $K, L \in \N$ the sets
    \begin{align*}
        \mc{O}^{K,L}(F) &= \left\{ \sum_{k=0}^K \underset{|l| \leq L}{\sum_{l = (l_1, \dots, l_k) \in \N^k}} \sum_{f = (f_1, \dots, f_k) \in F^k} c^{k,l}_f \prod_{j=1}^k \del_X^{l_j} f_j : c^{k,l}_f \in \C[\epsilon] \right\} \,.
    \end{align*}
    Studying the recurrence relations \eqref{eqn:rGP-sig-rec-0}--\eqref{eqn:rGP-sig-rec-n+1}, \eqref{eqn:rGP-r-rec-0}--\eqref{eqn:rGP-r-rec-n+1}, as we do in the proof of Theorem \ref{thm:approx-2} at the end of Section \ref{section:transmissioncoefficient}, reveals that 
    $\frac{\delta}{\delta W_\pm} \mc{H}^{\rtGP}_n(W) \in \mc{O}^{n+1,2\floor{\frac{n}{2}}}(W_+, W_-)$. 
    A similar analysis of \eqref{eqn:KdV-sig-rec} yields $\frac{\delta}{\delta U} \mc{E}^\KdV_m(U) \in \mc{O}^{m+1,2m}(U)$.
    Then
    \begin{align*}
        \bm{R}_n &\in \mc{O}^{n+2,2\floor{\frac{n}{2}}+1}(W_+, W_-) + \frac{\mc{O}^{n+2,2\floor{\frac{n}{2}}+1}(W_+, W_-)}{1 + \frac{\epsilon^2}{2} (W_+ - W_-)} + \frac{\mc{O}^{n+2,2\floor{\frac{n}{2}}+1}(W_+, W_-)}{(1+ \frac{\epsilon^2}{2} (W_+ - W_-))^2} \,.
    \end{align*}
    We use Sobolev embedding $H^1 \xhookrightarrow{\quad} L^\infty$ and the $H^h$-algebra property for $h \geq 1$ as well as \cite[Theorem 2.87]{BCD} to find a constant $C(n, h, \delta) > 0$ such that
    \begin{align*}
        \|\bm{R}_n\|_{H^h} &\leq C(n, h, \delta) \left( 1 + \|W\|_{H^{n+1+h}}^{n+4} \right) \,.
    \end{align*}
    Then
    \begin{align*}
        \int_{\R} \del_X^h V \epsilon^2 \del_X^h \bm{R}_n \dd X
        &\leq \mathds{1}_{\{n \geq 2\}} \left( \|V\|_{H^h}^2 + \epsilon^4 C(N, h, \delta)^2 \left( 1 + \|W\|_{H^{n+1+h}}^{2n+8} \right) \right) \,.
    \end{align*}
    We are able to include the indicator function because $\bm{R}_0 = \bm{R}_1 = 0$.
    A more detailed analysis of \eqref{eqn:KdV-sig-rec} yields the representation
    \begin{align} \label{eqn:KdV-structure}
        \frac12 \del_X \frac{\delta}{\delta U} \mc{E}^\KdV_m 
        &= (-1)^m \del_X^{2m+1} U 
        + \sum_{k=2}^{m+1} \underset{|l| + 2 k = 2m+3}{\sum_{l = (l_1, \dots, l_k) \in \N^k}} c^{k,l} \prod_{j=1}^k \del_X^{l_j} U
    \end{align}
    with coefficients $c^{k,l} \in \C$.
    Then
    \begin{align*}
        & \del_X \frac{\delta}{\delta (\pm W_\pm)} \big( \mc{E}^\KdV_m(\pm W_\pm \pm V_\pm) - \mc{E}^\KdV_m(\pm W_\pm) \big)
        \\ &= \pm (-1)^m \del_X^{2m+1} V_\pm
        + \sum_{k=2}^{m+1} \underset{|l| \leq 2m-1}{\sum_{l = (l_1, \dots, l_k) \in \N^k}} c^{k,l} \sum_{i=1}^k \del_X^{l_i} (\pm V_\pm) \prod_{j=1}^{i-1} \del_X^{l_j} (\pm W_\pm) \prod_{j=i+1}^k \del_X^{l_j} (\pm W_\pm \pm V_\pm) \,.
    \end{align*}
    Using integration by parts and $\del_X^{m+h} V_\pm \in H^1(\R)$, we find that
    \begin{align*}
        \int_{\R} \del_X^h V_\pm \del_X^h \del_X^{2m+1} V_\pm \dd X
        &= \int_{\R} (-1)^m \del_X\left( \frac12 \big(\del_X^{m+h} V_\pm\big)^2 \right) \dd X = 0 \,.
    \end{align*}
    Similarly, using repeated integration by parts and the Sobolev embedding $H^1 \xhookrightarrow{\quad} L^\infty$, we estimate
    \begin{align*}
        & \left| \sum_{k=2}^{m+1} \underset{|l| \leq 2m-1}{\sum_{l = (l_1, \dots, l_k) \in \N^k}} c^{k,l} \sum_{i=1}^k \int_{\R} \del_X^h V_\pm \del_X^h\left( \del_X^{l_i} V_\pm \prod_{j=1}^{i-1} \del_X^{l_j} (\pm W_\pm) \prod_{j=i+1}^k \del_X^{l_j} (\pm U_\pm) \right) \dd X \right|
        \\ &\leq C^{m+1} \sum_{r=0}^h \sum_{k=2}^{m+1} \underset{|l| \leq 2m-1+2(h-r)}{\sum_{l = (l_1, \dots, l_k) \in \N^k}} \sum_{i=1}^k \int_{\R} (\del_X^r V_\pm)^2 \prod_{j=1}^{i-1} \del_X^{l_j} (\pm W_\pm) \prod_{j=i+1}^k \del_X^{l_j} (\pm U_\pm) \dd X
        \\ &\leq C^{m+1} \sum_{r=0}^h \int_{\R} (\del_X^r V_\pm)^2 \dd X \sum_{k=2}^{m+1} \underset{|l| \leq 2m-1+2(h-r)}{\sum_{l = (l_1, \dots, l_k) \in \N^k}} \sum_{i=1}^k \underset{j \neq i}{\prod_{j=1}^k} \big(\|\del_X^{l_j} W_\pm\|_{L^\infty} + \|\del_X^{l_j} U_\pm\|_{L^\infty}\big)
        \\ &\leq C^{m+1} \sum_{r=0}^h \int_{\R} (\del_X^r V_\pm)^2 \dd X \big( 1 + \|(W_\pm, U_\pm)\|_{H^{2(m+h-r)}}^m \big) \,.
    \end{align*}

    \textbf{Grönwall's lemma for linear systems.}
    In summary, for all $l \in \{0, \dots, h\}$ and $\tau \in [0, \tau_N]$ we have
    \begin{align*}
        \int_{\R} (\del_X^l V(\bm{T}^\tau))^2 \dd X 
        &\leq \mathds{1}_{\{\tau \geq \tau_1\}} C(N, h, \delta) \left( a_l + \int_0^\tau \sum_{r=0}^h H_{l,r} \int_{\R} (\del_X^r V(\bm{T}^{\tau'}))^2 \dd X \dd \tau' \right) \,,
    \end{align*}
    where we consider the real numbers
    \begin{align*}
        a_l &= (\tau_N - \tau_1) \epsilon^4 \big(1 + \|W\|_{L^\infty_{\bm{T}} H^{N+1+l}_X}^{2N+8}\big)
        \\ H_{l,r} &= \mathds{1}_{r \leq l} \big(1 + \|(W, U)\|_{L^\infty_{\bm{T}} H^{2(M+l-r)}_X}^M\big)
    \end{align*}
    as the entries of a vector $a \in \R^{h+1}$ and a matrix $H \in \R^{(h+1) \times (h+1)}$.
    Grönwall's lemma for linear systems (see \cite[Theorem]{ChandraDavis1976}) yields
    \begin{align*}
        \sup_{\tau \in [0, \tau_N]} \int_{\R} (\del_X^h V(\bm{T}^\tau))^2 \dd X 
        &\leq C \Bigg(\Bigg(1 + \int_0^{\tau_N - \tau_1} e^{C (\tau_N - \tau_1 - \tau) H} H \dd \tau \Bigg) a \Bigg)_h 
        \\ &= C (e^{C (\tau_N - \tau_1) H} a)_h
    \end{align*}
    for some constant $C = C(N, h, \delta) > 0$.
    It remains to estimate $(e^{C (\tau_N - \tau_1) H} a)_h$.
    We define some notation for the part $\ti{H}$ of $H$ which does not lie on the diagonal:
    \begin{align*}
        \ti{H} &= H - \Big( 1 + \|(W, U)\|_{L^\infty_{\bm{T}} H^{2M}_X}^M \Big)
        \\ \ti{H}_{l,r} &= \mathds{1}_{\{0 \leq r < l\}} g_{l-r}
        \quad \text{ where } \quad g_j = 1 + \|(W, U)\|_{L^\infty_{\bm{T}} H^{2(M+j)}_X}^M \,.
    \end{align*}
    Then since $\ti{H}$ and $H - \ti{H}$ commute, we have
    \begin{align*}
        |(e^{(\tau_N - \tau_1) H} a)_h| &\leq |(e^{(\tau_N - \tau_1) \ti{H}} a)_h| e^{(\tau_N - \tau_1) \big(1 + \|(W, U)\|_{L^\infty_{\bm{T}} H^{2M}_X}^M\big)} \,.
    \end{align*}
    Next, we estimate for $l \geq 1$
    \begin{align*}
        |(\ti{H}^l)_{h,r}| &= \sum_{h = j_0 > \dots > j_l = r} g_{j_0 - j_1} \dots g_{j_{l-1} - j_l} 
        \leq \binom{h-r-1}{l-1} g_{h-r}^l
    \end{align*}
    and then
    \begin{align*}
        (e^{(\tau_N - \tau_1) \ti{H}} a)_h &= \sum_{l=0}^\infty \frac{(\tau_N - \tau_1)^l}{l!} (\ti{H}^l a)_h
        = \sum_{l=0}^\infty \frac{(\tau_N - \tau_1)^l}{l!} \sum_{r=0}^h (\ti{H}^l)_{h,r} a_r
        \leq (h + 1) a_h (1 + (\tau_N - \tau_1) g_h)^h \,.
    \end{align*}
    This yields
    \begin{align*}
        |(e^{(\tau_N - \tau_1) H} a)_h| &\leq \epsilon^4 (\tau_N - \tau_1) C \big(1 + \|W\|_{L^\infty_{\bm{T}} H^{N+1+h}_X}^{2N+8}\big) 
        \big(1 + (\tau_N - \tau_1)\|(W, U)\|_{L^\infty_{\bm{T}} H^{2(M+h)}_X}^M\big)^h
        \\ &\quad \times \exp\left( C (\tau_N - \tau_1) \big(1 + \|(W, U)\|_{L^\infty_{\bm{T}} H^{2M}_X}^M\big) \right) \,.
    \end{align*}

\end{proof}

\appendix

\section{\texorpdfstring{Approximation of the perturbation determinant $\alpha^\rGP$ by $\alpha^\KdV$}
{Approximation of the perturbation determinant αᴳᴾ by αᴷᵈⱽ}}
\label{section:fredholmdeterminant}

\subsection{Definitions}
Let $n \in \{1, 2\}$.
We use the convention of identifying functions $f \in L^2(\R; \C^n)$ with the operator of multiplication by $f$.
We write $(L(L^2(\R; \C^n)), \|\cdot\|_{L^2 \rightarrow L^2})$ for the space of bounded linear operators on $L^2(\R; \C^n)$ equipped with the operator norm, 
and say that an operator $A \in L(L^2(\R; \C^n))$ is Hilbert-Schmidt if there exists a kernel $K \in L^2(\R \times \R; \C^{n \times n})$ such that
\begin{align*}
    A f(X) &= \int_{\R} K(X, Y) f(Y) \dd Y \,.
\end{align*}
In this case we define the Hilbert-Schmidt norm as
\begin{align*}
   \|A\|_{\HS}^2 &= \int_{\R^2} |K(X, Y)|^2 \dd X \dd Y \,.
\end{align*}
We say that an operator $A \in L(L^2(\R; \C^n))$ is trace-class if $A = A_1 A_2$ for two Hilbert-Schmidt operators $A_1$ and $A_2$.
Writing $K_1$ and $K_2$ for their kernels, we define the trace
\begin{align*}
    \tr[A] &= \tr[A_1 A_2] = \int_{\R} \int_{\R} \tr[K_1(X, Y) K_2(Y, X)] \dd X \dd Y \,.
\end{align*}
In particular
\begin{align*}
    \|A\|_{\HS}^2 &= \tr[A^\ast A] \,.
\end{align*}
If the trace $\tr[K]$ of the kernel $K$ of $A$ is continuous, then
\begin{align*}
    \tr[A] &= \int_{\R} \tr[K(X, X)] \dd X \,.
\end{align*}
The following lemma collects some results regarding the trace and the Hilbert-Schmidt norm.
\begin{lemma}[{\cite[Lemmas 1.4, 1.5]{KillipVisanZhang2018}}] \label{lem:20}
    Let $A, A_j: L^2(\R; \C^n) \rightarrow L^2(\R; \C^n)$, $1 \leq j \leq l$, $2 \leq l \in \N$ be linear operators with finite Hilbert-Schmidt norm.
    Then
    \begin{align}
        \label{eqn:HS-estimate-1} \|A\|_{L^2 \rightarrow L^2} &\leq \|A\|_{\HS}
        \\ \label{eqn:HS-estimate-2} |\tr[A_1 \dots A_l]| &\leq \|A_1\|_{\HS} \dots \|A_l\|_{\HS} \,.
    \end{align}
    If $A = A(s) \in C^1_b(\R; \HS)$ with $\sup_{s \in \R} \|A(s)\|_{\HS} < \frac13$,
    then there exists an open interval $I \ni 0$ on which the series
    \begin{align*}
        \alpha(s) &= \sum_{l=2}^\infty \frac{(-1)^l}{l} \tr[A(s)^l]
    \end{align*}
    is absolutely convergent and defines a function $\alpha \in C^1(I; \C)$ with derivative
    \begin{align} \label{eqn:alpha-diff}
        \del_s \alpha(s) &= \sum_{l=2}^\infty (-1)^l \tr[A(s)^{l-1} \del_s A(s)] \,.
    \end{align} 
    We have the bound
    \begin{align*}
        \sup_{s \in I} |\alpha(s)| &\leq \frac23 \sup_{s \in I} \|A(s)\|_{\HS}^2 \,.
    \end{align*}
\end{lemma}

Recall from \eqref{eqn:surface} the spectral parameters $(\lambda, z, k) \in \K \subset \C^3$ which satisfy $z^2 = (\sqrt{2} \epsilon k)^2 = \lambda^2 - 1$.
We assume for the rest of this section that $z, k \in i \R_+$ are large and choose $\lambda = \sqrt{1 + 2 \epsilon^2 k^2}$ as the principal square root. We define
\begin{align*}
    \kappa &= \frac{z}{\sqrt{2} \epsilon i} = \frac{k}{i} \in \R_+ & \kappa^2 &= \frac{1 - \lambda^2}{2 \epsilon^2} \,.
\end{align*}

For notational convenience, we use the $\kappa$-dependent Sobolev norms
\begin{align*}
    \|f\|_{H^{-1}_\kappa} &= \left( \int_{\R} \frac{1}{\xi^2 + 4 \kappa^2} |\ha{f}(\xi)|^2 \dd \xi \right)^{\frac12} \leq \min\left\{\frac{1}{\min\{1,2\kappa\}} \|f\|_{H^{-1}}, \frac{1}{2 \kappa} \|f\|_{L^2}\right\}
    \\ \|f\|_{L^2_\kappa} &= \left( \int_{\R} \frac{\xi^2 + 2 \kappa^2}{\xi^2 + 4 \kappa^2} |\ha{f}(\xi)|^2 \dd \xi \right)^{\frac12} \leq \|f\|_{L^2} \,.
\end{align*}
Here the Fourier transform is defined by
\begin{align*}
    \ha{f}(\xi) &= \frac{1}{\sqrt{2 \pi}} \int_{\R} e^{- i x \xi} f(x) \dd x \,.
\end{align*}

We assume $W = (W_+, W_-) \in L^2(\R; \R^2)$ or $W \in \mc{S}(\R; \R^2)$ depending on the context.
We are interested in the Lax operators
\begin{align*}
    L^\rGP(W) &= \begin{pmatrix}
        - \epsilon^2 W_- + 1 & i \sqrt{2} \epsilon \del_X \\ i \sqrt{2} \epsilon \del_X & - \epsilon^2 W_+ - 1
    \end{pmatrix}
    & L^\KdV(\pm W_\pm) &= - \del_X^2 \pm W_\pm
    \\ L^\rGP_0 &= \begin{pmatrix}
        1 & i \sqrt{2} \epsilon \del_X \\ i \sqrt{2} \epsilon \del_X & - 1
    \end{pmatrix}
    & L^\KdV_0 &= - \del_X^2
\end{align*}
and their resolvents
\begin{align*}
    R^\rGP(W) &= \left(\frac{L^\rGP(W) - \lambda}{\epsilon^2}\right)^{-1}
    & R^\KdV(\pm W_\pm) &= (L^\KdV(\pm W_\pm) + \kappa^2)^{-1}
    \\ R^\rGP_0 &= \left(\frac{L^\rGP_0 - \lambda}{\epsilon^2}\right)^{-1} 
    & R^\KdV_0 &= (L^\KdV_0 + \kappa^2)^{-1} \,.
\end{align*}

Define
\begin{align*}
    \mc{W} &= \begin{pmatrix}
        - W_- & 0 
        \\ 0 & - W_+
    \end{pmatrix}
\end{align*}
and note that $L^\rGP(W) = L^\rGP_0 + \epsilon^2 \mc{W}$ and $L^\KdV(\pm W_\pm) = L^\KdV_0 \pm W_\pm$.
In the following, we often drop the dependence of the Lax operators and resolvents on $W$ and $(\pm W_\pm)$ from our notation.
With
\begin{align*}
    \mc{U} &= \begin{pmatrix}
        \frac{\lambda + 1}{2} & \frac{i \epsilon}{\sqrt{2}} \del_X
        \\ \frac{i \epsilon}{\sqrt{2}} \del_X & \frac{\lambda - 1}{2}
    \end{pmatrix}
\end{align*}
we can write
\begin{align*}
    R^\rGP_0 &= R^\KdV_0 \mc{U} = \mc{U} R^\KdV_0 = \sqrt{R^\KdV_0} \mc{U} \sqrt{R^\KdV_0} \,.
\end{align*}
Next, we use the Green's functions
\begin{align*}
    G^\KdV_0 &= \frac{1}{\sqrt{2 \pi}} \mc{F}^{-1}\left[R^\KdV_0\big\vert_{\del_X = i \xi}\right]
    = \frac{1}{2 \kappa} e^{- \kappa |X|}
    \\ G^\rGP_0 &= \frac{1}{\sqrt{2 \pi}} \mc{F}^{-1}\left[R^\rGP_0\Big\vert_{\del_X = i \xi}\right]
    = \frac{1}{2 \kappa} e^{- \kappa |X|} \begin{pmatrix}
        \frac{\lambda + 1}{2} & - \frac{i \epsilon}{\sqrt{2}} \kappa \sign(X)
        \\ - \frac{i \epsilon}{\sqrt{2}} \kappa \sign(X) & \frac{\lambda - 1}{2}
    \end{pmatrix}
\end{align*}
to define the kernels
\begin{align*}
    K^\KdV_0(X, Y) &= G^\KdV_0(X - Y)
    & K^\rGP_0(X, Y) &= G^\rGP_0(X - Y) \,.
\end{align*}
Then $K^\KdV_0 \in L^\infty(\R \times \R)$ and $K^\rGP_0 \in (L^\infty(\R \times \R))^{2 \times 2}$ are the integral kernels of $R^\KdV_0$ and $R^\rGP_0$ respectively.
In order to give a formula for the Green's function $G^\rGP$, for which $K^\rGP(X, Y) = G^\rGP(X, Y)$ is the kernel of $R^\rGP$, we use the Jost solutions $\Phi^{\rGP,\pm}_j$, their boundary data $E^{\rGP,\pm}$, and the transmission coefficient $a^\rGP$.
Their definitions may be found in Section \ref{section:transmissioncoefficient}.
For simplicity, we suppress the dependence of these scattering quantities on $k$ from our notation in the following lemma.
\begin{lemma} \label{lem:greens}
    The Green's function $G^\rGP$ of the operator
    \begin{align*}
        \frac{L^{\rGP} - \lambda}{\epsilon^2} &= \begin{pmatrix}
            - W_- - \frac{\lambda - 1}{\epsilon^2} & \frac{i \sqrt{2}}{\epsilon} \del_X \\
            \frac{i \sqrt{2}}{\epsilon} \del_X & - W_+ - \frac{\lambda + 1}{\epsilon^2}
        \end{pmatrix} \,,
    \end{align*}
    defined by the property that $\left(\frac{L^{\rGP} - \lambda}{\epsilon^2}\right) G^\rGP(X, Y) = \delta_0(X - Y)$, is given by
    \begin{align*}
        G^\rGP(X, Y)
        &= \frac{\epsilon}{i \sqrt{2}} \frac{1}{a^\rGP(k) \det E^{\rGP,+}} \begin{cases}
            \begin{pmatrix}
                - \Phi^{\rGP,-}_{1,1}(X) \Phi^{\rGP,+}_{2,1}(Y) & \Phi^{\rGP,-}_{1,1}(X) \Phi^{\rGP,+}_{2,2}(Y)
                \\ - \Phi^{\rGP,-}_{1,2}(X) \Phi^{\rGP,+}_{2,1}(Y) & \Phi^{\rGP,-}_{1,2}(X) \Phi^{\rGP,+}_{2,2}(Y)
            \end{pmatrix} &, X > Y
            \\ \begin{pmatrix}
                - \Phi^{\rGP,+}_{2,1}(X) \Phi^{\rGP,-}_{1,1}(Y) & \Phi^{\rGP,+}_{2,1}(X) \Phi^{\rGP,-}_{1,2}(Y)
                \\ - \Phi^{\rGP,+}_{2,2}(X) \Phi^{\rGP,-}_{1,1}(Y) & \Phi^{\rGP,+}_{2,2}(X) \Phi^{\rGP,-}_{1,2}(Y)
            \end{pmatrix} &, X < Y
        \end{cases} \,.
    \end{align*}
\end{lemma}
\begin{proof}
    For every $Y \in \R$ and $X \in \R \setminus \{Y\}$ each column of $G^\rGP(X, Y)$ is a solution to the scattering problem $\left(\frac{L^{\rGP} - \lambda}{\epsilon^2}\right) G^\rGP(X, Y) = 0$. 
    Furthermore, the jump condition
    \begin{align*}
        & \lim_{X \searrow Y} G^\rGP(X, Y) - \lim_{X \nearrow Y} G^\rGP(X, Y) 
        \\ &= \frac{\epsilon}{i \sqrt{2}} \frac{1}{a^\rGP(k) \det E^{\rGP,+}} 
        \begin{pmatrix}
            \Phi^{\rGP,+}_{2,1} \Phi^{\rGP,-}_{1,1} - \Phi^{\rGP,-}_{1,1} \Phi^{\rGP,+}_{2,1}
            & \Phi^{\rGP,-}_{1,1} \Phi^{\rGP,+}_{2,2} - \Phi^{\rGP,+}_{2,1} \Phi^{\rGP,-}_{1,2}
            \\ \Phi^{\rGP,+}_{2,2} \Phi^{\rGP,-}_{1,1} - \Phi^{\rGP,-}_{1,2} \Phi^{\rGP,+}_{2,1}
            & \Phi^{\rGP,-}_{1,2} \Phi^{\rGP,+}_{2,2} - \Phi^{\rGP,+}_{2,2} \Phi^{\rGP,-}_{1,2}
        \end{pmatrix}(Y) 
        \\ &= \frac{\epsilon}{i \sqrt{2}} \begin{pmatrix}
            0 & 1 \\ 1 & 0
        \end{pmatrix}
    \end{align*}
    holds, which implies that $\left(\frac{L^{\rGP} - \lambda}{\epsilon^2}\right) G^\rGP(X, Y) = \delta_0(X - Y)$ in the sense of distributions.
\end{proof}
An analogous representation for the Green's function of $L^\KdV$ exists, but is not necessary for our purposes.

\subsection{Approximation of perturbation determinants}
In this section we develop the theory of perturbation determinants of the Lax operators $L^\rGP$ and $L^\KdV$, 
following in structure the seminal papers \cite{KillipVisanHarropGriffiths2024,KillipVisan,KillipVisanZhang2018} which concern $L^\KdV$ and $L^\NLS$.
These techniques have recently been used effectively to construct low-regularity conserved quantities and prove well-posedness of various completely integrable PDEs (see also \cite{BringmannKillipVisan2021,KillipLaurensVisan2024}).
Our contribution is a proof that an analogue of the approximation property observed in Theorem \ref{thm:approx-2} and Lemma \ref{lem:approx-1} also holds for the perturbation determinants.
This is the content of Lemma \ref{lem:approx-3} below.
The perturbation determinants in question are the quantities
\begin{align*}
    \log a^\rGP(k; W) &= \alpha^\rGP(\kappa; W) - \tr[R^\rGP_0 \mc{W}] = - \log \det\left( \left(\frac{L^\rGP_0 - \lambda}{\epsilon^2}\right)^{-1} \frac{L^\rGP(W) - \lambda}{\epsilon^2} \right)
    \\ \log a^\KdV(k; \pm W_\pm) &= \alpha^\KdV(\kappa; \pm W_\pm) - \tr[R^\KdV_0 (\pm W_\pm)] = - \log \det\left( \left(L^\KdV_0 + \kappa^2\right)^{-1} \left(L^\KdV(\pm W_\pm) + \kappa^2\right) \right) \,.
\end{align*}
Here the formulation on the right-hand side should only be understood in a formal sense.
This can be rigorously defined via the absolutely convergent series
\begin{align*}
    \alpha^\rGP(\kappa; W) - \tr[R^\rGP_0 \mc{W}] &= \sum_{l=1}^\infty \frac{(-1)^l}{l} \tr[(R^\rGP_0 \mc{W})^l]
    \\ \alpha^\KdV(\kappa; \pm W_\pm) - \tr[R^\KdV_0 (\pm W_\pm)] &= \sum_{l=1}^\infty \frac{(-1)^l}{l} \tr[(R^\KdV_0 (\pm W_\pm))^l] \,,
\end{align*}
which is the content of Lemma \ref{lem:alpha} below. 
Subsequently, in Lemma \ref{lem:a-alpha-correspondence}, we prove the correspondence between these series and the transmission coefficient that we have stated above. 
We shall use the smallness of certain Hilbert-Schmidt norms to conclude the absolute convergence of these series as geometric series.
\begin{lemma} \label{lem:HS-explicit}
    We explicitly compute the following Hilbert-Schmidt norms:
    \begin{align}
        \label{eqn:HS-1} \left\|\sqrt{R^\KdV_0} W_\pm \sqrt{R^\KdV_0}\right\|_{\HS}^2 &= \frac{1}{\kappa} \|W_\pm\|_{H^{-1}_\kappa}^2
        \\ \label{eqn:HS-1.5} \left\|\sqrt{R^\KdV_0} i \del_X W_\pm \sqrt{R^\KdV_0}\right\|_{\HS}^2 &= \frac{1}{2 \kappa} \|W_\pm\|_{L^2_\kappa}^2
        \\ \label{eqn:HS-2} \left\|\sqrt{R^\KdV_0} \mc{U} \mc{W} \sqrt{R^\KdV_0}\right\|_{\HS}^2 
        &= \left|\frac{\lambda + 1}{2}\right|^2 \frac{1}{\kappa} \|W_-\|_{H^{-1}_\kappa}^2
        + \left|\frac{\lambda - 1}{2}\right|^2 \frac{1}{\kappa} \|W_+\|_{H^{-1}_\kappa}^2
        \\ &+ \frac{\epsilon^2}{2} \frac{1}{2 \kappa} \|W_-\|_{L^2_\kappa}^2
        + \frac{\epsilon^2}{2} \frac{1}{2 \kappa} \|W_+\|_{L^2_\kappa}^2 \,.
    \end{align}
\end{lemma}
\begin{proof}
    For \eqref{eqn:HS-1}, we refer to \cite[Proposition 2.1]{KillipVisanZhang2018}.
    We shall reduce \eqref{eqn:HS-1.5} and \eqref{eqn:HS-2} to applications of \eqref{eqn:HS-1}. 
    Towards this, we decompose
    \begin{align*}
        \left\|\sqrt{R^\KdV_0} \mc{U} \mc{W} \sqrt{R^\KdV_0}\right\|_{\HS}^2 
        &= \tr\left[\sqrt{R^\KdV_0} \mc{W} \mc{U}^\ast \sqrt{R^\KdV_0} \sqrt{R^\KdV_0} \mc{U} \mc{W} \sqrt{R^\KdV_0}\right]
        \\ &= \tr\left[R^\KdV_0 \begin{pmatrix}
            W_- \frac{\conj{\lambda} + 1}{2} & W_- \frac{i \epsilon}{\sqrt{2}} \del_X
            \\ W_+\frac{i \epsilon}{\sqrt{2}} \del_X & W_+\frac{\conj{\lambda} - 1}{2}
        \end{pmatrix} R^\KdV_0 \begin{pmatrix}
            \frac{\lambda + 1}{2} W_- & \frac{i \epsilon}{\sqrt{2}} \del_X W_+
            \\ \frac{i \epsilon}{\sqrt{2}} \del_X W_- & \frac{\lambda - 1}{2} W_+
        \end{pmatrix} \right]
        \\ &= \tr\left[R^\KdV_0 W_- \frac{\conj{\lambda} + 1}{2} R^\KdV_0 \frac{\lambda + 1}{2} W_-\right]
        + \tr\left[R^\KdV_0 W_- \frac{i \epsilon}{\sqrt{2}} \del_X R^\KdV_0 \frac{i \epsilon}{\sqrt{2}} \del_X W_-\right]
        \\ &+ \tr\left[R^\KdV_0 W_+ \frac{i \epsilon}{\sqrt{2}} \del_X R^\KdV_0 \frac{i \epsilon}{\sqrt{2}} \del_X W_+\right]
        + \tr\left[R^\KdV_0 W_+ \frac{\conj{\lambda} - 1}{2} R^\KdV_0 \frac{\lambda - 1}{2} W_+\right] 
        \\ &= \left|\frac{\lambda + 1}{2}\right|^2 \left\|\sqrt{R^\KdV_0} W_- \sqrt{R^\KdV_0}\right\|_{\HS}^2
        + \left|\frac{\lambda - 1}{2}\right|^2 \left\|\sqrt{R^\KdV_0} W_+ \sqrt{R^\KdV_0}\right\|_{\HS}^2
        \\ &+ \frac{\epsilon^2}{2} \left\|\sqrt{R^\KdV_0} W_- i \del_X \sqrt{R^\KdV_0}\right\|_{\HS}^2
        + \frac{\epsilon^2}{2} \left\|\sqrt{R^\KdV_0} W_+ i \del_X \sqrt{R^\KdV_0}\right\|_{\HS}^2 \,.
    \end{align*}
    It remains to compute
    \begin{align*}
        \left\|\sqrt{R^\KdV_0} i \del_X W_\pm \sqrt{R^\KdV_0}\right\|_{\HS}^2 
        = \left\|\sqrt{R^\KdV_0} W_\pm i \del_X \sqrt{R^\KdV_0}\right\|_{\HS}^2 
        &= \tr[R^\KdV_0 W_\pm i \del_X R^\KdV_0 i \del_X W_\pm] \,,
    \end{align*}
    which can be written, using the kernels
    \begin{align*}
        R^\KdV_0 W_\pm i \del_X f(X) &= \int_{\R} \frac{- i}{2 \kappa} e^{- \kappa |X - Y|} (\kappa \sign(X - Y) W_\pm(Y) + \del_Y W_\pm(Y)) f(Y) \dd Y
        \\ R^\KdV_0 i \del_X W_\pm f(X) &= \int_{\R} \frac{- i}{2 \kappa} e^{- \kappa |X - Y|} \kappa \sign(X - Y) W_\pm(Y) f(Y) \dd Y \,,
    \end{align*}
    as
    \begin{align*}
        & \tr[R^\KdV_0 W_\pm i \del_X R^\KdV_0 i \del_X W_\pm] 
        \\ &= \int_{\R} \int_{\R} \frac{- 1}{4 \kappa^2} e^{- \kappa |X - Y|} (\kappa \sign(X - Y) W_\pm(Y) + \del_Y W_\pm(Y)) e^{- \kappa |Y - X|} \kappa \sign(Y - X) W_\pm(X) \dd X \dd Y
        \\ &= \int_{\R} \int_{\R} \frac{1}{4 \kappa^2} e^{- 2 \kappa |X - Y|} \kappa^2 W_\pm(X) W_\pm(Y) \dd X \dd Y
        \\ &+ \int_{\R} \int_{\R} \frac{1}{4 \kappa^2} e^{- 2 \kappa |X - Y|} \kappa \sign(X - Y) W_\pm(X) \del_Y W_\pm(Y) \dd X \dd Y
        \\ &= \int_{\R} \int_{\R} \frac{1}{4 \kappa^2} e^{- 2 \kappa |X - Y|} \kappa^2 W_\pm(X) W_\pm(Y) \dd X \dd Y
        \\ &- \int_{\R} \int_{\R} \frac{1}{4 \kappa^2} e^{- 2 \kappa |X - Y|} (2 \kappa^2 - 2 \kappa \delta_0(X - Y)) W_\pm(X) W_\pm(Y) \dd X \dd Y
        \\ &= - \int_{\R} \int_{\R} \frac{1}{4 \kappa^2} e^{- 2 \kappa |X - Y|} \kappa^2 W_\pm(X) W_\pm(Y) \dd X \dd Y
        + 2 \kappa \int_{\R} \frac{1}{4 \kappa^2} |W_\pm(X)|^2 \dd X \,.
    \end{align*}
    We already know from \eqref{eqn:HS-1} that
    \begin{align*}
        \int_{\R} \int_{\R} \frac{1}{4 \kappa^2} e^{- 2 \kappa |X - Y|} \kappa^2 W_\pm(X) W_\pm(Y) \dd X \dd Y
        &= \kappa \|W_\pm\|_{H^{-1}_\kappa}^2 \,,
    \end{align*}
    which finishes the proof.
\end{proof}
As a consequence of this lemma we have $\sqrt{R^\KdV_0} \mc{W} \sqrt{R^\KdV_0}, \sqrt{R^\KdV_0} \mc{U} \mc{W} \sqrt{R^\KdV_0} \in L(L^2(\R; \C^2))$.
In the following, we sometimes use the operators $R^\KdV_0 \mc{W}$ and $R^\rGP_0 \mc{W}$ instead for the sake of clarity in formal calculations, 
but always in a context where eventually a trace or Hilbert-Schmidt norm is taken, returning us to the well-behaved versions of these operators.
\begin{lemma} \label{lem:alpha}
    Let $\epsilon \in (0, 1)$. For all $0 < \kappa < \frac{1}{\sqrt{2} \epsilon}$ such that
    \begin{align*}
        \frac{1}{\kappa} \|W\|_{H^{-1}_\kappa}^2 &\leq \frac{1}{30}
        & \frac{1}{\kappa} \|W\|_{L^2_\kappa}^2 &\leq \frac{1}{15}
    \end{align*} 
    the series
    \begin{align*}
        \alpha^\KdV(\kappa; \pm W_\pm) &= \sum_{l=2}^\infty \frac{(-1)^l}{l} \tr[(R^\KdV_0 (\pm W_\pm))^l]
        \\ \alpha^\rGP(\kappa; W) &= \sum_{l=2}^\infty \frac{(-1)^l}{l} \tr[(R^\rGP_0 \mc{W})^l]
    \end{align*}
    are absolutely convergent and satisfy
    \begin{align*}
        \frac32 |\alpha^\KdV(\kappa; \pm W_\pm)| &\leq \frac{1}{\kappa} \|W_\pm\|_{H^{-1}_\kappa}^2
        \\ \frac32 |\alpha^\rGP(\kappa; W)| &\leq \left|\frac{\lambda + 1}{2}\right|^2 \frac{1}{\kappa} \|W_-\|_{H^{-1}_\kappa}^2
            + \left|\frac{\lambda - 1}{2}\right|^2 \frac{1}{\kappa} \|W_+\|_{H^{-1}_\kappa}^2
            + \frac{\epsilon^2}{2} \frac{1}{2 \kappa} \|W_-\|_{L^2_\kappa}^2
            + \frac{\epsilon^2}{2} \frac{1}{2 \kappa} \|W_+\|_{L^2_\kappa}^2 \,.
    \end{align*}
    Moreover, the resolvent
    \begin{align*}
        R^\rGP &= \left( \frac{L^\rGP - \lambda}{\epsilon^2} \right)^{-1} \in L(L^2(\R; \C^2))
    \end{align*} 
    is represented by the series
    \begin{align} \label{eqn:resolvent}
        R^\rGP &= \sum_{l=0}^\infty (-1)^l (R^\rGP_0 \mc{W})^l R^\rGP_0
        = \sum_{l=0}^\infty (-1)^l \sqrt{R^\KdV_0} \left(\sqrt{R^\KdV_0} \mc{U} \mc{W} \sqrt{R^\KdV_0}\right)^l \sqrt{R^\KdV_0} \mc{U} \,,
    \end{align}
    which is absolutely convergent with respect to the operator norm $\|\cdot\|_{L^2 \rightarrow L^2}$.
\end{lemma}
\begin{proof}
    This follows from Lemma \ref{lem:20} together with the Hilbert-Schmidt norms computed in Lemma \ref{lem:HS-explicit}.
\end{proof}
We have adapted the method of proof of the following lemma from \cite[Theorem 2.6]{Klaus2022}. 
We refer to the paragraph above this theorem for some context on the history of this argument, and other approaches to proving that the perturbation determinant equals the logarithm of the transmission coefficient.
\begin{lemma} \label{lem:a-alpha-correspondence}
    Assume the setting of Lemma \ref{lem:alpha}, and in addition that $W \in \mc{S}(\R; \R^2)$ so that we are in a setting where we have defined the transmission coefficients $a^\rGP(k; W)$ and $a^\KdV(k; \pm W_\pm)$ (see \eqref{eqn:trans-3}).
    We have
    \begin{align} \label{eqn:rGP-alpha-a}
        \alpha^\rGP(\kappa; W) - \log a^\rGP(k; W)
        &= \tr[R^\rGP_0 \mc{W}]
        = \frac{1}{2 \kappa} \int_{\R} \frac12 (W_+ - W_-) - \frac{\lambda}{2} (W_+ + W_-) \dd X
    \end{align}
    and similarly
    \begin{align} \label{eqn:KdV-alpha-a}
        \alpha^\KdV(\kappa; \pm W_\pm) - \log a^\KdV(k; \pm W_\pm)
        &= \tr[R^\KdV_0 (\pm W_\pm)]
        = \frac{1}{2 \kappa} \int_{\R} \pm W_\pm \dd X \,.
    \end{align}
\end{lemma}
\begin{proof}
    We compute
    \begin{align*}
        \tr[R^\rGP_0 \mc{W}] &= \tr\left[ R^\KdV_0 \begin{pmatrix}
        \frac{\lambda + 1}{2} & \frac{i \epsilon}{\sqrt{2}} \del_X
        \\ \frac{i \epsilon}{\sqrt{2}} \del_X & \frac{\lambda - 1}{2}
    \end{pmatrix} \mc{W} \right]
    \\ &= \tr\left[R^\KdV_0 \frac{\lambda+1}{2} (- W_-)\right]
    + \tr\left[R^\KdV_0 \frac{\lambda-1}{2} (- W_+)\right]
    \\ &= \frac{\lambda+1}{2} \int_{\R} \frac{1}{2 \kappa} e^{- \kappa |X - X|} (- W_-(X)) \dd X
    + \frac{\lambda-1}{2} \int_{\R} \frac{1}{2 \kappa} e^{- \kappa |X - X|} (- W_+(X)) \dd X
    \\ &= \frac{1}{2 \kappa} \int_{\R} \frac12 (W_+ - W_-) - \frac{\lambda}{2} (W_+ + W_-) \dd X 
    \end{align*}
    and
    \begin{align*}
        \tr[R^\KdV_0 (\pm W_\pm)] &= \frac{1}{2 \kappa} \int_{\R} \pm W_\pm \dd X \,.
    \end{align*}
    For \eqref{eqn:KdV-alpha-a} we refer to \cite[Proposition 2.4]{KillipVisan}, so it remains to prove \eqref{eqn:rGP-alpha-a}.

    By explicitly computing the Jost solution $\Psi^{\rGP,-}_{1,1}(X; 0)$, which is constant, we find that $\log a^\rGP(k; 0) = 0$, which matches $\alpha^\rGP(\kappa; 0) = 0$.
    It therefore suffices to prove for any $W, V \in \mc{S}(\R; \R^2)$ with
    \begin{align*}
        V &= (V_+, V_-) & \mc{V} &= \begin{pmatrix} - V_- & 0 \\ 0 & - V_+ \end{pmatrix}
    \end{align*}
    that the following limits exist and are identical:
    \begin{align*}
        \frac{\dd}{\dd s}\Big\vert_{s=0} \log a^\rGP(k; W + s V) &= \frac{\dd}{\dd s}\Big\vert_{s=0} (\alpha^\rGP(\kappa; W + s V) - \tr[R^\rGP_0 (\mc{W} + s \mc{V})]) \,.
    \end{align*}
    Set $\Phi^\rGP = (\Phi^{\rGP,-}_1|\Phi^{\rGP,+}_2)$.
    We start with the formal calculation
    \begin{align*}
        0 &= \frac{\dd}{\dd s} \Big\vert_{s=0} \left( \frac{L^\rGP(W + s V) - \lambda}{\epsilon^2} \Phi^{\rGP,\pm}_j(k; W + s V) \right)
        \\ &= \mc{V} \Phi^{\rGP,\pm}_j(k; W) + \frac{L^\rGP(W) - \lambda}{\epsilon^2} \frac{\dd}{\dd s}\Big\vert_{s=0} \Phi^{\rGP,\pm}_j(k; W + s V) \,,
    \end{align*}
    which implies
    \begin{align*}
        \frac{\dd}{\dd s}\Big\vert_{s=0} \Phi^\rGP(X, k; W + s V) 
        &= \int_{\R} G^\rGP(X, Y) (- \mc{V}(Y)) \Phi^\rGP(Y, k; W) \dd Y \,.
    \end{align*}
    For the derivative of the transmission coefficient this yields
    \begin{align*}
        & \frac{\dd}{\dd s}\Big\vert_{s=0} \log a^\rGP(k; W + s V) 
        \\ &= \frac{\dd}{\dd s}\Big\vert_{s=0} \log\left( \frac{\det \Phi^\rGP(X, k; W + s V)}{\det E^{\rGP,+}(k)} \right)
        \\ &= \tr\left[ \Phi^{\rGP}(X, k; W)^{-1} \int_{\R} G^\rGP(X, Y) (- \mc{V}(Y)) \Phi^\rGP(Y, k; W) \dd Y \right]
        \\ &= \int_{\R} \tr\left[\Phi^\rGP(Y, k; W) \Phi^{\rGP}(X, k; W)^{-1} G^\rGP(X, Y) (- \mc{V}(Y))\right] \dd Y
        \\ &= \int_{\R} \tr\left[ \left( \mathds{1}_{\{X < Y\}} \lim_{X' \nearrow Y} G^\rGP(X', Y) +  \mathds{1}_{\{X > Y\}} \lim_{X' \searrow Y} G^\rGP(X', Y) \right) (- \mc{V}(Y))\right] \dd Y
        \\ &= \int_{\R} G^\rGP_{1,1}(X, X) V_-(X) + G^\rGP_{2,2}(X, X) V_+(X) \dd X \,.
    \end{align*}
    Here we have used that $G^\rGP_{1,1}$ and $G^\rGP_{2,2}$ are continuous. 
    Since the Green's function is bounded and $V$ is Schwartz, the integrals are absolutely convergent and the derivative actually exists.
    On the other hand, we know from Lemma \ref{lem:20} that $\alpha^\rGP(\kappa; W + s V)$ is continuously differentiable.
    Specifically \eqref{eqn:alpha-diff} yields
    \begin{align*}
        \frac{\dd}{\dd s}\Big\vert_{s=0} \alpha^\rGP(\kappa; W + s V)
        &= \sum_{l=2}^\infty (-1)^l \tr[(R^\rGP_0 \mc{W})^{l-1} R^\rGP_0 \mc{V}]
        \\ &= \tr[(R^\rGP_0 - R^\rGP) \mc{V}]
        \\ &= \int_{\R} \tr[(G^\rGP - G^\rGP_0)(X, X) (- \mc{V})(X)] \dd X
        \\ &= \int_{\R} (G^\rGP - G^\rGP_0)_{1,1}(X, X) V_-(X) + (G^\rGP - G^\rGP_0)_{2,2}(X, X) V_+(X) \dd X \,.
    \end{align*}
    Here we have used that $G^\rGP$ and $G^\rGP_0$ have continuous diagonal parts.
    The proof is finished with the additional observation that
    \begin{align*}
        - \frac{\dd}{\dd s}\Big\vert_{s=0} \tr[R^\rGP_0 (\mc{W} + s \mc{V})] &= \int_{\R} (G^\rGP_0)_{1,1}(X, X) V_-(X) + (G^\rGP_0)_{2,2}(X, X) V_+(X) \dd X \,.
    \end{align*}
\end{proof}

Recall from \eqref{eqn:def-even-odd} the definition of the operators $\Even_\lambda$ and $\Odd_\lambda$.
In the following, we suppress the dependence of various quantities on $\lambda$ from the notation, but for the application of these operators this dependence is of course relevant.
\begin{lemma} \label{lem:approx-3}
    Let $W = (W_+, W_-) \in \mc{S}(\R; \R^2)$, $\epsilon \in (0, 1)$, and $0 < \kappa < \frac{1}{\sqrt{2} \epsilon}$ with $\lambda \in (0, 1)$.
    Assume that
    \begin{align*}
        \max\left\{\frac{1}{\kappa},\frac{1}{\kappa^3}\right\} \|W\|_{L^2}^2 &\leq \frac{1}{30} \,.
    \end{align*}
    Then the series $\alpha^\rGP(\kappa; W)$ and $\alpha^\KdV(\kappa; \pm W_\pm)$ converge.
    Furthermore, we have
    \begin{align} 
        \label{eqn:approx-l-1} \Even_\lambda\left[\tr[R^\rGP_0 \mc{W}]\right] \mp \frac{1}{\lambda} \Odd_\lambda \left[\tr[R^\rGP_0 \mc{W}]\right] &= \tr[R^\KdV_0 (\pm W_\pm)]
        \\ \label{eqn:approx-alpha} \Even_\lambda\left[\alpha^\rGP(\kappa; W)\right] \mp \frac{1}{\lambda} \Odd_\lambda \left[\alpha^\rGP(\kappa; W)\right] &= \alpha^\KdV(\kappa; (\pm W_\pm)) + \bm{E}(\kappa; W)
    \end{align}
    with
    \begin{align*}
        |\bm{E}(\kappa; W)| &\leq 8 \frac{\epsilon^2}{\kappa} \|W\|_{L^2}^2 \,.
    \end{align*}
\end{lemma}
\begin{proof}
    Combining \eqref{eqn:rGP-alpha-a} and \eqref{eqn:KdV-alpha-a} directly yields \eqref{eqn:approx-l-1}.
    To prove \eqref{eqn:approx-alpha} we need to understand
    \begin{align*}
        \Even_\lambda\left[\tr[(R^\rGP_0 \mc{W})^l]\right] \mp \frac{1}{\lambda} \Odd_\lambda\left[\tr[(R^\rGP_0 \mc{W})^l]\right] \,.
    \end{align*}
    Using the Pauli matrices
    \begin{align*}
        \sigma_3 &= \begin{pmatrix} 1 & 0 \\ 0 & - 1 \end{pmatrix}
        & \sigma_1 &= \begin{pmatrix} 0 & 1 \\ 1 & 0 \end{pmatrix} \,,
    \end{align*}
    we write
    \begin{align*}
        \Even_\lambda[R^\rGP_0 \mc{W}] &= \frac12 R^\KdV_0 (\sigma_3 + i \sqrt{2} \epsilon \sigma_1 \del_X) \mc{W}
        \\ \Odd_\lambda[R^\rGP_0 \mc{W}] &= \frac12 R^\KdV_0 \lambda \mc{W} \,.
    \end{align*}
    Then
    \begin{align*}
        \Even_\lambda\left[\tr[(R^\rGP_0 \mc{W})^l]\right] \mp \frac{1}{\lambda} \Odd_\lambda\left[\tr[(R^\rGP_0 \mc{W})^l]\right]
        &= \tr\left[\Even_\lambda[(R^\rGP_0 \mc{W})^l] \mp \frac{1}{\lambda} \Odd_\lambda[(R^\rGP_0 \mc{W})^l] \right] \,.
    \end{align*}
    For $l \in \N$ and $\omega \in \{0, 1, 2\}^l$ we define 
    \begin{align*}
        \conj{\omega} = \sum_{j=1}^l \mathds{1}_{\{\omega_j = 1\}} \,.
    \end{align*}
    One can show directly by induction that
    \begin{align*}
        \Even_\lambda[(R^\rGP_0 \mc{W})^l] &= \underset{\conj{\omega} \text{ even}}{\sum_{\omega \in \{0,1\}^l}} 2^{-l} \prod_{j=1}^l 
        \begin{cases}
            R^\KdV_0 (\sigma_3 + \sqrt{2} \epsilon \sigma_1 i \del_X) \mc{W} &, \omega_j = 0
            \\ R^\KdV_0 \lambda \mc{W} &, \omega_j = 1
        \end{cases} \\
        &= \underset{\conj{\omega} \text{ even}}{\sum_{\omega \in \{0,1,2\}^l}} 2^{-l} \prod_{j=1}^l 
        \begin{cases}
            R^\KdV_0 \sigma_3 \mc{W} &, \omega_j = 0
            \\ R^\KdV_0 \lambda \mc{W} &, \omega_j = 1
            \\ R^\KdV_0 \sqrt{2} \epsilon \sigma_1 i \del_X \mc{W} &, \omega_j = 2
        \end{cases}
    \end{align*}
    and
    \begin{align*}
        \Odd_\lambda[(R^\rGP_0 \mc{W})^l] &= \underset{\conj{\omega} \text{ odd}}{\sum_{\omega \in \{0,1\}^l}} 2^{-l} \prod_{j=1}^l 
        \begin{cases}
            R^\KdV_0 (\sigma_3 + \sqrt{2} \epsilon \sigma_1 i \del_X) \mc{W} &, \omega_j = 0
            \\ R^\KdV_0 \lambda \mc{W} &, \omega_j = 1
        \end{cases}
        \\ &= \underset{\conj{\omega} \text{ odd}}{\sum_{\omega \in \{0,1,2\}^l}} 2^{-l} \prod_{j=1}^l 
        \begin{cases}
            R^\KdV_0 \sigma_3 \mc{W} &, \omega_j = 0
            \\ R^\KdV_0 \lambda \mc{W} &, \omega_j = 1
            \\ R^\KdV_0 \sqrt{2} \epsilon \sigma_1 i \del_X \mc{W} &, \omega_j = 2
        \end{cases} \,.
    \end{align*}
    Here the ordering of the product of possibly non-commuting operators does not matter, as long as it is consistent for all $\omega \in \{0,1\}^l$.
    We shall fix as a convention that the first term in the product, i.~e. the term with $j = 1$, is multiplied on the left, while the last term is multiplied on the right.
    Then
    \begin{align*}
        \Even_\lambda[(R^\rGP_0 \mc{W})^l] \mp \frac{1}{\lambda} \Odd_\lambda[(R^\rGP_0 \mc{W})^l] &= \sum_{\omega \in \{0,1,2\}^l} 2^{-l} \left(\frac{\mp 1}{\lambda}\right)^{\mathds{1}_{\{\conj{\omega} \text{ odd}\}}} \prod_{j=1}^l 
        \begin{cases}
            R^\KdV_0 \sigma_3 \mc{W} &, \omega_j = 0
            \\ R^\KdV_0 \lambda \mc{W} &, \omega_j = 1
            \\ R^\KdV_0 \sqrt{2} \epsilon \sigma_1 i \del_X \mc{W} &, \omega_j = 2
        \end{cases} \\
        &= \underset{\exists j: \omega_j = 2}{\sum_{\omega \in \{0,1,2\}^l}} 2^{-l} \left(\frac{\mp 1}{\lambda}\right)^{\mathds{1}_{\{\conj{\omega} \text{ odd}\}}} \prod_{j=1}^l 
        \begin{cases}
            R^\KdV_0 \sigma_3 \mc{W} &, \omega_j = 0
            \\ R^\KdV_0 \lambda \mc{W} &, \omega_j = 1
            \\ R^\KdV_0 \sqrt{2} \epsilon \sigma_1 i \del_X \mc{W} &, \omega_j = 2
        \end{cases} \\
        &+ \sum_{\omega \in \{0,1\}^l} 2^{-l} \left(\frac{\mp 1}{\lambda}\right)^{\mathds{1}_{\{\conj{\omega} \text{ odd}\}}} \prod_{j=1}^l 
        \begin{cases}
            R^\KdV_0 \sigma_3 \mc{W} &, \omega_j = 0
            \\ R^\KdV_0 \lambda \mc{W} &, \omega_j = 1
        \end{cases}
        \\ &= I + II \,.
    \end{align*}
    Using that the trace of off-diagonal matrices vanishes, we estimate
    \begin{align*}
        \left| \sum_{l=2}^\infty \frac{(-1)^l}{l} \tr[I] \right|
        &\leq \sum_{l=2}^\infty \underset{|\{j: \omega_j = 2\}| \geq 2 \text{ even}}{\sum_{\omega \in \{0,1,2\}^l}} |\lambda|^{- \mathds{1}_{\{\conj{\omega} \text{ odd}\}}} 2^{-l} \prod_{j=1}^l \begin{cases}
            \big\|\sqrt{R^\KdV_0} \sigma_3 \mc{W} \sqrt{R^\KdV_0}\big\|_{\HS} &, \omega_j = 0 \\
            \big\|\sqrt{R^\KdV_0} \mc{W} \sqrt{R^\KdV_0}\big\|_{\HS} |\lambda| &, \omega_j = 1 \\
            \big\|\sqrt{R^\KdV_0} \sqrt{2} \epsilon \sigma_1 i \del_X \mc{W} \sqrt{R^\KdV_0}\big\|_{\HS} &, \omega_j = 2
        \end{cases} \\
        &\leq \sum_{l=2}^\infty \underset{|\{j: \omega_j = 2\}| \geq 2 \text{ even}}{\sum_{\omega \in \{0,1,2\}^l}} |\lambda|^{- \mathds{1}_{\{\conj{\omega} \text{ odd}\}}} 2^{-l} \kappa^{-\frac{l}{2}} \prod_{j=1}^l \begin{cases}
            \|W\|_{H^{-1}_\kappa} &, \omega_j = 0 \\
            \|W\|_{H^{-1}_\kappa} |\lambda| &, \omega_j = 1 \\
            \|W\|_{L^2_\kappa} \epsilon &, \omega_j = 2
        \end{cases}
    \end{align*}
    Here we have used \eqref{eqn:HS-estimate-2}, \eqref{eqn:HS-1}, and \eqref{eqn:HS-1.5}.
    Noting that $\|W\|_{H^{-1}_\kappa} \leq \frac{1}{\sqrt{2} \kappa} \|W\|_{L^2_\kappa}$ and $|\lambda| \leq 1$, we finish our estimate:
    \begin{align*}
        \left| \sum_{l=2}^\infty \frac{(-1)^l}{l} \tr[I] \right|
        &\leq \sum_{l=2}^\infty \underset{|\{j: \omega_j = 2\}| \geq 2 \text{ even}}{\sum_{\omega \in \{0,1,2\}^l}} \epsilon^{|\{\omega_j = 2\}|} 
        2^{-l} \kappa^{-\frac{l}{2}} |\lambda|^{|\{\omega_j = 1\}| - \mathds{1}_{\{\conj{\omega} \text{ odd}\}}} \|W\|_{H^{-1}_\kappa}^{|\{\omega_j = 0\}| + |\{\omega_j = 1\}|} \|W\|_{L^2_\kappa}^{|\{\omega_j = 2\}|}
        \\ &\leq \sum_{l=2}^\infty \|W\|_{L^2_\kappa}^l 2^{-l} \underset{|\{j: \omega_j = 2\}| \geq 2 \text{ even}}{\sum_{\omega \in \{0,1,2\}^l}} \epsilon^{|\{\omega_j=2\}|} (\sqrt{2} \kappa)^{- |\{\omega_j=0\}| - |\{\omega_j=1\}| - \frac{l}{2}}
        \\ &\leq 2 \epsilon^2 \kappa^2 \sum_{l=2}^\infty \kappa^{-\frac{l}{2}} \|W\|_{L^2_\kappa}^l 2^{-l} \underset{n \text{ even}}{\sum_{n = 2}^l} \sum_{j=0}^l \binom{l}{n, j} (\sqrt{2} \epsilon \kappa)^{n-2} (\sqrt{2} \kappa)^{- \frac32 l}
        \\ &\leq 2 \epsilon^2 \kappa^2 \sum_{l=2}^\infty \kappa^{- \frac32 l} \|W\|_{L^2_\kappa}^l 2^{-l} 3^l \sqrt{2}^{-\frac32 l}
        \leq \frac{\epsilon^2 \kappa^{-1} \|W\|_{L^2}^2}{1 - \kappa^{-\frac32} \|W\|_{L^2}} \,.
    \end{align*}
    In order to compute $II$, we decompose
    \begin{align*}
        \{0,1\}^l &= \{\omega \in \{0,1\}^l: \omega_1 = 0 \text{ and } \conj{\omega} \text{ even}\} \cup \{\omega \in \{0,1\}^l: \omega_1 = 1 \text{ and } \conj{\omega} \text{ odd}\} \\
        &\cup \{\omega \in \{0,1\}^l: \omega_1 = 0 \text{ and } \conj{\omega} \text{ odd}\} \cup \{\omega \in \{0,1\}^l: \omega_1 = 1 \text{ and } \conj{\omega} \text{ even}\}
    \end{align*}
    and observe that
    \begin{align*}
        \omega &\longmapsto (0, \omega_2, \dots, \omega_l)
        & \omega &\longmapsto (1, \omega_2, \dots, \omega_l)
    \end{align*}
    are bijections between the first two and last two sets. 
    We may sum only over the sets on the left and correspondingly adjust each term as follows:
    \begin{align*}
        II &= \underset{\conj{\omega} \text{ even}}{\underset{\omega_1 = 0}{\sum_{\omega \in \{0,1\}^l}}} 2^{-l} R^\KdV_0  (\sigma_3 \mp 1) \mc{W} \prod_{j=2}^l 
        \begin{cases}
            R^\KdV_0 \sigma_3 \mc{W} &, \omega_j = 0
            \\ R^\KdV_0 \lambda \mc{W} &, \omega_j = 1
        \end{cases}
        \\ &+ \underset{\conj{\omega} \text{ odd}}{\underset{\omega_1 = 0}{\sum_{\omega \in \{0,1\}^l}}} 2^{-l} R^\KdV_0  \left( \frac{\mp \sigma_3}{\lambda} + \lambda \right) \mc{W} \prod_{j=2}^l 
        \begin{cases}
            R^\KdV_0 \sigma_3 \mc{W} &, \omega_j = 0
            \\ R^\KdV_0 \lambda \mc{W} &, \omega_j = 1
        \end{cases}
        \\ &= \underset{\conj{\omega} \text{ even}}{\underset{\omega_1 = 0}{\sum_{\omega \in \{0,1\}^l}}} 2^{-l} (\sigma_3 \mp 1) \lambda^{\conj{\omega}} \sigma_3^{l - 1 - \conj{\omega}} (R^\KdV_0 \mc{W})^l
        + \underset{\conj{\omega} \text{ odd}}{\underset{\omega_1 = 0}{\sum_{\omega \in \{0,1\}^l}}} 2^{-l} (\sigma_3 \mp 1 \pm 2 \epsilon^2 \kappa^2) \left(\frac{\mp 1}{\lambda}\right) \lambda^{\conj{\omega}} \sigma_3^{l - 1 - \conj{\omega}} (R^\KdV_0 \mc{W})^l
        \\ &= (\sigma_3 \mp 1) (R^\KdV_0 \sigma_3 \mc{W})^l \sigma_3 \underset{\omega_1 = 0}{\sum_{\omega \in \{0,1\}^l}} 2^{-l} \left(\frac{\mp 1}{\lambda}\right)^{\mathds{1}_{\{\conj{\omega} \text{ odd}\}}} \lambda^{\conj{\omega}} \sigma_3^{\conj{\omega}}
        + 2 \epsilon^2 \kappa^2 (R^\KdV_0 \sigma_3 \mc{W})^l \sigma_3 \underset{\conj{\omega} \text{ odd}}{\underset{\omega_1 = 0}{\sum_{\omega \in \{0,1\}^l}}} 2^{-l} \left(\frac{-1}{\lambda}\right) \lambda^{\conj{\omega}} \sigma_3^{\conj{\omega}}
        \\ &= III + IV \,.
    \end{align*}
    The further treatment of $III$ and $IV$ is a matter of computation.
    We have
    \begin{align*}
        III &= (\sigma_3 \mp 1) (R^\KdV_0 \sigma_3 \mc{W})^l 2^{-l} \sigma_3 
        \left( \underset{\conj{\omega} \text{ even}}{\sum_{\conj{\omega}=0}^{l-1}} \binom{l - 1}{\conj{\omega}} \lambda^{\conj{\omega}} 
        \mp \sigma_3 \underset{\conj{\omega} \text{ odd}}{\sum_{\conj{\omega}=0}^{l-1}} \binom{l - 1}{\conj{\omega}} \lambda^{\conj{\omega} - 1} \right)
        \\ &= \frac{\sigma_3 \mp 1}{2} (R^\KdV_0 \sigma_3 \mc{W})^l 2^{1-l} \left( \sigma_3 \sum_{j=0}^{\floor{\frac{l-1}{2}}} \binom{l - 1}{2 j} (1 - 2 \epsilon^2 \kappa^2)^j 
        \mp \sum_{j=0}^{\floor{\frac{l}{2}} - 1} \binom{l - 1}{2 j + 1} (1 - 2 \epsilon^2 \kappa^2)^j \right)
        \\ &= \frac{(\sigma_3 \mp 1)^2}{4} (R^\KdV_0 \sigma_3 \mc{W})^l
        + \frac{\sigma_3 \mp 1}{2} (R^\KdV_0 \sigma_3 \mc{W})^l 2^{1-l} \sum_{j=0}^\infty \left( \sigma_3 \binom{l - 1}{2 j} \mp \binom{l - 1}{2 j + 1} \right) ((1 - 2 \epsilon^2 \kappa^2)^j - 1)
        \\ &= V + VI \,.
    \end{align*}
    The most important term is $V$, because
    \begin{align*}
        \sum_{l=2}^\infty \frac{(-1)^l}{l} \tr[V] 
        &= \sum_{l=2}^\infty \frac{(-1)^l}{l} \left( \frac{(1 \mp 1)^2}{4} \tr[(R^\KdV_0 (- W_-))^l] + \frac{(- 1 \mp 1)^2}{4} \tr[(R^\KdV_0 W_+)^l] \right)
        = \alpha^\KdV(\kappa; \pm W_\pm) \,.
    \end{align*}
    We calculate further
    \begin{align*}
        IV &= - 2 \epsilon^2 \kappa^2 (R^\KdV_0 \sigma_3 \mc{W})^l 2^{- l} \underset{\conj{\omega} \text{ odd}}{\sum_{\conj{\omega} = 0}^{l-1}} \binom{l - 1}{\conj{\omega}} \lambda^{\conj{\omega} - 1}
        \\ &= - 2 \epsilon^2 \kappa^2 (R^\KdV_0 \sigma_3 \mc{W})^l 2^{- l} \sum_{j = 0}^{\floor{\frac{l}{2}} - 1} \binom{l - 1}{2 j + 1} (1 - 2 \epsilon^2 \kappa^2)^j
        \\ &= (R^\KdV_0 \sigma_3 \mc{W})^l 2^{- l} \sum_{m=1}^\infty (- 2 \epsilon^2 \kappa^2)^m \sum_{j = 0}^\infty \binom{l - 1}{2 j + 1} \binom{j}{m - 1}
    \end{align*}
    and
    \begin{align*}
        VI &= (R^\KdV_0 \sigma_3 \mc{W})^l 2^{-l} (\sigma_3 \mp 1) \sum_{j=0}^\infty \left( \sigma_3 \binom{l - 1}{2 j} \mp \binom{l - 1}{2 j + 1} \right) \sum_{m=1}^\infty \binom{j}{m} (- 2 \epsilon^2 \kappa^2)^m
        \\ &= (R^\KdV_0 \sigma_3 \mc{W})^l 2^{-l} \sum_{m=1}^\infty (- 2 \epsilon^2 \kappa^2)^m \sum_{j=0}^\infty (1 \mp \sigma_3) \binom{l}{2 j + 1} \binom{j}{m} \,.
    \end{align*}
    Then a binomial identity, which may be verified by a calculation with generating functions, yields
    \begin{align*}
        IV + VI &= (R^\KdV_0 \sigma_3 \mc{W})^l 2^{-l} \sum_{m=1}^\infty (- 2 \epsilon^2 \kappa^2)^m \sum_{j=0}^\infty \left( (1 \mp \sigma_3) \binom{l}{2 j + 1} \binom{j}{m} + \binom{l - 1}{2j + 1} \binom{j}{m-1} \right)
        \\ &= (R^\KdV_0 \sigma_3 \mc{W})^l 2^{-l} \sum_{m=1}^\infty (- 2 \epsilon^2 \kappa^2)^m \left( (1 \mp \sigma_3) 2^{l-1-2m} \binom{l-1-m}{m} + 2^{l-2m} \binom{l-1-m}{m-1} \right)
        \\ &= (R^\KdV_0 \sigma_3 \mc{W})^l \sum_{m=1}^\infty (- 2 \epsilon^2 \kappa^2)^m 4^{-m} \left( \frac{1 \mp \sigma_3}{2} \binom{l-1-m}{m} + \binom{l-1-m}{m-1} \right) \,.
    \end{align*}
    We find that
    \begin{align*}
        & \sum_{l=2}^\infty \frac{(-1)^l}{l} \tr[IV + VI]
        \\ &= \sum_{l=2}^\infty \frac{(-1)^l}{l} \sum_{m=1}^\infty (- 2 \epsilon^2 \kappa^2)^m 4^{-m} \left( \tr\left[(R^\KdV_0 \sigma_3 \mc{W})^l \frac{1 \mp \sigma_3}{2}\right] \binom{l-1-m}{m} + \tr[(R^\KdV_0 \sigma_3 \mc{W})^l] \binom{l-1-m}{m-1} \right)
        \\ &= \sum_{l=2}^\infty \frac{(-1)^l}{l} \sum_{m=1}^\infty (- 2 \epsilon^2 \kappa^2)^m 4^{-m} \left( \tr[(R^\KdV_0 (\pm W_\pm))^l] \binom{l-1-m}{m} \right.
        \\ &\left. \hspace{5cm} + \big(\tr[(R^\KdV_0 (- W_-))^l] + \tr[(R^\KdV_0 W_+)^l]\big) \frac12 \binom{l-1-m}{m-1} \right) \,,
    \end{align*}
    and finally
    \begin{align*}
        \left| \sum_{l=2}^\infty \frac{(-1)^l}{l} \tr[IV + VI] \right|
        &\leq \sum_{m=1}^\infty \epsilon^{2m} \kappa^{2m} 2^{-m} \sum_{l=2m}^\infty 2^{l-m} \kappa^{-\frac{l}{2}} \big(\|W_-\|_{H^{-1}_\kappa}^l + \|W_+\|_{H^{-1}_\kappa}^l\big)
        \\ &\leq \sum_{m=1}^\infty \epsilon^{2m} \kappa^m \|W\|_{H^{-1}_\kappa}^{2m} \frac{2}{1 - 2 \kappa^{-\frac12} \|W\|_{H^{-1}_\kappa}}
        \\ &\leq \frac{2 \epsilon^2 \kappa \|W\|_{H^{-1}_\kappa}^2}{1 - 2 \epsilon^2 \kappa \|W\|_{H^{-1}_\kappa}^2} \frac{2}{1 - 2 \kappa^{-\frac12} \|W\|_{H^{-1}_\kappa}}
        \\ &\leq \frac{\epsilon^2 \kappa^{-1} \|W\|_{L^2}^2}{1 - \epsilon^2 \kappa^{-1} \|W\|_{L^2}^2} \frac{1}{1 - \kappa^{-\frac32} \|W\|_{L^2}}
        \\ &\leq \frac{\epsilon^2 \kappa^{-1} \|W\|_{L^2}^2}{(1 - \kappa^{-\frac32} \|W\|_{L^2})^2} \,.
    \end{align*}
    Finally, we estimate
    \begin{align*}
        \frac{1}{{(1 - \kappa^{-\frac32} \|W\|_{L^2})^2}} + \frac{1}{{1 - \kappa^{-\frac32} \|W\|_{L^2}}} 
        &\leq \frac{1}{(1 - \frac{1}{\sqrt{3}})^2} + \frac{1}{1 - \frac{1}{\sqrt{3}}} \leq 8 \,.
    \end{align*}
\end{proof}

\section{Overview of Hamiltonian structures}
\label{appendix:poissonstructures}

In this section we present concisely and formally the Hamiltonian structures of \GP, \hGP, \rGP, and \KdV.

\subsection{\texorpdfstring{\GP}{GP}}
\label{appendix:GP}
We consider formally a manifold where we denote points by $\bm{q} = (q, \conj{q})$ and tangent vectors by $\bm{p} = (p, \conj{p})$ and $\bm{r} = (r, \conj{r})$.
We define the Hamiltonian operator $J^\GP$, the Riemannian metric $g^\GP$ and the symplectic form $\omega^\GP$ by
\begin{align*}
    J^\GP_{\bm{q}} &= - \frac{i}{2} \begin{pmatrix} 1 & 0 \\ 0 & - 1 \end{pmatrix}
    \\ g^\GP_{\bm{q}}(\bm{p}, \bm{r}) 
    &= \frac12 \int_{\R} (p \conj{r} + \conj{p} r) \dd x
    \\ \omega^\GP_{\bm{q}}(\bm{p}, \bm{r})
    &= - i \int_{\R} (p \conj{r} - \conj{p} r) \dd x
\end{align*}
They satisfy
\begin{align*}
    \omega^\GP_{\bm{q}}(\bm{p}, J^\GP_{\bm{q}} \bm{r}) &= g^\GP_{\bm{q}}(\bm{p}, \bm{r}) \,.
\end{align*}
We define the gradient $\nabla^{g^\GP}_{\bm{q}}$ with respect to the metric $g^\GP$ as well as the functional derivative via
\begin{align*}
    g^\GP_{\bm{q}}(\bm{p}, \nabla^{g^\GP}_{\bm{q}} \mc{H}) 
    = \int_{\R} \bm{p} \left( \frac{\delta \mc{H}(\bm{q})}{\delta q}, \frac{\delta \mc{H}(\bm{q})}{\delta \conj{q}} \right)^T \dd x
    &= \frac{\dd}{\dd s}\Big\vert_{s=0} \mc{H}(\bm{q} + s \bm{p}) \,.
\end{align*}
This implies
\begin{align*}
    \nabla^{g^\GP}_{\bm{q}} \mc{H} &= \begin{pmatrix} 0 & 2 \\ 2 & 0 \end{pmatrix}
    \left( \frac{\delta \mc{H}(\bm{q})}{\delta q}, \frac{\delta \mc{H}(\bm{q})}{\delta \conj{q}} \right)^T \,.
\end{align*}
Then \eqref{eqn:GP-n} can be written as
\begin{align*}
    i \del_{t_n} \bm{q} &= i J^\GP_{\bm{q}} \nabla^{g^\GP}_{\bm{q}} \mc{H}^\GP_n 
    = \begin{pmatrix} 0 & 1 \\ - 1 & 0 \end{pmatrix} \left( \frac{\delta \mc{H}^\GP_n(\bm{q})}{\delta q}, \frac{\delta \mc{H}^\GP_n(\bm{q})}{\delta \conj{q}} \right)^T \,.
\end{align*}
The Poisson bracket is
\begin{align*}
    \{\mc{F},\mc{G}\}^\GP(\bm{q}) &= \int_{\R} \left( \frac{\delta \mc{F}}{\delta q} \frac{\delta \mc{G}}{\delta \conj{q}} - \frac{\delta \mc{F}}{\delta \conj{q}} \frac{\delta \mc{G}}{\delta q} \right) \dd x \,.
\end{align*}

\subsection{\texorpdfstring{\hGP}{hGP}}
\label{appendix:hGP}
We consider formally a manifold where we denote points by $w = (w_+, w_-)$ and tangent vectors by $u = (u_+, u_-)$ and $v = (v_+, v_-)$.
We define the Hamiltonian operator $J^\hGP$, the Riemannian metric $g^\hGP$ and the symplectic form $\omega^\hGP$ by
\begin{align*}
    J^\hGP_{w} &= \frac{1}{8} \begin{pmatrix}
        - \del_x a^{-1} + 4 a \del_x^{-1}
        & \del_x a^{-1} + 4 a \del_x^{-1}
        \\ - \del_x a^{-1} - 4 a \del_x^{-1}
        & \del_x a^{-1} - 4 a \del_x^{-1}
    \end{pmatrix}
    \\ g^\hGP_{w}(u, v) 
    &= \int_{\R} \frac14 (u_+ - u_-) (v_+ - v_-) + a^2 \del_x^{-1} (u_+ + u_-) \del_x^{-1} (v_+ + v_-) \dd x
    \\ \omega^\hGP_{w}(u, v) 
    &= \int_{\R} a (v_+ - v_-) \del_x^{-1}(u_+ + u_-) - a (u_+ - u_-) \del_x^{-1}(v_+ + v_-) \dd x \,.
\end{align*}
They satisfy
\begin{align*}
    \omega^\hGP_{w}(u, J^\hGP_{w} v) &= g^\hGP_{w}(u, v) \,.
\end{align*}
We define the gradient $\nabla^{g^\hGP}_{w}$ with respect to the metric $g^\hGP$ as well as the functional derivative via
\begin{align*}
    g^\hGP_{w}(v, \nabla^{g^\hGP}_{w} \mc{H}) 
    = \int_{\R} (v_+, v_-) \left( \frac{\delta \mc{H}(w)}{\delta w_+}, \frac{\delta \mc{H}(w)}{\delta w_-} \right)^T \dd x
    &= \frac{\dd}{\dd s}\Big\vert_{s=0} \mc{H}(w + s v) \,.
\end{align*} 
This implies
\begin{align*}
    \nabla^{g^\hGP}_{w} \mc{H} &= \frac14 \begin{pmatrix}
        4 - \del_x a^{-2} \del_x
        & - 4 - \del_x a^{-2} \del_x
        \\ - 4 - \del_x a^{-2} \del_x
        & 4 - \del_x a^{-2} \del_x
    \end{pmatrix} \left( \frac{\delta \mc{H}(w)}{\delta w_+}, \frac{\delta \mc{H}(w)}{\delta w_-} \right)^T \,.
\end{align*}
Then \eqref{eqn:hGP-n} can be written as
\begin{align*}
    \del_{t_n} w &= J^\hGP_{w} \nabla^{g^\hGP}_{w} \mc{H}^\hGP_n 
    = \frac{1}{4} \begin{pmatrix}
        - \del_x a^{-1} - a^{-1} \del_x
        & \del_x a^{-1} - a^{-1} \del_x
        \\ - \del_x a^{-1} + a^{-1} \del_x
        & \del_x a^{-1} + a^{-1} \del_x
    \end{pmatrix} \left( \frac{\delta \mc{H}^\hGP_n(w)}{\delta w_+}, \frac{\delta \mc{H}^\hGP_n(w)}{\delta w_-} \right)^T \,.
\end{align*}
The Poisson bracket is
\begin{align*}
    \{\mc{F}, \mc{G}\}^\hGP(w) &= \frac14 \int_{\R} \begin{pmatrix}
        \frac{\delta \mc{F}}{\delta w_+} & \frac{\delta \mc{F}}{\delta w_-}
    \end{pmatrix} \begin{pmatrix}
    - \del_x a^{-1} - a^{-1} \del_x & \del_x a^{-1} - a^{-1} \del_x\\
    - \del_x a^{-1} + a^{-1} \del_x & \del_x a^{-1} + a^{-1} \del_x
    \end{pmatrix} \begin{pmatrix}
        \frac{\delta \mc{G}}{\delta w_+} \\ \frac{\delta \mc{G}}{\delta w_-}
    \end{pmatrix} \dd x \,.
\end{align*}

\subsection{\texorpdfstring{\rGP}{rGP}}
\label{appendix:rGP}
We consider formally a manifold where we denote points by $W = (W_+, W_-)$ and tangent vectors by $U = (U_+, U_-)$ and $V = (V_+, V_-)$.
We define the Hamiltonian operator $J^\rGP$, the Riemannian metric $g^\rGP$ and the symplectic form $\omega^\rGP$ by
\begin{align*}
    J^\rGP_{W} &= \frac{1}{8} \begin{pmatrix}
        - \sqrt{2} \epsilon \del_X A^{-1} + \frac{4}{\sqrt{2} \epsilon} A \del_X^{-1}
        & \sqrt{2} \epsilon \del_X A^{-1} + \frac{4}{\sqrt{2} \epsilon} A \del_X^{-1}
        \\ - \sqrt{2} \epsilon \del_X A^{-1} - \frac{4}{\sqrt{2} \epsilon} A \del_X^{-1}
        & \sqrt{2} \epsilon \del_X A^{-1} - \frac{4}{\sqrt{2} \epsilon} A \del_X^{-1}
    \end{pmatrix}
    \\ g^\rGP_{W}(U, V) 
    &= \int_{\R} \frac{\epsilon^3}{4 \sqrt{2}} (U_+ - U_-) (V_+ - V_-) + \frac{\epsilon}{2 \sqrt{2}} A^2 \del_X^{-1} (U_+ + U_-) \del_X^{-1} (V_+ + V_-) \dd X
    \\ \omega^\rGP_{W}(U, V)
    &= \frac{\epsilon^2}{2} \int A \left( (V_+ - V_-) \del_X^{-1} (U_+ + U_-) - (U_+ - U_-) \del_X^{-1} (V_+ + V_-) \right) \dd X \,.
\end{align*}
They satisfy
\begin{align*}
    \omega^\rGP_{W}(U, J^\rGP_{W} V) &= g^\rGP_{W}(U, V) \,.
\end{align*}
We define the gradient $\nabla^{g^\rGP}_{W}$ with respect to the metric $g^\rGP$ as well as the functional derivative via
\begin{align*}
    g^\rGP_{W}(V, \nabla^{g^\rGP}_{W} \mc{H}) 
    = \int_{\R} (V_+, V_-) \left( \frac{\delta \mc{H}(W)}{\delta W_+}, \frac{\delta \mc{H}(W)}{\delta W_-} \right)^T \dd X
    &= \frac{\dd}{\dd s}\Big\vert_{s=0} \mc{H}(W + s V) \,.
\end{align*} 
This implies
\begin{align*}
    \nabla^{g^\rGP}_{W} \mc{H} &= \frac{1}{\sqrt{2} \epsilon^3} \begin{pmatrix}
        2 - \epsilon^2 \del_X A^{-2} \del_X
        & - 2 - \epsilon^2 \del_X A^{-2} \del_X
        \\ - 2 - \epsilon^2 \del_X A^{-2} \del_X
        & 2 - \epsilon^2 \del_X A^{-2} \del_X
    \end{pmatrix} \left( \frac{\delta \mc{H}(W)}{\delta W_+}, \frac{\delta \mc{H}(W)}{\delta W_-} \right)^T \,.
\end{align*}
Then \eqref{eqn:rGP-n} can be written as
\begin{align*}
    \del_{T_n} W &= \frac{\dd t_n}{\dd T_n} J^\rGP_{W} \nabla^{g^\rGP}_{W} \mc{H}^\rGP_n 
    = \frac{\dd t_n}{\dd T_n} \frac{1}{2 \epsilon^2} \begin{pmatrix}
        - \del_X A^{-1} - A^{-1} \del_X
        & \del_X A^{-1} - A^{-1} \del_X
        \\ - \del_X A^{-1} + A^{-1} \del_X
        & \del_X A^{-1} + A^{-1} \del_X
    \end{pmatrix} \left( \frac{\delta \mc{H}^\rGP_n(W)}{\delta W_+}, \frac{\delta \mc{H}^\rGP_n(W)}{\delta W_-} \right)^T \,.
\end{align*}
The Poisson bracket is
\begin{align*}
    \{\mc{F}, \mc{G}\}^\rGP(W) &= \frac{1}{2\epsilon^2} \int_{\R} \begin{pmatrix}
        \frac{\delta \mc{F}}{\delta W_+} & \frac{\delta \mc{F}}{\delta W_-} 
    \end{pmatrix} \begin{pmatrix}
    - \del_X A^{-1} - A^{-1} \del_X & \del_X A^{-1} - A^{-1} \del_X\\
    - \del_X A^{-1} + A^{-1} \del_X & \del_X A^{-1} + A^{-1} \del_X
    \end{pmatrix} \begin{pmatrix}
        \frac{\delta \mc{G}}{\delta W_+} \\ \frac{ \delta\mc{G}}{\delta W_-} 
    \end{pmatrix} \dd X \,.
\end{align*}

\subsection{\texorpdfstring{\KdV}{KdV}}
\label{appendix:KdV}
We consider formally a manifold where we denote points by $U$ and tangent vectors by $F$ and $G$.
We define the Hamiltonian operator $J^\KdV$, the Riemannian metric $g^\KdV$ and the symplectic form $\omega^\KdV$ by
\begin{align*}
    J^\KdV_U &= \frac12 \del_X
    \\ g^\KdV_U(F, G) &= \int_{\R} F G \dd X
    \\ \omega^\KdV_U &= \frac12 \int_{\R} F \del_X^{-1} G - G \del_X^{-1} F \dd X \,.
\end{align*}
They satisfy
\begin{align*}
    \omega^\KdV_U(F, J^\KdV_U G) &= g^\KdV_U(F, G) \,.
\end{align*}
We define the gradient $\nabla^{g^\KdV}_U$ with respect to the metric $g^\KdV$ as well as the functional derivative via
\begin{align*}
    g^\KdV_U(F, \nabla^{g^\KdV}_U \mc{H}) 
    = \int_{\R} F \frac{\delta \mc{H}(U)}{\delta U} \dd X
    &= \frac{\dd}{\dd s}\Big\vert_{s=0} \mc{H}(U + s F) \,.
\end{align*}
This implies
\begin{align*}
    \nabla^{g^\KdV}_U \mc{H} &= \frac{\delta \mc{H}(U)}{\delta U} \,.
\end{align*}
Then \eqref{eqn:KdV-n} can be written as
\begin{align*}
    \del_{T_n} U &= J^\KdV_U \nabla^{g^\KdV}_U \mc{E}^\KdV_n 
    = \frac12 \del_X \frac{\delta \mc{E}^\KdV_n(U)}{\delta U} \,.
\end{align*}
The Poisson bracket is
\begin{align*}
    \{\mc{F}, \mc{G}\}^\KdV(U) &= \frac12 \int_{\R} \frac{\delta \mc{F}}{\delta U} \del_X \frac{\delta \mc{G}}{\delta U} \dd X \,.
\end{align*}

\textbf{Statements and Declarations.} \\
This work was funded by the Deutsche Forschungsgemeinschaft (DFG, German Research Foundation) – Project-ID 258734477 – SFB 1173.

No datasets were generated or analysed during the current study.

The author has no competing interests to declare that are relevant to the content of this article.

SymPy \cite{sympy} was used to carry out symbolic computations for the purposes of exploration and verification of various identities involving recurrence relations.
Various versions of the OpenAI services ChatGPT and Codex were used for mathematical and grammatical proofreading, literature search, and exploratory discussion. 
\printbibliography

\begin{minipage}[t]{0.48\textwidth}
\raggedright\small
\textbf{Robert Wegner}\\
Institute for Analysis\\
Karlsruhe Institute of Technology\\
Englerstraße 2\\
76131 Karlsruhe, Germany\\
\href{mailto:robert.wegner@kit.edu}{robert.wegner@kit.edu}
\end{minipage}

\end{document}